\documentclass[reqno]{amsart}
\usepackage{amsmath,amssymb,amsthm,amsfonts}
\usepackage{hyperref}
\usepackage{mathrsfs}
\usepackage{appendix}
\usepackage{graphicx}
\usepackage{geometry}
\allowdisplaybreaks[4]%³¤¹«Ê½»»Ò³

\usepackage{setspace}%ʹÓüä¾àºê°ü
\usepackage{color}

\newtheorem{lemma}{Lemma}[section]
\newtheorem{theorem}{Theorem}[section]

\newtheorem{proposition}{Proposition}[section]
\newtheorem{remark}{Remark}[section]

\numberwithin{equation}{section}
\newcommand{\beq}{\begin{equation}}
\newcommand{\eeq}{\end{equation}}
\newcommand{\ben}{\begin{eqnarray}}
\newcommand{\een}{\end{eqnarray}}
\newcommand{\beno}{\begin{eqnarray*}}
\newcommand{\eeno}{\end{eqnarray*}}

\begin{document}
\title[ Stability of 3D Couette flow in Boussinesq system]{Nonlinear stability of 3-D plane Couette flow for Boussinesq system in a finite channel}

\author[Jinkai Li]{Jinkai Li}
\address[J. Li]{South China Research Center for Applied Mathematics and Interdisciplinary Studies, South China Normal University, Guangzhou, 510631, China}
\email{jklimath@gmail.com;jklimath@m.scnu.edu.cn}
\author[Zhilin Lin]{Zhilin Lin}
\address[Z. Lin]{School of Mathematical Sciences, South China Normal University,
Guangzhou, 510631, China}
\email{zllin@m.scnu.edu.cn}
\author[Dong Wang]{Dong Wang}
\address[D. Wang]{South China Research Center for Applied Mathematics and Interdisciplinary Studies, South China Normal University, Guangzhou, 510631, China}
\email{WDongter@m.scnu.edu.cn}

%\keywords{stability; 3D Boussinesq equation; Couette flow; Nonslip boundary condition.}
%\subjclass[2010]{26D10, 35Q35, 35Q86, 76D03, 76D05, 86A05, 86A10.}

%%% ----------------------------------------------------------------------

\begin{abstract}
In this paper, we study the hydrodynamics stability of the 3-D plane Couette flow $v_s=(y,0,0)$ with constant temperature state $\rho_s$ for the Boussinesq system in a finite channel $\mathbb{T}\times [-1,1]\times \mathbb{T}$ with non-slip boundary condition in $y$ direction. We prove that if the initial velocity field $v_{in}$ and initial temperature $\rho_{in}$ satisfy
\begin{align*}
\|v_{in}-(y,0,0)\|_{H^2} \leq c_0 \nu, \ \ \|\rho_{in}-\rho_s\|_{H^2} \leq c_0 \nu^2,
\end{align*}
for some $c_0>0$ independent of the viscosity and thermal diffusion $\nu$, then the solutions to 3D Boussinesq system do not transit from the Couette flow and constant temperature state. To our best knowledge, this is the first result about the nonlinear stability of plane Couette flow for 3-D Boussinesq system in the domain with non-slip boundary condition. This result, together with \cite{cuilw}, implies that the strong boundary layer effect does not lead to strong instability.
\end{abstract}

%%% ----------------------------------------------------------------------
\maketitle
\tableofcontents
\allowdisplaybreaks

%\tableofcontents

\section{Introduction}
\label{sec1}

In this paper, we study the hydrodynamic stability and transition threshold problem for the 3-D Boussinesq system in finite channel at high Reynolds number ${\rm Re}\gg 1$. The 3-D Boussinesq system in a finite channel $(x,y,z)\in \Omega:=\mathbb{T}\times [-1,1]\times \mathbb{T}$ read as
\begin{align}\label{3dbeq}
\begin{cases}
\partial_t v-\nu \Delta v+(v\cdot \nabla v)+\nabla p=\rho g {\bf e}_2,\\
\partial_t \rho-\mu\Delta \rho +(v\cdot \nabla )\rho=0,\\
\nabla \cdot v=0, \quad (v,\rho)|_{t=0}=(v_{in}, \rho_{in})(x,y,z),
\end{cases}
\end{align}
where $v=(v^1,v^2,v^3) \in \mathbb{R}^3$, $p\in \mathbb{R}$, and $\rho \in \mathbb{R}_+$ denote the velocity of the fluid, pressure, and the temperature, respectively. The constant $g$ is the gravitational constant, ${\bf e}_2=(0,1,0)$ represents the unit vector, $\nu={\rm Re}^{-1}$ is the viscosity coefficient, while $\mu$ denotes the thermal diffusion coefficient. For simplicity, we focus on the case of $\nu=\mu$ and $g=1$ in this paper.

The plane Couette flow with constant temperature state $(v_s,p_s,\rho_s)$ with
\begin{align}\label{pcf}
v_s=(y,0,0), \ \ p_s=y+c, \ \ \rho_s=1/g
\end{align}
 is the steady solution to \eqref{3dbeq}. Our interest is to understand the hydrodynamics stability of \eqref{pcf} at large Reynolds number ${\rm Re}\gg 1$ (or equivalently, $0<\nu \ll 1$).

To formulate the problem, let us introduce the perturbation around the steady state:
$$u=v-v_s,\ P=p-p_s,\ \theta=\rho-\rho_s,$$
then $(u,\theta,P)$ enjoys
\begin{align}
\label{eqBou}
\begin{cases}
\partial_t u-\nu \Delta u+y\partial_xu+\left(
\begin{array}{c}  u^2 \\ 0 \\ 0\end{array}\right)+(u\cdot \nabla) u+\nabla P= \left(
\begin{array}{c} 0 \\ \theta \\ 0\end{array}\right),\\
\partial_t \theta-\nu\Delta \theta+y\partial_x\theta +u\cdot \nabla \theta=0,\\
\nabla \cdot u=0, \\
 (u,\theta)|_{t=0}=(u_{in}, \theta_{in}).
\end{cases}
\end{align}
On the boundaries $\{y=\pm 1\}$, the following boundary conditions are imposed on $u,\theta$:
\begin{align}
\label{bdd}
u(t,x, \pm 1,z)=\theta(t,x,\pm 1,z)=0.
\end{align}
The non-slip boundary condition is imposed on the velocity field and the isothermal boundary condition is imposed on the temperature.

In \eqref{eqBou}, the pressure is decomposed into three parts
\begin{align}\label{pressure-decom}
P=P^{L}+P^{NL}+P^{\theta},
\end{align}
 in which $P^{L}, P^{NL}, P^{\theta}$ are determined by
\begin{align}
\label{eqP}
\begin{cases}
\Delta P^L=-2\partial_x u^2,\ \Delta P^\theta=\partial_y \theta\\
\Delta P^{NL}=-\nabla\cdot (u\cdot\nabla u)=-\partial_iu_j\partial_ju_i,\\
(\partial_yP^L-\nu\Delta u^2)|_{y=\pm1}=0, (\partial_yP^{NL},\partial_y P^\theta)|_{y=\pm 1}=(0,0).
\end{cases}
\end{align}
Compared with the case of Navier-Stokes (NS) equations \cite{chenwz}, there is an additional part $P^\theta$ in \eqref{pressure-decom}, which is resulted from the appearance of temperature.

The Boussinesq system is one of the most important models in the atmospheric and oceanic dynamics, which is used to describe heat transfer phenomena, see, e.g., \cite{conP-DoeC1, conP-DoeC2, GetlAV1, GillA1, Majda1, Pedl1}. In particular, if $\rho=0$, the Boussinesq system \eqref{3dbeq} (or the perturbation \eqref{eqBou})  is reduced to the classic Navier-Stokes equations. The stability of laminar flows in the NS equations has been studied for several years. Indeed, to understand the mechanism for the stability and transition of laminar flow, Trefethen et al. \cite{Tre} suggested that one may focus on the transition threshold problem of laminar flow, of which the mathematical version was introduced by Bedrossian, Germain and Masmoudi \cite{BedGM2}:

{\it Given a norm $\|.\|_X$, find a number $\beta=\beta(X)$ so that}
\begin{align*}
&\|v_{in}-v_s\|\lesssim \nu^{\beta}\Longrightarrow {stability};\\
&\|v_{in}-v_s\|\gg \nu^{\beta}\Longrightarrow {instability}.
\end{align*}
The exponent $\beta$ is referred to as the transition threshold in the applied literature.

The transition threshold problem concerning different kind of shear flows in the Navier-Stokes equations has been studied in recent years.  In particular, there are lots of works about the plane Couette flow. If the effect of physical boundary is ignored, for the 3-D plane Couette flow, $X$ is of either Gevrey class or Sobolev space, the transition threshold satisfies $\beta\leq 1$ for the case without boundaries, see \cite{BedGM1,BedGM3,WeiZ2}. Results for the 2-D plane Couette flow can be found in \cite{BedMV1,BedMWang,MasZ1,WeiZ1}.
In the domain with physical boundaries, the boundary layer effect (weak or strong sense) should be considered. For the 2-D Couette flow in a channel with non-slip boundary condition,  Chen, Li, Wei and Zhang \cite{ChenLwz1} proved $\beta\leq \frac{1}{2}$. For the 2-D Couette flow in a channel with Naver-slip boundary condition, Wei and Zhang \cite{Weiz3} showed that $\beta\leq \frac{1}{3}$. For the 3-D Couette flow in a channel with non-slip boundary condition, Chen, Wei and Zhang \cite{chenwz} showed that $\beta\leq 1$ in the Sobolev spaces. For the transition threshold problem for the Poiseuille flow, Kolmogorov flow and other flows, we refer to \cite{BeekH1, chendlz,coti1,Delz1,dingl1,Lih1,LiWz1,weizz} and the references therein.

 Motivated by \cite{Tre,BedGM2}, we focus on the transition threshold problem of the laminar flow $v_s$ with temperature state $\rho_s$ to the Boussinesq system, which
can be formulated as follows:

 {\it Given a norm $\|.\|_{X_i}(i=1,2)$, find $\beta_i=\beta_i(X_1,X_2)$ so that}
\begin{align*}
& \| v_{in}-v_s\|_{X_1}\lesssim \min\{\nu,\mu\}^{\beta_1}, \ \| \rho_{in}-\rho_s\|_{X_2}\lesssim \min\{\nu,\mu\}^{\beta_2} \Longrightarrow {stability};\\
&{ Otherwise}\Longrightarrow {probably \ instability}.
\end{align*}

The above transition threshold problem to the Boussinesq equations has been studied for several flows. For the 2D Couette flow in Boussinesq system, Masmoudi, Said Houari and Zhao \cite{masSz} proved that the flow is stable in $\mathbb{T}\times\mathbb{R}$ without thermal diffusion in Gevrey-$\frac{1}{s}$ spaces with $\frac{1}{3}<s\leq 1$. In the case with thermal diffusion, for the Couette flow with constant temperature state, Zhang and Zi \cite{zhangz} showed that the transition thresholds are $\beta_1\leq \frac{1}{3},\beta_2 \leq \frac{5}{6}$ respectively. Recently, Niu and Zhao \cite{Niuz} improved the transition threshold of temperature from $\beta_2 \leq \frac{5}{6}$ to $\beta_2 \leq \frac{2}{3}$. More results can be found in \cite{BedBC,DengW,zhaiz,Renw} and the references therein.

The results for the 3-D case are few. For the 3-D stratified Couette flow in $\mathbb{T}\times \mathbb{R}\times \mathbb{T}$, Coti Zelati-Del Zotto \cite{coti2} found that the buoyancy forces around the linear stratification will lead to the dispersive effect, which can suppress the 3-D lift-up effect. With this effect, Coti Zelati, Del Zotto and Widmayer \cite{coti3} show that the nonlinear transition threshold for 3-D stratified Couette flow is less than 1. Indeed, they considered the background solutions are $v_s=(y,0,0), \rho_s=1+ay$ for the 3-D Boussinesq system, and then they proved that $\beta_1, \beta_2\leq \frac{11}{12}$ as $b=\sqrt{ag}>\frac{1}{2}$ and $\beta_1, \beta_2\leq \frac{8}{9}$ as $b=\sqrt{ag}>c_2\nu^{-\frac{1}{2}}$. For the 3-D Couette flow with constant temperature in $\mathbb{T}\times \mathbb{R}\times \mathbb{T}$, Cui, Wang and Wang \cite{cuilw} showed that $\beta_1 \leq 1$ and $\beta_2 \leq 2$. Note that there is no stratification effect to the temperature, thus there is no dispersive effect to suppress the 3-D lift-up effect. The transition thresholds $\beta_1 \leq 1$ and $\beta_2 \leq 2$ are larger than those for the 3-D stratified Couette flow.

Note that no boundary effects are considered in the works stated in the above two paragraphs.
It is natural to consider the effect of physical boundary in the problem. Our goal of this paper is to understand the stability for the Couette flow with constant temperature state \eqref{pcf} of the Boussinesq system in $\mathbb{T}\times [-1,1]\times\mathbb{T}$ subject to non-slip boundary condition on velocity field, which may lead to strong boundary layer. Compared with the NS flow, due to the appearance of temperature, additional to 3-D lift-up effect, there is an additional growth in time for the velocity field, which yields that the size of the temperature is smaller than that of the velocity, see Remark \ref{rk1} for a formal analysis.
In addition, compared with \cite{cuilw}, there is a strong boundary layer effect due to the non-slip boundary condition on the velocity field. This effect will lead to singularity near boundaries. This work establishes the same transition thresholds $\beta_1 \leq 1,\beta_2 \leq 2$ as in \cite{cuilw} where they considered the problem in $\mathbb{T}\times \mathbb{R}\times \mathbb{T}$. This implies that the strong boundary layer effect does not lead to strong instability.

To state the main result, let us define
\begin{align*}
\mathbb{P}_0 f(t,y,z)= \overline{f}(t,y,z)=\frac{1}{2\pi}\int_{\mathbb{T}}f(t,x,y,z) {\rm d}x,\quad \mathbb{P}_{\neq} f=f_{\neq} =f-\mathbb{P}_0 f .
\end{align*}

The main result is stated as follows.
\begin{theorem}
  \label{the1}
Assume that $(u_{in},\theta_{in})\in H^2\cap H^1_0$ with ${\rm div}  u_{in}=0$. There exist $\nu_0, c_0, \epsilon, C>0,$ independent of $\nu$ so that if $\|u_{in}\|_{H^2}\leq c_0 \nu$ and $\|\theta_{in}\|_{H^2}\leq c_0 \nu^2$, $0<\nu\leq \nu_0$, then the solution $(u,\theta)$ of \eqref{eqBou} with \eqref{bdd} is global in time and enjoys the following estimates:

{\rm (i)} uniform bounds and decay of background streak solutions
\begin{align}
& \|\bar{u}^1(t)\|_{H^2}+\|\bar{u}^1(t)\|_{L^\infty}\leq C\nu^{-1}\min\{\nu t+\nu^{\frac{2}{3}}, {\rm e}^{-\nu t}\} \left(\|u_{in}\|_{H^2}+\nu^{-1}\|\theta_{in}\|_{H^2}\right) ,\label{dynamics-1}\\
&\|\bar{u}^2(t)\|_{H^2}+\|\bar{u}^3(t)\|_{H^1}+\|(\bar{u}^2,\bar{u}^3)(t)\|_{L^\infty}
\leq C {\rm e}^{-\nu t} \left(\|u_{in}\|_{H^2}+\nu^{-1}\|\theta_{in}\|_{H^2}\right),\label{dynamics-2}\\
& \|\bar{\theta}(t)\|_{H^1}+\|\partial_z\nabla \bar{\theta}(t)\|_{L^2}\leq C {\rm e}^{-\nu t} \nu \left(\|u_{in}\|_{H^2}+\nu^{-1}\|\theta_{in}\|_{H^2}\right) ,\label{dynamics-3}
\end{align}
for any $t\in (0, \infty)$, and

{\rm (ii)} rapid convergence to the streak
\begin{align}
&\|(\partial_x,\partial_z)\partial_xu_{\neq} (t)\|_{L^2}+\|(\partial_x,\partial_z)\nabla u_{\neq}^2(t)\|_{L^2}+\|(\partial_x^2+\partial_z^2)u_{\neq}^3(t)\|_{L^2} +\nu^{\frac{1}{4}}\|u_{\neq}^2(t)\|_{H^2} \nonumber\\
&\qquad\qquad\qquad
+\nu^{\frac{1}{3}}\|(u_{\neq}^1,u_{\neq}^3)(t)\|_{H^1}+\|u_{\neq}^2(t)\|_{L^\infty}
+\nu^\frac{1}{6}\|(u_{\neq}^1,u_{\neq}^3)(t)\|_{L^\infty}\nonumber\\
&\qquad\qquad\qquad\qquad\qquad\qquad\qquad \qquad\leq
C{\rm e}^{-2\epsilon \nu^{\frac{1}{3}}t}\left(\|u_{in}\|_{H^2}+\nu^{-1}\|\theta_{in}\|_{H^2}\right),\label{dynamics-4}
\\
&\|(\partial_x,\partial_z) \nabla\theta_{\neq}(t)\|_{L^2}+\|\theta_{\neq}(t)\|_{L^\infty}\leq
C\nu{\rm e}^{-2\epsilon \nu^{\frac{1}{3}}t}\left(\|u_{in}\|_{H^2}+\nu^{-1}\|\theta_{in}\|_{H^2}\right),\label{dynamics-5}
\end{align}
for any $t\in (0, \infty)$.
\end{theorem}
\begin{remark}
It is noted that we have assumed that $\nu =\mu$ in this paper, and then $\min\{\nu,\mu\}=\nu$. Indeed, the main conclusions still hold for $\nu \neq \mu$.
\end{remark}

\begin{remark}\label{rk1}
The transition threshold estimates $\|u_{in}\|_{H^2}\leq c_0 \nu$ and $\|\theta_{in}\|_{H^2}\leq c_0 \nu^2$ seem to be optimal. Let us give the formal analysis in the linear level. Indeed, the zero modes for linear equations enjoy
 \begin{align*}
\begin{cases}
(\partial_t-\nu \Delta )\bar{u}^1+\bar{u}^2=0,\\
(\partial_t-\nu \Delta )\bar{u}^2-\overline{\theta}+\partial_y \overline{P^\theta}=0,\\
(\partial_t-\nu \Delta )\bar{u}^3+\partial_z \overline{P^\theta}=0,\\
(\partial_t-\nu \Delta )\overline{\theta}=0,\ \Delta P^\theta=\partial_y \theta,
\end{cases}
\end{align*}
which gives that
\begin{align}\label{0-mode-solu}
\begin{cases}
\bar{u}^1(t) = {\rm e}^{\nu t \Delta}\bar{u}^1 (0)-t {\rm e}^{\nu t \Delta} \bar{u}^2 (0) - \frac{t^2}{2} {\rm e}^{\nu t \Delta}\left( \overline{\theta} (0)-\partial_y^2 \Delta^{-1} \overline{\theta} (0) \right)  ,\\
\bar{u}^2 (t) = {\rm e}^{\nu t \Delta} \bar{u}^2 (0)+t {\rm e}^{\nu t \Delta} \left( \overline{\theta} (0)-\partial_y^2 \Delta^{-1} \overline{\theta} (0) \right),\\
\bar{u}^3 (t) =  {\rm e}^{\nu t \Delta} \bar{u}^3 (0)-t {\rm e}^{\nu t \Delta} \partial_y\partial_z  \Delta^{-1} \overline{\theta} (0) ,\\
\overline{\theta}(t) =  {\rm e}^{\nu t \Delta} \overline{\theta} (0) .
\end{cases}
\end{align}
 Following the asymptotic analysis by Chapman \cite{Chapman} (see also \cite{chenwz}), due to the 3-D lift-up effect, the growth term involving $t{\rm }e^{\nu t\Delta}$ for $\bar{u}^1$ will cause a loss of $\nu^{-1}$ in space-time evolution and $\|u_{in}\|_{H^2} \leq c_0 \nu$ in \cite{chenwz}. In addition, due to the appearance of temperature, there are growth $t{\rm e}^{\nu t\Delta}$ for $(\bar{u}^2,\bar{u}^3)$ and $t^2{\rm e}^{\nu t\Delta}$ for $\bar{u}^1$, which lead to the loss of $\nu^{-1}$ and $\nu^{-2}$ for $(\bar{u}^2,\bar{u}^3)$ and $\bar{u}^1$ respectively. Thus, the scaling law for the initial data holds:
\begin{align*}
\bar{u}^1_{in} \sim \nu^{-1} ( \bar{u}^2_{in}, \bar{u}^3_{in}) \sim \nu^{-2} \bar{\theta}_{in}.
\end{align*}
The above analysis suggests formally that the transition thresholds $\|u_{in}\|_{H^2}\leq c_0 \nu$ and $\|\theta_{in}\|_{H^2}\leq c_0 \nu^2$ seems to be optimal in this setting.
\end{remark}

\begin{remark}\label{rk2}
From \eqref{dynamics-1}-\eqref{dynamics-5}, we can obtain that
\begin{align*}
\|u^1(t) \|_{L^\infty}+ \nu^{-1} \|(u^2,u^3)(t)\|_{L^\infty}+\nu^{-2} \| \theta (t) \|_{L^\infty} \leq C c_0 {\rm e}^{-\nu t} \to 0 \ \ {\text as} \ \ t \to +\infty,
\end{align*}
which implies that the plane Couette flow with constant temperature state \eqref{pcf} is nonlinear stable in $L^\infty$ sense.
\end{remark}

\begin{remark}\label{rk3}
Note that in this paper, we deal with the problem in $\mathbb{T}\times [-1,1]\times \mathbb{T}$. It is an interesting problem to study the counterpart in the domain $\mathbb{R}\times [-1,1]\times \mathbb{T}$, for which, the long wave effect will play key role.
\end{remark}

\section{Sketch and key ideas of the proof}\label{sec2}
\subsection{Stability/instability mechanisms and key ideas}
As in the Navier-Stokes flow, the main effects affecting the transition threshold are the 3-D lift-up, enhanced dissipation, inviscid damping and boundary layer, see \cite{BedGM1,BedGM2,BedGM3,BGM3,ChenLwz1,chenwz} and the references therein. Among these effects, the enhanced dissipation and inviscid damping are stability mechanisms, while the 3-D lift-up and boundary layer are instability mechanisms.

 In addition, due to the buoyancy effect (caused by the appearance of temperature), there are additional growth for velocity fields, see Remark \ref{rk1} for detail.

Let us recall that in the NS flow \cite{BedGM1,BedGM2,BedGM3,BGM3,chenwz}, due to the 3-D lift-up effect, there is a factor $t {\rm e}^{\nu t\Delta}$ for $\bar{u}^1$, which leads to the growth of $\bar{u}^1$, but not for $(\bar{u}^2,\bar{u}^3)$ in the NS flow. However, in our problem, \eqref{0-mode-solu} implies that there are growth for all components of the velocity field, which is very different from that for the NS equations.

In addition, the growth factors are $t^2 {\rm e}^{\nu t\Delta}$ and $t {\rm e}^{\nu t\Delta}$ for $\bar{u}^1$ and $(\bar{u}^2,\bar{u}^3)$, respectively, which is faster than those in the NS equations. The growths are due to the 3-D lift-up effect and the coupling effect between the velocity and temperature. The growths are faster than those of NS case and lead to the losses of $\nu^{-1},\nu^{-2}$ in space-time evolution, which along with the result of \cite{chenwz} ($u_{in}\sim o(\nu)$ for velocity field), suggest that the initial perturbations may enjoy $\nu^{-1} \theta_{in} \sim u_{in} \sim o(\nu),$ that is, $u_{in} \sim o(\nu),\ \theta_{in} \sim o(\nu^2).$
This also suggests that the transition thresholds are to be optimal in some sense.

Recall that the perturbations around \eqref{pcf} enjoy
\begin{align*}
\begin{cases}
\partial_t u-\nu \Delta u+y\partial_xu+\left(
\begin{array}{c}  u^2 \\ 0 \\ 0\end{array}\right)+(u\cdot \nabla u)+\nabla P^{L}+\nabla P^{NL}+\nabla P^{\theta}= \left(
\begin{array}{c} 0 \\ \theta \\ 0\end{array}\right),\\
\partial_t \theta-\nu\Delta \theta+y\partial_x\theta +(u\cdot \nabla )\theta=0.
\end{cases}
\end{align*}
The $x$-zero modes $(\bar{u},\bar{\theta})$ satisfy
\begin{align}\label{0-mode}
\begin{cases}
\partial_t \bar{u}^1-\nu \Delta \bar{u}^1+\bar{u}^2+\overline{u\cdot \nabla u^1}=0,\\
\partial_t \bar{u}^2-\nu \Delta \bar{u}^2+\partial_y \overline{P}+(\bar{u}^2 \partial_y +\bar{u}^3 \partial_z) \bar{u}^2+\overline{u_{\neq}\cdot \nabla u^2_{\neq}} =\bar{\theta},\\
\partial_t \bar{u}^3-\nu \Delta \bar{u}^3+\partial_z \overline{P}+(\bar{u}^2 \partial_y +\bar{u}^3 \partial_z) \bar{u}^3+\overline{u_{\neq}\cdot \nabla u^3_{\neq}} =0,\\
\partial_t \bar{\theta}-\nu \Delta \bar{\theta}+\overline{u\cdot \nabla \theta}=0.
\end{cases}
\end{align}
To handle the nonzero modes, we consider the following system
\begin{align}\label{non0-mode}
\begin{cases}
\partial_t\Delta u^2-\nu\Delta^2u^2+y\partial_x\Delta u^2+(\partial_x^2+\partial_z^2)(u\cdot \nabla u^2)-\partial_y[\partial_x(u\cdot\nabla u^1)+\partial_z(u\cdot\nabla u^3)]\\
\quad\quad=(\partial_x^2+\partial_z^2)\theta,\\
\partial_t\omega^2-\nu\Delta \omega^2+y\partial_x\omega^2+\partial_zu^2+\partial_z(u\cdot\nabla u^1)-\partial_x(u\cdot\nabla u^3)=0,\\
\partial_t\theta_{\neq}-\nu \Delta \theta_{\neq}+y\partial_x\theta_{\neq}+(u\cdot \nabla \theta)_{\neq}=0,\\
(\partial_yu^2, u^2, \omega^2,\theta_{\neq})|_{y=\pm1}=(0,0,0,0),\\
 u^2|_{t=0}=u^2(0), \omega^2|_{t=0}=\omega^2 (0), \theta_{\neq}|_{t=0}=\theta_{\neq}(0),
\end{cases}
\end{align}
here $\omega^2 =\partial_z u^1-\partial_x u^3$. In the above system \eqref{non0-mode}, the boundary condition for $u^2$ is the non-slip boundary condition (if $u^2$ is viewed as the stream function), and the boundary condition for $\omega^2,\theta_{\neq}$ can be viewed as Navier-slip boundary condition, which do not result in the strong boundary layer effect.

The 3-D lift-up effect will lead to the secondly instability. To achieve the desired transition thresholds, inspired by \cite{chenwz}, we use the toy model which is around a class of more general laminar flow near the Couette flow $V(y,z)$ with \eqref{eqVcon}, see \cite{chenwz} for details. The linear analysis and resolvent estimates for the linearized system, space-time estimates for toy models are established in \cite{chenwz}, and we do not perform them here.

The main new difference is the treatment of the temperature. Due to the coupling effect between the velocity and temperature, the 3-D lift-up will affect the transition threshold of temperature. Although the boundary condition for $\theta_{\neq}$ does not lead to strong boundary layer effect, the secondly instability will affect the temperature via the coupling effect (by the term $\bar{u}^1 \partial_x \theta_{\neq}$). Precisely, if one consider the toy model of the form
\begin{align*}
(\partial_t-\nu \Delta +y \partial_x)\theta_{\neq}=-\bar{u}^1 \partial_x \theta_{\neq}+\cdots,
\end{align*}
then the term $\bar{u}^1 \partial_x \theta_{\neq}$ is viewed as a perturbation term. However, due to the 3-D lift-up, this term is not small and the transition threshold for temperature is larger than $\beta_2 >2$ in this way.

To achieve $\beta_2 \leq 2$, one should treat with the secondly instability in temperature. Or precisely, the term $\bar{u}^1 \partial_x \theta_{\neq}$ can not be viewed as a perturbation. Inspired by \cite{chenwz}, the toy model for temperature used in this paper is
\begin{align*}
(\partial_t -\nu \Delta +V(t,y,z)\partial_x) \theta_{\neq}={\rm good \ source}, \ V(t,y,z)=y+{u}^{1,0}(t,y,z).
\end{align*}
where $u^{1,0}$ is determined by the decomposition $\bar{u}^1=u^{1,0}+u^{1,1}$ which solves \eqref{equ10}.
The above toy model will be applied to handle $\theta_{\neq}$ with freezing coefficient in time method.
It should be noted that although the boundary condition for $\omega^2$ is the same as that for $\theta_{\neq}$, and we do not use the above arguments for $\omega^2$. This is due to the coupling effect between velocity and temperature.

To apply the above toy model to handle temperature, the resolvent estimates and space-time estimates about it are needed.
Indeed, the main task is to show the flow $V$ is stable with Navier-slip boundary condition, see Proposition \ref{prores1} and Proposition \ref{spaceV} for details, see also \cite{chenwz}. With the resolvent estimates and space-time estimates at hands, one can handle the temperature via freezing coefficient in time method, see subsection \ref{E-4:E-5} for details.

\subsection{Energy functionals}
The main ideas of construction are similar to that in \cite{chenwz}. The main difference is the treatment of temperature. Before introducing the norms, let us introduce
\begin{align*}
&\|f\|_{X_a}=\|{\rm e}^{a \nu^{\frac{1}{3}}t}f\|_{L^\infty L^2}
+\nu^{\frac{1}{2}}\|{\rm e}^{a \nu^{\frac{1}{3}}t}\nabla f\|_{L^2 L^2}
+\nu^{\frac{1}{6}}\|{\rm e}^{a \nu^{\frac{1}{3}}t}f\|_{L^2 L^2},\\
&\|f\|_{Y_a}=\|{\rm e}^{a \nu^{\frac{1}{3}}t}f\|_{L^\infty L^2}+\nu^{\frac{1}{2}}\|{\rm e}^{a \nu^{\frac{1}{3}}t}\nabla f\|_{L^2 L^2}.
\end{align*}

 To handle $\bar{u}^1$, inspired by \cite{chenwz}, we decompose $\bar{u}^1$ as $\bar{u}^1=u^{1,0}+u^{1,1}$, which solve
\begin{align}
\label{equ10}
\begin{cases}
(\partial_t-\nu\Delta) u^{1,0}+\bar{u}^2+\bar{u}^2\partial_yu^{1,0}+\bar{u}^3\partial_zu^{1,0}=0,\\
(\partial_t-\nu\Delta) u^{1,1}+\bar{u}^2\partial_yu^{1,1}+\bar{u}^3\partial_zu^{1,1}+\overline{u_{\neq}\cdot\nabla u_{\neq}^1}=0,\\
u^{1,0}|_{t=0}=0, u^{1,1}|_{t=0}=\bar{u}^1 (0), (u^{1,0}, u^{1,1})|_{y=\pm 1}=(0,0).
\end{cases}
\end{align}
Thus, $u^{1,1}\partial_x$ could be viewed as a good perturbation.
Inspired by \cite{chenwz}, to handle the 3-D lift-up effect, let us introduce
\begin{align}\label{E-1}
E_1=E_{1,0}+\nu^{-\frac{2}{3}}E_{1,1},
\end{align}
where
\begin{align*}
E_{1,0}=&\|u^{1,0}\|_{L^\infty H^4}+\nu^{-1}\|\partial_tu^{1,0}\|_{L^\infty H^2}+\nu^{-\frac{1}{2} }\|\partial_tu^{1,0}\|_{L^2 H^3},\\
E_{1,1}=&\|u^{1,1}\|_{L^\infty H^2}+\nu^{\frac{1}{2}}\|\nabla u^{1,1}\|_{L^2 H^2}.
\end{align*}

The following functional $E_2$ is introduced to control $(\bar{u}^2,\bar{u}^3)$:
\begin{align*}
E_2=&\|\Delta \bar{u}^2\|_{Y_0}+\nu^{\frac{1}{2}}\|\Delta \bar{u}^2\|_{L^2 L^2}+\nu^{-\frac{1}{2}}\|\partial_t\nabla \bar{u}^2\|_{L^2L^2}
+\|\nabla \bar{u}^3\|_{Y_0}+\nu^{\frac{1}{2}}\|\nabla \bar{u}^3\|_{L^2 L^2}\\
&+\nu^{-\frac{1}{2}}\|\partial_t\bar{u}^3\|_{L^2L^2}+\left\|\min\{(\nu^{\frac{2}{3}}+\nu t)^{\frac{1}{2}}, 1-y^2\}\Delta \bar{u}^3\right\|_{Y_0}\\
&+\nu^{-\frac{1}{2}}\left\|\min\{(\nu^{\frac{2}{3}}+\nu t)^{\frac{1}{2}}, 1-y^2\}\partial_t\nabla \bar{u}^3\right\|_{L^2L^2}.
\end{align*}

To handle the non-zero modes of $(\Delta u^2_{\neq}, \omega^2_{\neq})$, let us introduce
$$E_3=E_{3,0}+E_{3,1},$$
where
\begin{align*}
E_{3,0}=&\|(\partial_x,\partial_z)\nabla u^2_{\neq}\|_{Y_2}+\|(\partial_x^2,\partial_z^2)u^3_{\neq}\|_{Y_2}\\
&+\nu^{\frac{3}{4}}\| {\rm e}^{2\epsilon \nu^{\frac{1}{3}}t}\Delta\nabla u^2_{\neq}\|_{L^2L^2}+\|{\rm e}^{2 \epsilon \nu^{\frac{1}{3}}t}\partial_x\nabla u^2_{\neq}\|_{L^2L^2},\\
E_{3,1}=&\nu^{\frac{1}{3}}\|\nabla \omega^2_{\neq}\|_{Y_2}.
\end{align*}

To handle the temperature $\theta$, we introduce
$$E_4=E_{4,0}+E_{4,1},$$
 where
\begin{align*}
E_{4,0}=\|\bar{\theta}\|_{Y_0}+\|\nabla\bar{\theta}\|_{Y_0}+\|\partial_z\nabla\bar{\theta}\|_{Y_0},\ \
E_{4,1}=\|\partial_x^2\theta_{\neq}\|_{X_3}+\|\partial_z^2\theta_{\neq}\|_{X_2}.
\end{align*}
In the above norms, $E_{4,0}$ is to handle the zero mode of temperature $\overline{\theta}$, $E_{4,1}$ is to control the non-zero mode of temperature $\theta_{\neq}$. The control of $E_{4,1}$ is based on the space-time estimates and freezing coefficient in time, see subsection \ref{E-4:E-5} for details.

Inspired by \cite{chenwz}, let us introduce
\begin{align*}
E_5=&\nu^{\frac{1}{6}}\|{\rm e}^{3\epsilon \nu^{\frac{1}{3}}t}\partial_x^2 (u^2_{\neq}, u^3_{\neq}) \|_{L^2L^2}.
\end{align*}
The estimate for $E_5$ is based on the quasilinear argument \cite{chenwz}.

With the above energy functionals and nonlinear estimates, the nonlinear stability follows from the continuity argument.

\section{Resolvent estimates for some linearized systems}\label{sec3}
\subsection{Resolvent estimates with Navier-slip boundary condition}
In this subsection, we consider the following problem
\begin{align}
\label{eqres}
\begin{cases}
-\nu\Delta w+{\rm i} k(V(y,z)-\lambda)w-a(\nu k^2)^{\frac{1}{3}}w=F,\\
w|_{y=\pm 1}=0, \partial_xw={\rm i} kw, \partial_xF={\rm i} kF, \\
\Delta \varphi=w, \varphi|_{y=\pm 1}=0,
\end{cases}
\end{align}
which $V(y,z)$ satisfies
\begin{align}
\|V-y\|_{H^4}<\varepsilon_0,(V-y)|_{y=\pm 1}=0,\label{eqVcon}
\end{align}
with $\varepsilon_0$ small enough determined later. In addition, we always assume that $\lambda\in \mathbb{R}$ and $a\in [0, \epsilon_1]$ for $\epsilon_1$ small enough determined later.

The following proposition states the resolvent estimates for \eqref{eqres}, which can be found in \cite{chenwz}.
\begin{proposition}\label{prores1}
Let $w\in H^2(\Omega)$ be a solution to \eqref{eqres} with $F\in H^1(\Omega)$. Then it holds that
\begin{align*}
&\nu^{\frac{2}{3}}|k|^{\frac{1}{3}}\|\nabla w\|_{L^2}+(\nu k^2)^{\frac{1}{3}}\|w\|_{L^2}+\nu\|\Delta w\|_{L^2}+|k|\|(V-\lambda)w\|_{L^2} \leq C\|F\|_{L^2},\\
&\nu^{\frac{2}{3}}|k|^{\frac{1}{3}}\|\nabla w\|_{L^2}+(\nu k^2)^{\frac{1}{3}}\|w\|_{L^2}+\nu\|\Delta w\|_{L^2} \leq C\nu^{\frac{1}{6} }|k|^{-\frac{2}{3}}\|\nabla F\|_{L^2},\\
&\nu^{\frac{2}{3}}|k|^{\frac{1}{3}}\|\nabla w\|_{L^2}+(\nu k^2)^{\frac{1}{3}}\|w\|_{L^2} \leq C\nu^{-\frac{1}{3}}|k|^{\frac{1}{3}}\|F\|_{H^{-1}}.
\end{align*}
If $\nu k^2\leq 1$, then it holds that
\begin{align*}
&\nu^{\frac{1}{6}}|k|^{\frac{4}{3}}\|\nabla \varphi\|_{L^2}+\nu^{\frac{1}{6}} |k|^{\frac{11}{6}}\|\nabla[(V-\lambda)\varphi]\|_{L^2_{x,z}L^\infty_y}\leq C\|F\|_{L^2}.
\end{align*}
\end{proposition}
\begin{proof}
See Proposition 4.1 of \cite{chenwz}.
\end{proof}
\begin{proposition}
\label{prores2}
Let $w\in H^2(\Omega)$ be a solution of \eqref{eqres} with $F\in H^1(\Omega)$. If $F=F_1+F_2+\partial_xf_1+\partial_yf_2+\partial_zf_3$ and $F_2|_{y=\pm 1}=0$, then it holds
\begin{align*}
\nu^{\frac{2}{3}}\|\nabla w\|_{L^2}+(\nu |k|)^{\frac{1}{3}}\|w\|_{L^2} \leq C\left(|k|^{-\frac{1}{3}}\|F_1\|_{L^2}+\nu^{\frac{1}{6}}|k|^{-1}\|\nabla F_2\|_{L^2}+\nu^{-\frac{1}{3}}\|(f_1, f_2, f_3)\|_{L^2} \right).
\end{align*}
\end{proposition}
\begin{proof}
This proposition is a direct consequence of Proposition \ref{prores1}, and we omit the proof.
\end{proof}
\subsection{Resolvent estimate for the linear system}
In this subsection, we consider the following linearized Navier-Stokes system:
\begin{align}
\label{eqres2}
\begin{cases}
-\nu\Delta W+{\rm i} k(V(y,z)-\lambda)W-a(\nu k^2)^{1/3}W+(\partial_y+\kappa\partial_z)p^{L1}\\
\quad\quad+G_1+\nu(\Delta\kappa)U+2\nu \nabla\kappa\cdot\nabla U=0,\\
-\nu\Delta U+{\rm i} k(V(y,z)-\lambda)U-a(\nu k^2)^{1/3}U+G_2+\partial_zp^{L1}=0,\\
W|_{y=\pm 1}=\partial_yW|_{y=\pm1}=U|_{y=\pm1}=0,
\end{cases}
\end{align}
where
\begin{align*}
\Delta p^{L1}=-2 {\rm i} k\partial_yV \cdot W, \quad \partial_xW={\rm i} kW,\quad\partial_xU={\rm i} kU,\quad \partial_xp^{L1}={\rm i} kp^{L1},
\end{align*}
and $\lambda\in \mathbb{R}, a\in [0,\epsilon_1]$, $V$ satisfies \eqref{eqVcon}.

The resolvent estimates for \eqref{eqres2} are stated as follows.
\begin{proposition}\label{proWU}
Let $W\in H^4(\Omega)$, $U\in H^2(\Omega)$ be a solution of \eqref{eqres2}. Then it holds that
\begin{align*}
&\nu^{\frac{1}{3} }\left(\|\partial_x^2U\|_{L^2}^2+\|\partial_x(\partial_z-\kappa\partial_y)U\|_{L^2}^2\right)
+\nu\left(\|\nabla\partial_x^2U\|_{L^2}^2+\|\nabla\partial_x(\partial_z-\kappa\partial_y)U\|_{L^2}^2\right)\\
&\quad\quad+\nu^{\frac{1}{3}}\|\partial_x\nabla W\|_{L^2}^2+\nu\|\partial_x\Delta W\|_{L^2}^2+\nu^{\frac{5}{3}}\|\partial_x\Delta W\|_{L^2}^2\leq C\nu^{-1}\left(\|\nabla G_1\|_{L^2}^2+\|\partial_xG_2\|_{L^2}^2\right).
\end{align*}
\end{proposition}
\begin{proof}
See Proposition 9.1 in \cite{chenwz}.
\end{proof}

\section{Space-time estimates of the linearized systems}\label{sec4}
In this section, we establish some space-time estimates for the linearized equation.
\subsection{Space-time estimates for linearized system around the Couette flow with Navier-slip boundary condition}
In this subsection, we study the following system
\begin{align}\label{spacey}
\begin{cases}
\partial_t\omega-\nu(\partial_y^2-\eta^2)\omega+ {\rm i} ky\omega=-{\rm i} kf_1-\partial_yf_2-{\rm i} lf_3-f_4,\\
\omega|_{y=\pm1}=0,\omega|_{t=0}=\omega_{in},
\end{cases}
\end{align}
where $(k,l)\in \mathbb{Z}^2, k\neq 0, \eta^2=k^2+l^2$.

The following proposition comes from \cite{chenwz}.
\begin{proposition}\label{space-na-y}
Let $\omega$ be a solution of \eqref{spacey} with $f_4(t,\pm 1)=0$ and $\omega_{in}(\pm1)=0$. Then there exists $a>0$ such that
\begin{align*}
&\|{\rm e}^{a\nu^{\frac{1}{3}}t}\omega\|_{L^\infty L^2}^2+\nu\|{\rm e}^{a\nu^{\frac{1}{3}}t}\omega'\|_{L^2 L^2}^2+\left(\nu \eta^2+(\nu k^2)^{\frac{1}{3}}\right)\|{\rm e}^{a\nu^{\frac{1}{3}}t}\omega\|_{L^2 L^2}^2\\
&\leq C\left(\|\omega_{in}\|_{L^2}^2+\nu^{-1}\|{\rm e}^{a\nu^{\frac{1}{3}}t}f_2\|_{L^2 L^2}^2+(|\eta||k|)^{-1}\|{\rm e}^{a\nu^{\frac{1}{3}}t}\partial_yf_4\|_{L^2 L^2}^2
+|\eta| |k|^{-1}\|{\rm e}^{a\nu^{\frac{1}{3}}t}f_4\|_{L^2 L^2}^2\right.\\
&\quad\quad\left. +\min\{(\nu \eta^2)^{-1}, (\nu k^2)^{-\frac{1}{3}}\}\|{\rm e}^{a\nu^{\frac{1}{3}}t}(kf_1+lf_3)\|_{L^2 L^2}^2\right).
\end{align*}
Moreover, we have
\begin{align*}
&\|{\rm e}^{a\nu^{\frac{1}{3}}t}\omega'\|_{L^\infty L^2}^2+\nu\|{\rm e}^{a\nu^{\frac{1}{3}}t}\omega''\|_{L^2 L^2}^2+\nu \eta^2\|{\rm e}^{a\nu^{\frac{1}{3}}t}\omega'\|_{L^2 L^2}^2\\
&\leq C\|\omega_{in}'\|_{L^2}^2+C\nu^{-\frac{2}{3}}|k|^{\frac{2}{3}}\left(\|\omega_{in}\|_{L^2}^2+(\eta|k|)^{-1}\|e^{a\nu^{\frac{1}{3}}t}\partial_yf_4\|_{L^2 L^2}^2
+\eta|k|^{-1}\|{\rm e}^{a\nu^{\frac{1}{3}}t}f_4\|_{L^2 L^2}^2\right)\\
&\quad\quad+C\nu^{-1}\left(\|{\rm e}^{a\nu^{\frac{1}{3}}t}(kf_1+lf_3)\|_{L^2 L^2}^2+\nu^{-\frac{2}{3}}|k|^{\frac{2}{3}}\|{\rm e}^{a\nu^{\frac{1}{3}}t}f_2\|_{L^2 L^2}^2+\|{\rm e}^{a\nu^{\frac{1}{3}}t}\partial_yf_2\|_{L^2 L^2}^2\right).
\end{align*}
\end{proposition}
\begin{proof}
See Proposition 10.1 in \cite{chenwz}.
\end{proof}
\subsection{Space-time estimate for linearized system around general flow $V$ with Navier-slip boundary condition }\label{V-slip}
The main new trouble is to control the 3-D lift-up effect affect the temperature via the coupling effect between the velocity and temperature. To this end, we will use the quasi-linear method, this idea has been introduced in  \cite{chenwz} to treat with the NS flow with non-slip boundary condition. In this paper, we will use the idea to handle the temperature. Precisely, we consider
\begin{align}\label{spaceV}
\begin{cases}
\partial_t\omega-\nu(\partial_y^2+\partial_z^2-k^2)\omega+{\rm i} kV(y,z) \omega=-{\rm i} kf_1-\partial_yf_2-\partial_zf_3,\\
\omega|_{y=\pm1}=0,\omega|_{t=0}=\omega_{in},
\end{cases}
\end{align}
where $V$ satisfies \eqref{eqVcon}, $k\in \mathbb{Z}\setminus \{0\}.$
\begin{proposition}\label{space-na-V}
Let $\omega$ be a solution of \eqref{spaceV} with $\omega_{in}(\pm1)=0$. Then it holds that
\begin{align*}
&\| {\rm e}^{a\nu^{\frac{1}{3}}t} \omega\|_{L^\infty L^2}^2+\nu\|{\rm e}^{a\nu^{\frac{1}{3}}t} (\partial_y,\partial_z) \omega \|_{L^2 L^2}^2+\left(\nu k^2+(\nu k^2)^{\frac{1}{3}}\right)\|{\rm e}^{a\nu^{\frac{1}{3}}t} \omega\|_{L^2 L^2}^2\\
&\leq C\left(\|\omega_{in}\|_{L^2}^2+\nu^{-1}\|{\rm e}^{a\nu^{\frac{1}{3}}t}  (f_2,f_3)\|_{L^2 L^2}^2+\min\{(\nu k^2)^{-1}, (\nu k^2)^{-\frac{1}{3}}\}\|{\rm e}^{a\nu^{\frac{1}{3}}t}  kf_1\|_{L^2 L^2}^2\right).
\end{align*}
\end{proposition}

To obtain the desired estimates, let us consider the following Orr-Sommerfeld equation
\begin{align}
\label{eqresV}
\begin{cases}
-\nu(\partial_y^2+\partial_z^2-k^2) w+{\rm i} k(V(y,z)-\lambda)w-a(\nu k^2)^{\frac{1}{3}}w=F,\\
w|_{y=\pm 1}=0.
\end{cases}
\end{align}
Note that $V=V(y,z)$ and $w(y,z), F(y,z)$ solves \eqref{eqresV}, then $(w(y,z) {\rm e}^{{\rm i} kx}, F(y,z){\rm e}^{{\rm i}kx})$ solves \eqref{eqres}. Therefore, it follows from Propositions \ref{prores1} -- \ref{prores2} that
\begin{proposition}\label{proresoV}
Let $w\in H^2(\Omega)$ solves \eqref{eqresV} with $F\in H^1(\Omega)$. Then it holds that
\begin{align*}
&\nu^{\frac{2}{3}}|k|^{\frac{1}{3}}\|(\partial_y ,\partial_z)w \|_{L^2}+(\nu k^2)^{\frac{1}{3}}\|w\|_{L^2}\leq C\|F\|_{L^2},\\
&\nu^{\frac{2}{3}}|k|^{\frac{1}{3}}\|(\partial_y ,\partial_z) w\|_{L^2}+(\nu k^2)^{\frac{1}{3}}\|w\|_{L^2}\leq C\nu^{-\frac{1}{3}}|k|^{\frac{1}{3}}\|F\|_{H^{-1}}.
\end{align*}
Moreover, if $F=F_1+\partial_yF_2+\partial_zF_3$, it holds that
\begin{align*}
&\nu^{\frac{2}{3} }\|(\partial_y ,\partial_z) w\|_{L^2}+(\nu |k|)^{\frac{1}{3}}\|w\|_{L^2}\leq C\left(|k|^{-\frac{1}{3}}\|F_1\|_{L^2}+\nu^{-\frac{1}{3}}\|(F_2, F_3)\|_{L^2}\right).
\end{align*}
\end{proposition}
Let $\tilde{\omega}= {\rm e}^{a\nu^{\frac{1}{3}}t} \omega$ and $\tilde{f}_j= {\rm e}^{a\nu^{\frac{1}{3}}t} f_j$, then $\tilde{\omega}$ satisfies
\begin{align*}
\begin{cases}
\partial_t\tilde{\omega}-\nu(\partial_y^2+\partial_z^2-k^2)\tilde{\omega}+{\rm i}kV\tilde{\omega}-a\nu^{\frac{1}{3} }\tilde{\omega}=-{\rm i} k\tilde{f}_1-\partial_y\tilde{f}_2-\partial_z\tilde{f}_3,\\
\tilde{\omega}|_{y=\pm1}=0,\tilde{\omega}|_{t=0}=\omega_{in}.
\end{cases}
\end{align*}
We decompose $\tilde{\omega}=\omega_I+\omega_H$, where $\omega_I$ and $\omega_H$ satisfies
\begin{align}\label{eqomeI}
\begin{cases}
\partial_t\omega_I-\nu(\partial_y^2+\partial_z^2-k^2)\omega_I+{\rm i} kV\omega_I-a\nu^{\frac{1}{3}}\omega_I= -{\rm i} k\tilde{f}_1-\partial_y\tilde{f}_2-\partial_z\tilde{f}_3,\\
\omega_I|_{y=\pm1}=0,\omega_I|_{t=0}=0,
\end{cases}
\end{align}
and
\begin{align}\label{eqomeH}
\begin{cases}
\partial_t\omega_H-\nu(\partial_y^2+\partial_z^2-k^2)\omega_H+{\rm i} kV\omega_H-a \nu^{\frac{1}{3}}\omega_H=0,\\
\omega_H|_{y=\pm1}=0,\omega_H|_{t=0}=\omega_{in}.
\end{cases}
\end{align}
\begin{lemma} \label{lemH}
Let $\omega_H$ solves \eqref{eqomeH} with $\omega_{in}|_{y=\pm 1}=0$. Then it holds that
\begin{align*}
\|\omega_H\|_{L^\infty L^2}^2+\nu\|(\partial_y , \partial_z) \omega_H \|_{L^2L^2}^2+\left(\nu k^2+(\nu k^2)^{\frac{1}{3}}\right)\|\omega_H\|_{L^2L^2}^2\leq C\|\omega_{in}\|_{L^2}^2.
\end{align*}
\end{lemma}
\begin{proof}
Let $L_k=\nu(k^2 -\partial_y^2-\partial_z^2)+{\rm i} ky$ with $D(L_k)=H^2\cap H^1_0$, then the solution is given by
 $$\omega_H={\rm e}^{-tL_k+ta\nu^{1/3}}\omega_{in}.$$
Thanks to $\forall f\in D(L_k)$, one has ${\bf Re} \langle L_kf,f\rangle=\nu\|(\partial_y,\partial_z)f\|_{L^2}^2+\nu k^2\|f\|_{L^2}^2$, the operator $L_k$ is accretive for any $k\in \mathbb{Z}\setminus\{0\}$.

Let \begin{align*}
\Psi(L_k)=\inf\{\|(L_k- {\rm i}\lambda)f\|:f\in D(L_k), \lambda\in \mathbb{R}, \|f\|=1\},
\end{align*}
it remains to bound $\Psi(L_k)$.

For $\nu k^2\geq 1$, $\Psi(L_k)\geq \nu k^2\geq (\nu k^2)^{1/3}\geq 2a\nu^{1/3}$.

For $\nu k^2 \leq 1$, by Proposition \ref{proresoV}, there exist $c>0$ so that for any $k\in \mathbb{Z} \setminus \{0\}$,
\begin{align*}
\Psi(L_k)\geq c(\nu k^2)^{1/3}\geq 2a \nu^{1/3}.
\end{align*}
Then it follows from Lemma \ref{Gearhart} that
\begin{align*}
\|\omega_H(t)\|_{L^2}\leq {\rm e}^{-t\Psi(L_k)+\pi/2+ta\nu^{1/3}}\|\omega_{in}\|_{L^2}\leq C {\rm e}^{-a (\nu k^2)^{\frac{1}{3}}t}\|\omega_{in}\|_{L^2},
\end{align*}
which implies that
\begin{align*}
\|\omega_H\|_{L^\infty L^2}^2+(\nu k^2)^{\frac{1}{3}}\|\omega_H\|_{L^2L^2}^2\leq C\|\omega_{in}\|_{L^2}^2.
\end{align*}

In addition, the energy argument yields that
\begin{align*}
\frac{1}{2}\frac{d}{dt}\|\omega_H\|_{L^2}^2+\nu\|(\partial_y,\partial_z)\omega_H\|_{L^2}^2+\nu k^2\|\omega_H\|_{L^2}^2=a \nu^{\frac{1}{3}}\|\omega_H\|_{L^2}^2,
\end{align*}
which yields that
\begin{align*}
\nu \|(\partial_y,\partial_z)\omega_H\|_{L^2L^2}^2+\nu k^2\|\omega_H\|_{L^2L^2}^2\leq C\left(\|\omega_{in}\|_{L^2}^2+a\nu^{\frac{1}{3} }\|\omega_H\|_{L^2 L^2}^2\right)\leq C\|\omega_{in}\|_{L^2}^2.
\end{align*}

The proof is completed.
\end{proof}

\begin{lemma} \label{lemI}
Let $\omega_I$ solves \eqref{eqomeI}. If $\nu k^2\leq 1$, it holds that
\begin{align*}
&\|\omega_I\|_{L^\infty L^2}^2+\nu\|(\partial_y , \partial_z) \omega_I \|_{L^2L^2}^2+\left[\nu k^2+(\nu k^2)^{\frac{1}{3}}\right]\|\omega_I\|_{L^2L^2}^2\\
&\leq C\left[(\nu k^2)^{-\frac{1}{3}}\|\tilde{f}_1\|_{L^2L^2}^2+\nu^{-1}\|(\tilde{f}_2,\tilde{f}_3)\|_{L^2L^2}^2\right].
\end{align*}
\end{lemma}
\begin{proof}
We first extend the solution $\omega_I$ to $t>T$ by solving \eqref{eqomeI} with $\tilde{f}_j=0$, $j=1,2,3$. Let
\begin{align*}
w(\lambda,y,z)=\frac{1}{2\pi}\int_{\mathbb{R}^+}\omega_I(t,y,z) {\rm e}^{-{\rm i} \lambda t}dt, \hat{f}_j(\lambda, y, z)=\frac{1}{2\pi}\int_{\mathbb{R}^+}\tilde{f}_j (t,y,z) {\rm e}^{-{\rm i}\lambda t}dt, j=1,2,3,
\end{align*}
then one has
\begin{align*}
\begin{cases}
-\nu(\partial_y^2+\partial_z^2-k^2) w+{\rm i}k(V(y,z)-\lambda)w-a(\nu k^2)^{\frac{1}{3}}w=-{\rm i} k\hat{f}_1-\partial_y\hat{f}_2-\partial_z\hat{f}_3,\\
w|_{y=\pm 1}=0.
\end{cases}
\end{align*}
The resolvent estimates in Proposition \ref{proresoV} give that
\begin{align*}
&\nu^{\frac{2}{3}}\|(\partial_y ,\partial_z) w\|_{L^2}+(\nu |k|)^{\frac{1}{3}}\|w\|_{L^2} \leq C\left(|k|^{-\frac{1}{3}}\|\hat{f}_1\|_{L^2}+\nu^{-\frac{1}{3}}\|(\hat{f}_2, \hat{f}_3)\|_{L^2} \right),
\end{align*}
 which, together with Plancherel's theorem, yields that
\begin{align}
&\nu\|(\partial_y,\partial_z)\omega_I\|_{L^2_tL^2_{y,z}}^2+(\nu k^2)^{\frac{1}{3}}\|\omega_I\|_{L^2_tL^2_{y,z}}^2\nonumber\\
=& \nu\|(\partial_y,\partial_z)w\|_{L^2_\lambda L^2_{y,z} }^2+(\nu k^2)^{\frac{1}{3}}\|w\|_{L^2_\lambda L^2_{y,z}}^2\nonumber\\
\leq&C \left[(\nu|k|^2)^{-\frac{1}{3}}\|\hat{f}_1\|_{L^2_\lambda L^2_{y,z}}^2+\nu^{-1}\|(\hat{f}_2, \hat{f}_3)\|_{L^2_\lambda L^2_{y,z}}^2\right]\nonumber\\
= &C \left[(\nu|k|^2)^{-\frac{1}{3}}\|\tilde{f}_1\|_{L^2_t L^2_{y,z}}^2+\nu^{-1}\|(\tilde{f}_2, \tilde{f}_3)\|_{L^2_t L^2_{y,z}}^2\right].\label{eqI1}
\end{align}

On the other hand, the standard energy argument yields that
\begin{align*}
&\frac{1}{2}\frac{d}{dt}\|\omega_I\|_{L^2}^2+\nu \|(\partial_y,\partial_z)\omega_I\|_{L^2}^2+\nu k^2\|\omega_I\|_{L^2}^2\\
=&\mathbf{Re} \langle  -{\rm i} k\tilde{f}_1-\partial_y\tilde{f}_2-\partial_z\tilde{f}_3,\omega_I\rangle+a\nu^{\frac{1}{3}}\|\omega_I\|_{L^2}^2\\
\leq& \|k\tilde{f}_1\|_{L^2} \|\omega_I\|_{L^2}+\|\tilde{f}_2\|_{L^2} \|\partial_y \omega_I\|_{L^2}+\|\tilde{f}_3\|_{L^2} \|\partial_z\omega_I\|_{L^2}+a\nu^{\frac{1}{3}}\|\omega_I\|_{L^2}^2\\
\leq& \frac{1}{4}\nu\|(\partial_y, \partial_z)\omega_I\|_{L^2}^2+C(\nu k^2)^{-\frac{1}{3}}\|\tilde{f}_1\|_{L^2}^2+ C(\nu k^2)^{\frac{1}{3}}\|\omega_I\|_{L^2}^2+C \nu^{-1}\|(\tilde{f}_2, \tilde{f}_3)\|_{L^2}^2,
\end{align*}
which gives that
\begin{align*}
&\|\omega_I\|_{L^2}^2+\nu \|(\partial_y,\partial_z)\omega_I\|_{L^2}^2+\nu k^2\|\omega_I\|_{L^2}^2\\
\lesssim& (\nu k^2)^{-\frac{1}{3}}\|\tilde{f}_1\|_{L^2}^2+ (\nu k^2)^{\frac{1}{3}}\|\omega_I\|_{L^2}^2+\nu^{-1}\|(\tilde{f}_2, \tilde{f}_3)\|_{L^2}^2.
\end{align*}
Therefore, one has
\begin{align*}
&\|\omega_I\|_{L^\infty L^2}^2+\nu \|(\partial_y,\partial_z)\omega_I\|_{L^2 L^2}^2+\nu k^2\|\omega_I\|_{L^2 L^2}^2\\
\lesssim&  (\nu k^2)^{-\frac{1}{3}}\|\tilde{f}_1\|_{L^2 L^2}^2+ (\nu k^2)^{\frac{1}{3}}\|\omega_I\|_{L^2 L^2}^2+\nu^{-1}\|(\tilde{f}_2, \tilde{f}_3)\|_{L^2 L^2}^2\\
\lesssim&(\nu k^2)^{-\frac{1}{3}}\|\tilde{f}_1\|_{L^2 L^2}^2+\nu^{-1}\|(\tilde{f}_2, \tilde{f}_3)\|_{L^2 L^2}^2,
\end{align*}
where \eqref{eqI1} has been used. This completes the proof.
\end{proof}

Now we are in a position to prove Proposition \ref{space-na-V}.
\begin{proof}[Proof of Proposition \ref{space-na-V}]
The proof will be completed by two cases: $\nu k^2\leq 1$ and $\nu k^2\geq 1$.

{\bf Case I. $\nu k^2\leq 1$. }
In this case, it is clear that $(\nu k^2)^{-1/3}\leq (\nu k^2)^{-1}$, thus $ (\nu k^2)^{-1/3}\leq \min\{  (\nu k^2)^{-1/3},(\nu k^2)^{-1}\}$.
Using the Lemmas \ref{lemH} and \ref{lemI}, one has
\begin{align*}
&\|\tilde{\omega}\|_{L^\infty L^2}^2+\nu\|(\partial_y,\partial_z)\tilde{\omega}\|_{L^2 L^2}^2+\left(\nu k^2+(\nu k^2)^{\frac{1}{3}}\right)\|\tilde{\omega}\|_{L^2 L^2}^2\\
&\leq C\left(\|\omega_{in}\|_{L^2}^2+\nu^{-1}\|(\tilde{f_2},\tilde{f_3})\|_{L^2 L^2}^2+(\nu k^2)^{-\frac{1}{3}}\|k\tilde{f}_1\|_{L^2 L^2}^2\right)\\
&\leq C\left(\|\omega_{in}\|_{L^2}^2+\nu^{-1}\|(\tilde{f_2},\tilde{f_3})\|_{L^2 L^2}^2+\min\{(\nu k^2)^{-1}, (\nu k^2)^{-\frac{1}{3}}\}\|k\tilde{f}_1\|_{L^2 L^2}^2\right).
\end{align*}

{\bf  Case II. $\nu k^2\geq 1$.}
In this case, it is clear that $(\nu k^2)^{-1/3}\geq (\nu k^2)^{-1}$. The standard energy argument gives that
\begin{align*}
&\frac{1}{2}\frac{d}{dt}\|\tilde{\omega}\|_{L^2}^2+\nu\|(\partial_y,\partial_z)\tilde{\omega}\|_{L^2}^2
+(\nu k^2-a\nu^{1/3})\|\tilde{\omega}\|_{L^2}^2\\
&\leq \|k\tilde{f}_1\|_{L^2}\|\tilde{\omega}\|_{L^2}+\|\tilde{f}_2\|_{L^2}\|\partial_y\tilde{\omega}\|_{L^2}
+\|\tilde{f}_3\|_{L^2}\|\partial_z\tilde{\omega}\|_{L^2}\\
&\leq \frac{1}{4}\left[(\nu k^2)\|\tilde{\omega}\|_{L^2}^2+\nu\|(\partial_y,\partial_z)\tilde{\omega}\|_{L^2}^2\right]
+C\left[(\nu k^2)^{-1}\|k\tilde{f}_1\|_{L^2}^2+\nu^{-1}\|(\tilde{f}_2,\tilde{f}_3)\|_{L^2}^2\right],
\end{align*}
which gives that
\begin{align*}
\frac{d}{dt}\|\tilde{\omega}\|_{L^2}^2+\nu\|(\partial_y,\partial_z)\tilde{\omega}\|_{L^2}^2
+\nu k^2 \|\tilde{\omega}\|_{L^2}^2
\lesssim (\nu k^2)^{-1}\|k\tilde{f}_1\|_{L^2}^2+\nu^{-1}\|(\tilde{f}_2,\tilde{f}_3)\|_{L^2}^2,
\end{align*}
where we have used the facts that $\nu k^2-a\nu^{1/3}\geq \frac{3}{4}\nu k^2$ and $(\nu k^2)^{-1/3}\geq (\nu k^2)^{-1}$.

Thus, one has
\begin{align*}
&\|\tilde{\omega}\|_{L^\infty L^2}^2+\nu\|(\partial_y,\partial_z)\tilde{\omega}\|_{L^2 L^2}^2
+\nu k^2 \|\tilde{\omega}\|_{L^2 L^2}^2\\
\lesssim& \|\omega_{in}\|_{L^2}^2+ (\nu k^2)^{-1}\|k\tilde{f}_1\|_{L^2 L^2}^2+\nu^{-1}\|(\tilde{f}_2,\tilde{f}_3)\|_{L^2 L^2}^2\\
\lesssim& \|\omega_{in}\|_{L^2}^2+ \min\{ (\nu k^2)^{-\frac{1}{3}}, (\nu k^2)^{-1}\} \|k\tilde{f}_1\|_{L^2 L^2}^2+\nu^{-1}\|(\tilde{f}_2,\tilde{f}_3)\|_{L^2 L^2}^2,
\end{align*}
which completes the proof.
\end{proof}

 Let us also consider
\begin{align}\label{spaceV-1}
\begin{cases}
\partial_t\omega-\nu\Delta\omega+V\partial_x\omega=\partial_xf_1-\partial_yf_2-\partial_zf_3,\\
\omega|_{y=\pm1}=0,\omega|_{t=0}=\omega_{in},
\end{cases}
\end{align}
where $V$ satisfies \eqref{eqVcon}.

The following proposition is the direct consequence for Proposition \ref{space-na-V}.
\begin{proposition}\label{space-time-V21}
Let $\omega$ be a solution of \eqref{spaceV-1} with $\omega_{in}(\pm1)=0$. Then it holds that
\begin{align*}
&\| {\rm e}^{a\nu^{\frac{1}{3}}t}\omega\|_{L^\infty L^2}^2+\nu\|{\rm e}^{a\nu^{\frac{1}{3}}t}\nabla\omega\|_{L^2 L^2}^2+\nu^{\frac{1}{3}}\|{\rm e}^{a\nu^{\frac{1}{3}}t}\omega\|_{L^2 L^2}^2\\
\leq& C\bigg(\|\omega_{in}\|_{L^2} ^2+\nu^{-1}\|{\rm e}^{a\nu^{\frac{1}{3}}t}(f_2,f_3)\|_{L^2 L^2}^2\\
&+\min\{\nu^{-1}\| {\rm e}^{a\nu^{\frac{1}{3}}t}f_1\|_{L^2L^2}^2, \nu^{-\frac{1}{3}}\|{\rm e}^{a\nu^{\frac{1}{3}}t}|\partial_x|^{\frac{2}{3}} f_1\|_{L^2 L^2}^2\}\bigg).
\end{align*}
\end{proposition}
\subsection{Space-time estimate for linearized NS equations around the Couette flow with non-slip boundary condition }
Let us recall that the linearized NS equations around the Couette flow with non-slip boundary condition as follows:
\begin{align}\label{spacey1}
\begin{cases}
\partial_t\omega-\nu(\partial_y^2-\eta^2)\omega+{\rm i}ky\omega=F,\\
(\partial_y^2-\eta^2)\varphi=\omega, \partial_y\varphi|_{y=\pm 1}=\varphi|_{y=\pm 1}=0,\\
\omega|_{t=0}=\omega_{in},
\end{cases}
\end{align}
where $| \eta| \geq |k|\geq 1$ and $0<\nu\leq \nu_0$, $0\leq a\leq \epsilon_1\leq 1/8$.

The space-time estimates for \eqref{spacey1} are stated as follows.
\begin{proposition}\label{space-time-non1}
Let $\omega$ solves \eqref{spacey1} with $\partial_y\varphi_{in}|_{y=\pm1}=0$ and $F={\rm i} kf_1+\partial_yf_2+{\rm i} lf_3$. Then it holds that
\begin{align*}
&|k\eta|^{\frac{1}{2}}\|{\rm e}^{a\nu^{\frac{1}{3}}t}(\partial_y,\eta)\varphi\|_{L^2L^2}+\nu^{\frac{3}{4}}\|{\rm e}^{a\nu^{\frac{1}{3}}t} \partial_y\omega\|_{L^2L^2}
+\nu^{\frac{1}{2}} |\eta| \|{\rm e}^{a\nu^{\frac{1}{3}}t} \omega\|_{L^2L^2}\\
&\quad\quad+|\eta| \|{\rm e}^{a\nu^{\frac{1}{3}}t} (\partial_y,\eta)\varphi\|_{L^\infty L^2}+\nu^{\frac{1}{4}}\|{\rm e}^{a\nu^{\frac{1}{3}}t} \omega\|_{L^\infty L^2}\\
&\leq C\nu^{-\frac{1}{2}}\|{\rm e}^{a\nu^{\frac{1}{3}}t} (f_1,f_2,f_3)\|_{L^2L^2}+C\left(|\eta|^{-1}\|\partial_y\omega_{in}\|_{L^2}+\|\omega_{in}\|_{L^2}\right).
\end{align*}
\end{proposition}
\begin{proof}
See Proposition 10.2 in \cite{chenwz}.
\end{proof}

\section{Nonlinear interactions}

\subsection{Estimates for velocity in terms of energy}
\begin{lemma}\label{lemu1-0}
It holds that
\begin{align*}
&\|\bar{u}^1\|_{H^2}\leq CE_1\min\{\nu t+\nu^{2/3}, 1\},\\
&\left\|\frac{\partial_zu^{1,0}}{\min\{(\nu t)^{1/2}, 1-y^2\}}\right\|_{L^\infty L^\infty}\leq CE_{1,0}.
\end{align*}
\end{lemma}
\begin{proof}
See Lemma 11.6-11.7 in \cite{chenwz}.
\end{proof}
\begin{lemma}\label{lemux1}
It holds that for $k\geq 0$,
\begin{align*}
&\|\nabla^k(\partial_x,\partial_z)\partial_xu_{\neq}\|_{L^2}\leq C\left(\|\nabla^k(\partial_x^2+\partial_z^2)u_{\neq}^3\|_{L^2}+\|\nabla^{k+1}(\partial_x,\partial_z)u^2_{\neq}\|_{L^2}\right),\\
&\|\nabla^k(\partial_x,\partial_z)u_{\neq}\|_{L^2} \leq \|\nabla^k\omega_{\neq}^2\|_{L^2}+\|\nabla^{k+1}u^2_{\neq}\|_{L^2}.
\end{align*}
\end{lemma}
\begin{proof}
See Lemma 11.5 in \cite{chenwz}.
\end{proof}
\begin{lemma}\label{lemu0-2}
It holds that for $k\in \{2,3\}$,
\begin{align*}
&\|\bar{u}^2\|_{H^2}+\|\nabla \bar{u}^2\|_{H^1}+\|\bar{u}^3\|_{H^1}+\|\partial_z\bar{u}^3\|_{H^1}\leq CE_2,\\
&\|\bar{u}^k\|_{L^\infty L^\infty}+\nu^{\frac{1}{2}}\|\nabla \bar{u}^k\|_{L^2L^\infty}\leq CE_2,\\
&\|\nabla (\overline{u^k} f)\|_{L^2}+\| \overline{u^k} \nabla f\|_{L^2} \leq C E_2 \|f\|_{H^1}.
\end{align*}
\end{lemma}
\begin{proof}
See Lemma 11.8 in \cite{chenwz}.
\end{proof}
\subsection{Interaction between zero-modes}
\begin{lemma}\label{uthe-00}
It holds that
\begin{align*}
\|\bar{u}\cdot\nabla\bar{\theta}\|_{L^2L^2}+\|\partial_z(\bar{u}\cdot\nabla\bar{\theta})\|_{L^2L^2}
\leq C\nu^{-\frac{1}{2}}E_2E_{4,0}.
\end{align*}
\end{lemma}
\begin{proof}
Note that $\bar{u}\cdot\nabla \bar{\theta}=\bar{u}^2\partial_y\bar{\theta}+\bar{u}^3\partial_z\bar{\theta}$, it follows from the Lemma \ref{lemu0-2} that
\begin{align*}
\|\bar{u}\cdot\nabla \bar{\theta}\|_{L^2}\leq
\|\bar{u}^2\partial_y\bar{\theta}\|_{L^2}+\|\bar{u}^3\partial_z\bar{\theta}\|_{L^2}
\leq \left(\|\bar{u}^2\|_{L^\infty}+\|\bar{u}^3\|_{L^\infty}\right)\|\nabla \bar{\theta}\|_{L^2} \leq CE_2\|\nabla\bar{\theta}\|_{L^2} ,
\end{align*}
and
\begin{align*}
\|\partial_z(\bar{u}\cdot\nabla\bar{\theta})\|_{L^2}\leq&  \|\bar{u}^2\partial_z\partial_y\bar{\theta}\|_{L^2}+\|\partial_z\bar{u}^2\partial_y\bar{\theta}\|_{L^2}
+\|\bar{u}^3\partial_z^2\bar{\theta}\|_{L^2} +\|\partial_z\bar{u}^3\partial_z\bar{\theta}\|_{L^2} \\
\leq & C\left(\|\bar{u}^2\|_{L^\infty}+\|\bar{u}^3\|_{L^\infty}\right)\|\nabla\partial_z \bar{\theta}\|_{L^2}
+\left(\|\nabla\bar{u}^2\|_{L^\infty}+\|\nabla\bar{u}^3\|_{L^\infty}\right)\|\nabla \bar{\theta}\|_{L^2}\\
\leq& CE_2\|\nabla\partial_z\bar{\theta}\|_{L^2}+\left(\|\nabla\bar{u}^2\|_{L^\infty}+\|\nabla\bar{u}^3\|_{L^\infty}\right)E_{4,0},
\end{align*}
from which and Lemma \ref{lemu0-2}, one has
\begin{align*}
&\|\bar{u}\cdot\nabla\bar{\theta}\|_{L^2L^2}+\|\partial_z(\bar{u}\cdot\nabla\bar{\theta})\|_{L^2L^2}\\
\leq& CE_2(\|\nabla \bar{\theta}\|_{L^2L^2}+\|\nabla\partial_z\bar{\theta}\|_{L^2L^2})+CE_{4,0}
\left(\|\nabla\bar{u}^2\|_{L^2L^\infty}+\|\nabla\bar{u}^3\|_{L^2L^\infty}\right)\\
\leq& C\nu^{-\frac{1}{2}}E_2E_{4,0},
\end{align*}
which yields the conclusion.
\end{proof}
\begin{lemma}\label{lemuu-00}
It holds that
\begin{align*}
&\|\partial_t\nabla(\bar{u}^2\partial_yu^{1,0},\bar{u}^3\partial_zu^{1,0})\|_{L^2L^2}^2\leq C\nu E_{1,0}^2E_2^2,\\
&\|\Delta(\bar{u}^2\partial_y\bar{u}^{1,0},\bar{u}^3\partial_zu^{1,0})\|_{L^\infty L^2}^2\leq CE_2^2E_{1,0}^2,\\
&\|\nabla(\bar{u}^2\partial_yu^{1,1}, \bar{u}^3\partial_zu^{1,1})\|_{L^2L^2}^2\leq C\nu^{-1}E_2^2E_{1,1}^2.
\end{align*}
\end{lemma}
\begin{proof}
The first estimate can be founded in Lemma 11.12 of \cite{chenwz}.
The second estimate is a consequence of Lemma 11.11 of \cite{chenwz}. We focus on the last estimate.

 In fact, one has
\begin{align*}
&\|\bar{u}^2\partial_y\nabla u^{1,1}\|_{L^2}+\|\bar{u}^3\nabla\partial_zu^{1,1}\|_{L^2}+\|\nabla\bar{u}^2\partial_yu^{1,1}\|_{L^2}
+\|\nabla\bar{u}^3\partial_zu^{1,1}\|_{L^2}\\
\leq& \left(\|\bar{u}^2\|_{L^\infty}+\|\bar{u}^3\|_{L^\infty} \right)\|u^{1,1}\|_{H^2}
+\|(\nabla\bar{u}^2, \nabla\bar{u}^3)\|_{L^\infty_yL^2_z}\|(\partial_yu^{1,1}, \partial_zu^{1,1})\|_{L^2_yL^\infty_z}\\
\leq& C\left(\|\Delta \bar{u}^2\|_{L^2}+\|\Delta \bar{u}^3\|_{L^2} \right)\|u^{1,1}\|_{H^2}\leq C\left(\|\Delta \bar{u}^2\|_{L^2}+\|\Delta \bar{u}^3\|_{L^2} \right)E_{1,1},
\end{align*}
which implies that
\begin{align*}
\|\nabla(\bar{u}^2\partial_yu^{1,1}, \bar{u}^3\partial_zu^{1,1})\|_{L^2L^2}\leq
C\left(\|\Delta \bar{u}^2\|_{L^2L^2}+\|\Delta \bar{u}^3\|_{L^2L^2}\right)E_{1,1}\leq C\nu^{-\frac{1}{2}}E_2E_{1,1}.
\end{align*}

The proof is completed.
  \end{proof}
\subsection{Interaction between nonzero-modes}
The following lemma gives the interaction between nonzero-modes of velocity and temperature.
\begin{lemma}\label{lemuthe-11}
It holds that
\begin{align*}
\|{\rm e}^{4\epsilon \nu^{\frac{1}{3}}t}(\partial_z,1)(u_{\neq}\cdot \nabla \theta_{\neq})\|_{L^2L^2}+\|{\rm e}^{4\epsilon \nu^{\frac{1}{3}}t}(\partial_x^2,1)(u_{\neq} \theta_{\neq})\|_{L^2L^2}\leq C\nu^{-\frac{1}{2}}E_{3,0}E_{4,1}.
\end{align*}
\end{lemma}
\begin{proof}
First of all, for $j\in \{1,3\}$ and $x, z \in \mathbb{T}^2$, it follows from Lemma \ref{lemux1} that
\begin{align*}
\|\partial_z(u_{\neq}^j\partial_j\theta_{\neq})\|_{L^2}
\leq & \|\partial_zu_{\neq}^j\partial_j\theta_{\neq}\|_{L^2}+\|u_{\neq}^j\partial_j\partial_z\theta_{\neq}\|_{L^2}\\
\leq& \|\partial_zu_{\neq}\|_{L^\infty_xL^2_{y,z}}\|\partial_j\theta_{\neq}\|_{L^2_xL^\infty_{y,z}}
+\|u_{\neq}^j\|_{L^\infty_{x,z}L^2_y}\|\partial_z\partial_j\theta_{\neq}\|_{L^2_{x,z}L^\infty_y}\\
\leq& C\left(\|\partial_z\partial_xu_{\neq}\|_{L^2}+\|(\partial_x,1)u_{\neq}\|_{L^2}\right)\|(\partial_x^2+\partial_z^2)\nabla \theta_{\neq}\|_{L^2}\\
\leq& C\left(\|(\partial_x,\partial_z)\nabla u_{\neq}^2\|_{L^2}+\|(\partial_x^2+\partial_z^2)u_{\neq}^3\|_{L^2} \right)\|(\partial_x^2+\partial_z^2)\nabla \theta_{\neq}\|_{L^2}.
\end{align*}
Meanwhile, using Lemma \ref{lemux1} again, we have
\begin{align*}
\|u_{\neq}\cdot\nabla \theta_{\neq}\|_{L^2}\leq & \|u_{\neq}^1\partial_x \theta_{\neq}\|_{L^2}+\|u_{\neq}^2\partial_y \theta_{\neq}\|_{L^2}+\|u_{\neq}^3\partial_z \theta_{\neq}\|_{L^2}\\
\leq& \|u_{\neq}^1\|_{L^\infty_xL^2_{y,z}}\|\partial_x\theta_{\neq}\|_{L^2_xL^\infty_{y,z}}
+\|u_{\neq}^2\|_{L^\infty_yL^2_{x,z}}\|\partial_y\theta_{\neq}\|_{L^2_yL^\infty_{x,z}}\\
&
+\|u_{\neq}^3\|_{L^\infty_xL^2_{y,z}}\|\partial_z\theta_{\neq}\|_{L^2_xL^\infty_{y,z}}\\
\leq& C\left(\|\partial_x u^1_{\neq}\|_{L^2}+\|(\partial_y,\partial_x)u^2_{\neq}\|_{L^2}+\|\partial_xu^3_{\neq}\|_{L^2}\right)
\|(\partial_x^2+\partial_z^2)\nabla \theta_{\neq}\|_{L^2}\\
\leq&C\left(\|(\partial_x,\partial_z)\nabla u_{\neq}^2\|_{L^2}+\|(\partial_x^2+\partial_z^2)u_{\neq}^3\|_{L^2}\right)\|(\partial_x^2+\partial_z^2)\nabla \theta_{\neq}\|_{L^2}
\end{align*}
and
\begin{align*}
\|\partial_z(u^2_{\neq}\partial_y\theta_{\neq})\|_{L^2}\leq &
\|\partial_zu^2_{\neq}\partial_y\theta_{\neq}\|_{L^2}+\|u^2_{\neq}\partial_y\partial_z\theta_{\neq}\|_{L^2}\\
\leq& \|\partial_zu_{\neq}^2\|_{L^\infty_y L^2_{x,z}}\|\partial_y\theta_{\neq}\|_{L^2_y L^\infty_{x,z}}
+\|u^2_{\neq}\|_{L^\infty_{y,z}L^2_x}\|\partial_y\partial_z\theta_{\neq}\|_{L^2_{y,z}L^\infty_x}\\
\leq& \|(\partial_z,\partial_x)\nabla u^2_{\neq}\|_{L^2} \|(\partial_x^2+\partial_z^2)\nabla\theta_{\neq}\|_{L^2}.
\end{align*}
Therefore, we deduce that
\begin{align*}
&\|{\rm e}^{4\epsilon \nu^\frac{1}{3}t}(\partial_z,1)(u_{\neq}\cdot \nabla \theta_{\neq})\|_{L^2L^2}\\
\leq& C\left(\|{\rm e}^{2\epsilon \nu^\frac{1}{3}t}(\partial_x,\partial_z)\nabla u_{\neq}^2\|_{L^\infty L^2}+\|{\rm e}^{2\epsilon \nu^\frac{1}{3}t}(\partial_x^2+\partial_z^2)u_{\neq}^3\|_{L^\infty L^2}\right)\\
&\cdot \|{\rm e}^{2\epsilon \nu^\frac{1}{3}t}(\partial_x^2+\partial_z^2)\nabla \theta_{\neq}\|_{L^2L^2}\\
\leq& C\nu^{-\frac{1}{2}}E_{3,0}E_{4,1}.
\end{align*}
By Sobolev embedding and Lemma \ref{lemux1}, we have
\begin{align*}
\|u_{\neq}\theta_{\neq}\|_{L^2} \leq & \|u_{\neq}\|_{L^2} \|\theta_{\neq}\|_{L^\infty} \\
\leq & C\|\partial_xu_{\neq}\|_{L^2} \left(\|\partial_x\partial_y\partial_z\theta_{\neq}\|_{L^2}+\|\partial_x\partial_y\theta_{\neq}\|_{L^2}
+\|\partial_x\partial_z\theta_{\neq}\|_{L^2}+\|\partial_x\theta_{\neq}\|_{L^2}\right)\\
\leq& C\left(\|(\partial_x,\partial_z)\nabla u^2_{\neq}\|_{L^2}+\|(\partial_x^2+\partial_z^2)u^3_{\neq}\|_{L^2}\right)\|(\partial_x^2+\partial_z^2)\nabla \theta_{\neq}\|_{L^2}
\end{align*}
and
\begin{align*}
&\|\partial_x^2u_{\neq}\cdot \theta_{\neq}\|_{L^2} +\|\partial_x u_{\neq}\partial_x\theta_{\neq}\|_{L^2}+\|u_{\neq}\partial_x^2\theta_{\neq}\|_{L^2}\\
\leq& \|\partial_x^2u_{\neq}\|_{L^2}\|\theta_{\neq}\|_{L^\infty}+\|\partial_xu_{\neq}\|_{L^\infty_zL^2_{x,y}}\|\partial_x\theta_{\neq}\|_{L^2_zL^\infty_{x,y}}
+\|u_{\neq}\|_{L^2_yL^\infty_{x,z}}\|\partial_x^2\theta_{\neq}\|_{L^\infty_yL^2_{x,z}}\\
\leq& \|\partial_x^2u_{\neq}\|_{L^2} \|\theta_{\neq}\|_{L^\infty}+(\|\partial_x\partial_zu_{\neq}\|_{L^2}+\|\partial_xu_{\neq}\|_{L^2})(\|\partial_x^2\partial_y\theta_{\neq}\|_{L^2}+\|\partial_x^2\theta_{\neq}\|_{L^2})
\\
\leq& \left(\|(\partial_x,\partial_z)\nabla u^2_{\neq}\|_{L^2}+\|(\partial_x^2+\partial_z^2)u^3_{\neq}\|_{L^2}\right)\|(\partial_x^2+\partial_z^2)\nabla\theta_{\neq}\|_{L^2}.
\end{align*}
Thus, we conclude that
\begin{align*}
&\|{\rm e}^{4\epsilon \nu^{\frac{1}{3}}t}(\partial_x^2,1)(u_{\neq} \theta_{\neq})\|_{L^2L^2}\\
\leq&C\left(\|{\rm e}^{2\epsilon \nu^\frac{1}{3}t}(\partial_x,\partial_z)\nabla u_{\neq}^2\|_{L^\infty L^2}+\|{\rm e}^{2\epsilon \nu^\frac{1}{3}t}(\partial_x^2+\partial_z^2)u_{\neq}^3\|_{L^\infty L^2}\right)\\
&\cdot \|{\rm e}^{2\epsilon \nu^\frac{1}{3}t}(\partial_x^2+\partial_z^2)\nabla \theta_{\neq}\|_{L^2L^2}\\
\leq& C\nu^{-\frac{1}{2}}E_{3,0}E_{4,1},
\end{align*}
which completes the proof.
\end{proof}
The following lemma states the interaction between nonzero-modes of velocity.
\begin{lemma}\label{lemuu-11}
It holds that
\begin{align*}
\|{\rm e}^{4\epsilon \nu^{\frac{1}{3}}t}|u_{\neq}|^2\|_{L^2L^2}^2+\|{\rm e}^{4\epsilon \nu^{\frac{1}{3}}t}u_{\neq}\cdot\nabla u_{\neq}\|_{L^2L^2}^2+\|{\rm e}^{4\epsilon \nu^{\frac{1}{3}}t}\partial_x(u_{\neq}\cdot\nabla u_{\neq})\|_{L^2L^2}^2\\
+\|{\rm e}^{4\epsilon \nu^{\frac{1}{3}}t}\partial_z(u_{\neq}\cdot\nabla u_{\neq}^3)\|_{L^2L^2}^2+\|{\rm e}^{4\epsilon \nu^{\frac{1}{3}}t}\nabla(u_{\neq}\cdot\nabla u_{\neq}^2)\|_{L^2L^2}^2\leq C\nu^{-1}E_3^4,
\end{align*}
and
\begin{align*}
\|{\rm e}^{4\epsilon \nu^{\frac{1}{3}}t}\nabla(u_{\neq}\cdot\nabla u_{\neq})\|_{L^2L^2}^2\leq C\nu^{-\frac{5}{3}}E_3^4.
\end{align*}
In particular, one has
\begin{align*}
\|(\nu^{2/3}+\nu t)^{1/2}\nabla (u_{\neq}\cdot \nabla u^3_{\neq})\|_{L^2L^2}^2\leq C\nu^{-1}E_3^4.
\end{align*}
\end{lemma}
\begin{proof}
See Lemma 11.10 of \cite{chenwz}.
\end{proof}
\subsection{Interaction between zero-mode and nonzero-mode}
\begin{lemma}\label{lemuthe-10}
For $j\in \{2,3\}$, there holds that
\begin{align}
&\|{\rm e}^{3\epsilon\nu^{\frac{1}{3}}t}u^{1,1}\partial_x^3\theta_{\neq}\|_{L^2L^2}
+\|{\rm e}^{2\epsilon\nu^{\frac{1}{3}}t} \partial_z^2(u^{1,1}\partial_x\theta_{\neq})\|_{L^2L^2}
\leq C\nu^{\frac{1}{6}}E_1E_{4,1}, \label{lem:u-tem-01}\\
&\|{\rm e}^{3\epsilon\nu^{\frac{1}{3}}t}\bar{u}^j\partial_j\partial_x^2\theta_{\neq}\|_{L^2L^2}
+\|{\rm e}^{2\epsilon\nu^{\frac{1}{3}}t}\partial_z^2(\bar{u}^j\theta_{\neq})\|_{L^2L^2}
\leq C\nu^{-\frac{1}{2}}E_2E_{4,1},\label{lem:u-tem-02}\\
&\|{\rm e}^{3\epsilon\nu^{\frac{1}{3}}t}\partial_x^2u_{\neq}^2 \cdot \bar{\theta}\|_{L^2L^2}
+\|{\rm e}^{3\epsilon\nu^{\frac{1}{3}}t}\partial_x^2u_{\neq}^3\cdot \bar{\theta}\|_{L^2L^2}
\leq C\nu^{-\frac{1}{6}}E_{4,0}E_5,\label{lem:u-tem-03}\\
&\|{\rm e}^{2\epsilon\nu^{\frac{1}{3}}t}\partial_z^2(u_{\neq}^2\bar{\theta})\|_{L^2L^2}
+\|{\rm e}^{2\epsilon\nu^{\frac{1}{3}}t}\partial_z^2(u_{\neq}^3\bar{\theta})\|_{L^2L^2}
\leq C\nu^{-\frac{1}{2}}E_{3,0}E_{4,0},\label{lem:u-tem-04}\\
&\|{\rm e}^{2\epsilon\nu^{\frac{1}{3}}t}\partial_z(u_{\neq}^2\partial_y\bar{\theta}+u_{\neq}^3\partial_z\bar{\theta})\|_{L^2L^2}\leq
C\nu^{-\frac{1}{2}}E_{3,0}E_{4,0},\label{lem:u-tem-05}\\
&\|{\rm e}^{2\epsilon \nu^{\frac{1}{3}}t}\partial_z^2(u^{1,0}\partial_x\theta_{\neq})\|_{L^2L^2}\leq C\nu^{\frac{1}{6}}E_1E_{4,1}.\label{lem:u-tem-06}
\end{align}
\end{lemma}
\begin{proof}
First of all, we have
\begin{align*}
&\|u^{1,1}\partial_x^3\theta_{\neq}\|_{L^2} \leq \|u^{1,1}\|_{L^\infty} \|\partial_x^3\theta_{\neq}\|_{L^2}  \leq \|u^{1,1}\|_{H^2}\|\partial_x^3\theta_{\neq}\|_{L^2} \leq C\nu^{\frac{2}{3}}E_1\|\partial_x^3\theta_{\neq}\|_{L^2},
\end{align*}
\begin{align*}
\|\partial_z^2(u^{1,1}\partial_x\theta_{\neq})\|_{L^2}
\leq& \|\partial_x \theta_{\neq} \partial_z^2u^{1,1}\|_{L^2}+2\|\partial_zu^{1,1}\partial_z\partial_x\theta_{\neq}\|_{L^2}
+\|u^{1,1}\partial_z^2\partial_x\theta_{\neq}\|_{L^2}\\
 \leq&  \|u^{1,1}\|_{H^2}\|\partial_x\theta_{\neq}\|_{L^2_x L^\infty_{y,z} }+2\|\partial_zu^{1,1}\|_{L^\infty_y L^2_z}\|\partial_z\partial_x\theta_{\neq}\|_{L^\infty_zL^2_{x,y}}\\
 &+\|u^{1,1}\|_{L^\infty} \|\partial_x\partial_z^2\theta_{\neq}\|_{L^2}\\
\leq & C\|u^{1,1}\|_{H^2}\|(\partial_x^2+\partial_z^2)\nabla \theta_{\neq}\|_{L^2}\leq C\nu^{\frac{2}{3}}E_1\|(\partial_x^2+\partial_z^2)\nabla \theta_{\neq}\|_{L^2},
\end{align*}
which gives
\begin{align*}
&\|{\rm e}^{3\epsilon\nu^{\frac{1}{3}}t}u^{1,1}\partial_x^3\theta_{\neq}\|_{L^2L^2}
+\|{\rm e}^{2\epsilon\nu^{\frac{1}{3}}t} \partial_z^2(u^{1,1}\partial_x\theta_{\neq})\|_{L^2L^2}\\
\leq& C\nu^{\frac{2}{3}}E_1\left(\|{\rm e}^{3\epsilon\nu^{\frac{1}{3}}t}\partial_x^3\theta_{\neq}\|_{L^2L^2}+
\|{\rm e}^{2\epsilon\nu^{\frac{1}{3}}t}(\partial_x^2+\partial_z^2)\nabla \theta_{\neq}\|_{L^2L^2}\right)\\
\leq& C\nu^{\frac{1}{6}}E_1E_{4,1}.
\end{align*}
Thus \eqref{lem:u-tem-01} follows.

For $j\in\{2,3\}$, using Lemma \ref{lemu0-2} to give that
\begin{align*}
\|\bar{u}^j\partial_j\partial_x^2\theta_{\neq}\|_{L^2}
\leq& \|\bar{u}^j\|_{L^\infty} \|\partial_x^2\partial_j\theta_{\neq}\|_{L^2}
\leq CE_2\|\partial_x^2\partial_j\theta_{\neq}\|_{L^2},\\
\|\partial_z^2(\bar{u}^j\theta_{\neq})\|_{L^2}
\leq& \|\partial_z^2\bar{u}^j\|_{L^2} \|\theta_{\neq}||_{L^\infty_{y,z}L^2_x}
+\|\partial_z\bar{u}^j\|_{L^\infty_yL^2_z}\|\partial_z\theta_{\neq}\|_{L^\infty_zL^2_{x,y}}\\
&
+\|\bar{u}^j\|_{L^\infty} \|\partial_z^2\theta_{\neq}\|_{L^2}\\
\leq& C\left(\|\Delta \bar{u}^2\|_{L^2}+\|\partial_z\bar{u}^3\|_{H^1}+\|\bar{u}^3\|_{H^1}\right)
\|(\partial_x^2+\partial_z^2)\nabla\theta_{\neq}\|_{L^2}\\
\leq& CE_2\|(\partial_x^2+\partial_z^2)\nabla\theta_{\neq}\|_{L^2},
\end{align*}
which yields that
\begin{align*}
&\|{\rm e}^{3\epsilon\nu^{\frac{1}{3}}t}\bar{u}^j\partial_j\partial_x^2\theta_{\neq}\|_{L^2L^2}
+\|{\rm e}^{2\epsilon\nu^{\frac{1}{3}}t}\partial_z^2(\bar{u}^j\theta_{\neq})\|_{L^2L^2}\\
&\leq CE_2\left(\|{\rm e}^{3\epsilon\nu^{\frac{1}{3}}t}\partial_x^2\nabla\theta_{\neq}\|_{L^2L^2}+
\|{\rm e}^{2\epsilon\nu^{\frac{1}{3}}t}(\partial_x^2+\partial_z^2)\nabla \theta_{\neq}\|_{L^2L^2}\right)\\
& \leq C\nu^{-\frac{1}{2} }E_2E_{4,1}.
\end{align*}
Therefore \eqref{lem:u-tem-02} follows.

Meanwhile, for $j\in \{2,3\}$, we have
\begin{align*}
\|{\rm e}^{3\epsilon\nu^{\frac{1}{3}}t} \partial_x^2u_{\neq}^j\bar{\theta}\|_{L^2L^2}
\leq& \|{\rm e}^{3\epsilon\nu^{\frac{1}{3}}t}\partial_x^2u_{\neq}^j\|_{L^2L^2}\|\bar{\theta}\|_{L^\infty L^\infty}\\
\leq& \|{\rm e}^{3\epsilon\nu^{\frac{1}{3}}t}\partial_x^2u_{\neq}^j\|_{L^2L^2}
\left(\|\bar{\theta}\|_{L^\infty L^2}+\|\nabla \bar{\theta}\|_{L^\infty L^2}+\|\partial_z\partial_y\bar{\theta}\|_{L^\infty L^2}\right)\\
\leq& C\nu^{-\frac{1}{6}}E_5E_{4,0},
\end{align*}
which gives \eqref{lem:u-tem-03}.

Indeed, for $j\in \{2,3\}$, it holds that
\begin{align*}
\|\partial_z^2(u_{\neq}^j\bar{\theta})\|_{L^2} \leq&
\|\partial_z^2u^j_{\neq}\|_{L^2} \|\bar{\theta}\|_{L^\infty}+\|\partial_z u^j_{\neq}\|_{L^2_{x,y} L^\infty_z }\|\bar{\theta}\|_{L^\infty_y L^2_z }+\|u^j_{\neq}\|_{L^2_{x,y}L^\infty_z}\|\partial_z^2\bar{\theta}\|_{ L^\infty_y L^2_z}\\
\leq& \left(\|(\partial_x,\partial_z)\nabla u^2_{\neq}\|_{L^2}+\|(\partial_x^2+\partial_z^2)u^3_{\neq}\|_{L^2}\right)\\
&\cdot \left(\|\partial_z^2\partial_y\bar{\theta}\|_{L^2}+\|\partial_y\partial_z\bar{\theta}\|_{L^2}+\|\partial_y\bar{\theta}\|_{L^2}\right)\\
\leq& CE_{3,0}\left(\|\partial_z^2\partial_y\bar{\theta}\|_{L^2}+\|\partial_y\partial_z\bar{\theta}\|_{L^2}+\|\partial_y\bar{\theta}\|_{L^2}\right) {\rm e}^{-2 \epsilon \nu^{\frac{1}{3}} t}
\end{align*}
and
\begin{align*}
\|\partial_z(u_{\neq}^2\partial_y\bar{\theta})\|_{L^2} \leq& \|u_{\neq}^2\|_{L^\infty_{y,z} L^2_x}\|\partial_z\partial_y\bar{\theta}\|_{L^2}+\|\partial_zu_{\neq}^2\|_{L^\infty_yL_{x,z}^2}\|\partial_y\bar{\theta}\|_{L^\infty_zL^2_y}\\
\leq&\|(\partial_x,\partial_z)\nabla u^2_{\neq}\|_{L^2}\left(\|\partial_z\partial_y\bar{\theta}\|_{L^2}+\|\partial_y\bar{\theta}\|_{L^2}\right)\\
\leq& CE_{3,0}\left(\|\partial_z\partial_y\bar{\theta}\|_{L^2}+\|\partial_y\bar{\theta}\|_{L^2}\right){\rm e}^{-2 \epsilon \nu^{\frac{1}{3}} t},
\end{align*}
which implies that
\begin{align*}
&\|{\rm e}^{2\epsilon\nu^{\frac{1}{3}}t}\partial_z^2(u_{\neq}^j\bar{\theta})\|_{L^2L^2}
+\|{\rm e}^{2\epsilon\nu^{\frac{1}{3}}t}\partial_z(u_{\neq}^j\partial_j\bar{\theta})\|_{L^2L^2}
\\
&\leq CE_{3,0}\left(\|\partial_z^2\partial_y\bar{\theta}\|_{L^2L^2}+\|\partial_y\partial_z\bar{\theta}\|_{L^2L^2}+\|\partial_y\bar{\theta}\|_{L^2L^2}\right)\\
&
\leq C\nu^{-\frac{1}{2}}E_{3,0}E_{4,0}.
\end{align*}
This gives \eqref{lem:u-tem-04} and \eqref{lem:u-tem-05}.

Thanks to Lemma \ref{lemu1-0}, one has
\begin{align*}
&\|\partial_z^2(u^{1,0}\partial_x\theta_{\neq})\|_{L^2} \lesssim
\|\partial_z^2 u^{1,0}\partial_x\theta_{\neq}\|_{L^2}+\|\partial_zu^{1,0}\partial_z\partial_x\theta_{\neq}\|_{L^2}
+\|u^{1,0}\partial_z^2\partial_x\theta_{\neq}\|_{L^2}\\
\lesssim& \|\partial_z^2 u^{1,0}\|_{L^2}\|\partial_x\theta_{\neq}\|_{L^2_x L^\infty_{y,z}} +\|\partial_zu^{1,0}\|_{L^2_zL^\infty_y}\|\partial_x\partial_z\theta_{\neq}\|_{L^2_{x,y} L^\infty_z}
+\|u^{1,0}\|_{H^2}\|\partial_z^2\partial_x\theta_{\neq}\|_{L^2}\\
\leq& C\|u^{1,0}\|_{H^2}\left(\|\partial_x\partial_y\partial_z\theta_{\neq}\|_{L^2}+\|\nabla\partial_x^2\theta_{\neq}\|_{L^2}
+\|\partial_x\partial_z^2\theta_{\neq}\|_{L^2}\right),\\
\leq & C(\nu^{\frac{2}{3}}+\nu t)E_1\left(\|\partial_x\partial_y\partial_z\theta_{\neq}\|_{L^2}+\|\nabla\partial_x^2\theta_{\neq}\|_{L^2}
+\|\partial_x\partial_z^2\theta_{\neq}\|_{L^2}\right).
\end{align*}
Note that for $(x,z) \in \mathbb{T}^2$, it holds that
$$\|\partial_x\partial_z\nabla f\|_{L^2}^2=\langle \partial_x^2\nabla f,\partial_z^2\nabla f \rangle\leq \|\partial_x^2\nabla f\|_{L^2}\|\partial_z^2\nabla f\|_{L^2},$$
one has
\begin{align*}
&\|{\rm e}^{2\epsilon \nu^{\frac{1}{3}}t}\partial_z^2(u^{1,0}\partial_x\theta_{\neq})\|_{L^2L^2}^2\\
&\leq
CE_1^2\nu^{\frac{4}{3}}\|(1+\nu^{1/3}t)^2 {\rm e}^{-\epsilon \nu^{\frac{1}{3}}t}\|_{L^\infty_t}\|{\rm e}^{2\epsilon\nu^{\frac{1}{3}}t}\partial_z^2\nabla \theta_{\neq}\|_{L^2L^2}\|{\rm e}^{3\epsilon\nu^{\frac{1}{3}}t}\partial_x^2\nabla \theta_{\neq}\|_{L^2L^2}\\
&\leq C\nu^{\frac{1}{3}}E_1^2E_{4,1}^2,
\end{align*}
which gives \eqref{lem:u-tem-06}.
\end{proof}
\begin{lemma}\label{lemuu-10}
It holds that for $j\in \{2,3\}$,
\begin{align*}
&\|{\rm e}^{2\epsilon \nu^{\frac{1}{3}}t}(\partial_x,\partial_z)(\bar{u}^1\partial_xu^3_{\neq})\|_{L^2L^2}^2
+\|{\rm e}^{2\epsilon \nu^{\frac{1}{3}}t}(\partial_x,\partial_z)(\bar{u}^1\partial_xu^2_{\neq})\|_{L^2L^2}^2\\
&\qquad\qquad\qquad\qquad\qquad+\|{\rm e}^{2\epsilon \nu^{\frac{1}{3}}t}\partial_x((u^2_{\neq}\partial_y+u^3_{\neq}\partial_z)\bar{u}^1)\|_{L^2L^2}^2
\leq C\nu E_1^2E_3E_5,\\
&\|{\rm e}^{2\epsilon \nu^{\frac{1}{3}}t}\partial_x(\bar{u}^1\partial_xu^1_{\neq})\|_{L^2L^2}^2\leq
C\nu\left(E_1^2E_3^{\frac{3}{2}}E_5^{\frac{1}{2}}+E_1^2E_3E_5\right),\\
&\|{\rm e}^{2\epsilon \nu^{\frac{1}{3}}t}(\partial_x,\partial_z)(\bar{u}^j\nabla u_{\neq})\|_{L^2L^2}
+\|{\rm e}^{2\epsilon \nu^{\frac{1}{3}}t}(\partial_x,\partial_z)(u_{\neq}\cdot\nabla \bar{u}^j)\|_{L^2L^2}
\leq C\nu^{-\frac{1}{2}}E_2E_3,\\
&\|{\rm e}^{2\epsilon \nu^{\frac{1}{3}}t}\partial_z(\bar{u}^1\partial_xu^1_{\neq})\|_{L^2L^2}^2
+\|{\rm e}^{2\epsilon \nu^{\frac{1}{3}}t}\partial_z((u_{\neq}^2\partial_y+u^3_{\neq}\partial_z)\bar{u}^1)\|_{L^2L^2}^2\\
&\qquad \qquad \qquad \qquad \qquad \qquad \qquad \qquad \qquad \leq C\nu^{\frac{1}{3}}E_1^2E_3(E_5+E_3^{\frac{3}{4}}E_5^{\frac{1}{4}}).
\end{align*}
\end{lemma}
\begin{proof}
See Lemma 11.13-11.16 in \cite{chenwz}.
\end{proof}

\section{Estimate of zero modes}
\subsection{Estimate of $E_1$}
\begin{proposition}\label{proE1}
It holds that
\begin{align*}
&E_{1,0}\leq C\nu^{-1}\left(\|\bar{u}(0)\|_{H^2}+E_2+E_2E_{1,0}\right),\\
&E_{1,1}\leq C\left(\|\bar{u}(0)\|_{H^2}+\nu^{-1}E_2E_{1,1}+\nu^{-\frac{4}{3}}E_3^2\right).
\end{align*}
\end{proposition}
\begin{proof}
Due to the equations for $\overline{u}^{1,j},j=0,1$ are the same as in \cite{chenwz}, then the proof can be achieved by following the same arguments as in Proposition 12.1 of \cite{chenwz}. We omit the details here.
\end{proof}

\subsection{Estimate of $E_{4,0}$}
In this subsection, we assume $0<\nu\leq \nu_0$, $\nu_0^{2/3}\leq 4\epsilon$, and then ${\rm e}^{\nu t}\leq {\rm e}^{4\epsilon \nu^{\frac{1}{3}}t}$.
\begin{proposition}\label{proE40}
It holds that
\begin{align*}
&E_{4,0}\leq C(\nu^{-1}E_2E_{4,0}+\nu^{-1}E_{3,0}E_{4,1}+\|\overline{\theta}(0)\|_{H^2}),\\
&\|{\rm e}^{\nu t}\bar{\theta}\|_{L^\infty L^2}^2+\nu\|{\rm e}^{\nu t}\nabla\bar{\theta}\|_{L^2L^2}^2 \leq C (\|\overline{\theta} (0)\|_{L^2}^2+ \nu^{-2}E_{3,0}^2E_{4,1}^2).
\end{align*}
\end{proposition}
\begin{proof}
Recall that $\bar{\theta}$ enjoys
\begin{align}\label{eqthe0}
\begin{cases}
\partial_t\bar{\theta}-\nu\Delta \bar{\theta}=-\overline{u\cdot\nabla \theta}=-\bar{u}\cdot\nabla\bar{\theta}-\overline{u_{\neq}\cdot\nabla\theta_{\neq}}
,\\
\bar{\theta}|_{y=\pm1}=0, \bar{\theta}|_{t=0}=\bar{\theta}_{in}=: \bar{\theta} (0),
\end{cases}
\end{align}
the basic energy estimate yields that
\begin{align*}
\frac{1}{2}\frac{d}{dt}\|\bar{\theta}\|_{L^2}^2+\nu\|\nabla\bar{\theta}\|_{L^2}^2= \langle-\bar{u}\cdot\nabla\bar{\theta}-\overline{u_{\neq}\cdot\nabla\theta_{\neq}},\overline{\theta}\rangle
\leq  \frac{1}{4}\nu \|\nabla \bar{\theta}\|_{L^2}^2+C\nu^{-1}\|u_{\neq}\theta_{\neq}\|_{L^2}^2,
\end{align*}
which gives that
\begin{align}
\frac{d}{dt}\|\bar{\theta}\|_{L^2}^2+\nu\|\nabla\bar{\theta}\|_{L^2}^2\leq C\nu^{-1}\|u_{\neq}\theta_{\neq}\|_{L^2}^2\label{eqthe0-1}
\end{align}
and
\begin{align*}
\frac{d}{dt}({\rm e}^{2\nu t}\|\bar{\theta}\|_{L^2}^2)+\nu {\rm e}^{2\nu t}\|\nabla\bar{\theta}\|_{L^2}^2\leq C\nu^{-1} {\rm e}^{2\nu t}\|u_{\neq}\theta_{\neq}\|_{L^2}^2+2\nu {\rm e}^{2\nu t}\|\bar{\theta}\|_{L^2}^2.
\end{align*}
Note that $\overline{\theta}|_{y=\pm 1}=0,$ it follows from the fact $\|\nabla f\|_{L^2}^2\geq \left(\frac{\pi}{2}\right)^2\|f\|_{L^2}^2$ that
\begin{align*}
\|{\rm e}^{\nu t}\bar{\theta}\|_{L^\infty L^2}^2+\nu\|{\rm e}^{\nu t}\nabla \bar{\theta}\|_{L^2L^2}^2\leq & \|\overline{\theta}(0)\|_{L^2}^2+C \nu^{-1} \|{\rm e}^{ \nu t}u_{\neq}\theta_{\neq}\|_{L^2L^2}^2\\
\leq& \|\overline{\theta}(0)\|_{L^2}^2+C\nu^{-1}  \|{\rm e}^{4\epsilon \nu^{\frac{1}{3}}t}u_{\neq}\theta_{\neq}\|_{L^2L^2}^2\\
 \lesssim&  \|\overline{\theta} (0)\|_{L^2}^2+ \nu^{-2}E_{3,0}^2E_{4,1}^2,
\end{align*}
where we have used Lemma \ref{lemuthe-11} in the last inequality.

Multiplying \eqref{eqthe0} by $-\Delta \bar{\theta}$ and $\partial_z^2\Delta \bar{\theta}$, respectively, it follows from the integration by part that
\begin{align}
\frac{1}{2}\frac{d}{dt}\|\nabla\bar{\theta}\|_{L^2}^2+\nu\|\Delta \bar{\theta}\|_{L^2}^2\leq \left(\|\bar{u}\cdot\nabla\bar{\theta}\|_{L^2}+\|\overline{u_{\neq}\cdot\nabla\theta_{\neq}}\|_{L^2}\right)\|\Delta \overline{\theta}\|_{L^2} \label{eqthe0-2}
\end{align}
and
\begin{align}
\frac{1}{2}\frac{d}{dt}\|\partial_z\nabla \bar{\theta}\|_{L^2}^2+\nu\|\partial_z\Delta \bar{\theta}\|_{L^2}^2\leq \left(\|\partial_z(\bar{u}\cdot\nabla\bar{\theta})\|_{L^2}+\|\partial_z(\overline{u_{\neq}\cdot\nabla\theta_{\neq}})\|_{L^2}\right)\|\partial_z\Delta \bar{\theta}\|_{L^2}.\label{eqthe0-3}
\end{align}
Summing \eqref{eqthe0-1}-\eqref{eqthe0-3}, and using Lemma \ref{uthe-00} and Lemma \ref{lemuthe-11}, we obtain that
\begin{align*}
E_{4,0}^2\leq& C\|\overline{\theta}(0)\|_{H^2}+ C\nu^{-1}\left(\|u_{\neq}\theta_{\neq}\|_{L^2L^2}^2+\|(\partial_z,1)(\bar{u}\cdot\nabla\bar{\theta},\overline{u_{\neq}\cdot\nabla\theta_{\neq}})\|_{L^2L^2}^2\right)\\ \leq & C\left(\|\bar{\theta}(0)\|_{H^2}^2+\nu^{-2}E_2^2E_{4,0}^2+\nu^{-2}E_{3,0}^2E_{4,1}^2\right).
\end{align*}

The proof is completed.
\end{proof}

\subsection{Estimate of $E_2$}
\begin{proposition}\label{proE2}
It holds that
\begin{align*}
E_2\leq C(1+\nu^{-1}E_2)(\|u(0)\|_{H^2}+\nu^{-1}E_3^2+\nu^{-1}\|\bar{\theta}(0)\|_2+ \nu^{-2}E_{3,0}E_{4,1}).
\end{align*}
\end{proposition}
To prove the Proposition \ref{proE2}, we need the following lemmas.

\begin{lemma}\label{lemE21}
It holds that
\begin{align*}
&\|{\rm e}^{\nu t}(\bar{u}^2, \bar{u}^3)\|_{L^\infty L^2}^2+\nu\|{\rm e}^{\nu t}(\nabla \bar{u}^2,\nabla\bar{u}^3)\|_{L^2L^2}^2\\
&\qquad\qquad\qquad\qquad \leq C\left(\|u(0)\|_{H^2}^2+\nu^{-2}E_3^4\right)+C\nu^{-1}\|{\rm e}^{\nu t}\nabla \bar{\theta}\|_{L^2L^2}^2,\\
&\|{\rm e}^{\nu t}\nabla(\bar{u}^2, \bar{u}^3)\|_{L^\infty L^2}^2+\nu^{-1}\|{\rm e}^{\nu t}(\partial_t \bar{u}^2,\partial_t\bar{u}^3)\|_{L^2L^2}^2\\
&\qquad\qquad\leq C\left(\|u(0)\|_{H^2}^2+\nu^{-2}E_3^4+\nu^{-1}\|{\rm e}^{\nu t}\nabla \bar{\theta}\|_{L^2L^2}^2\right)(\nu^{-2}E_2^2+1).
\end{align*}
\end{lemma}
\begin{proof}
Let us recall that $\bar{u}^j (j=2,3)$ enjoy
\begin{align}\label{equ023}
\begin{cases}
\partial_t\bar{u}^2-\nu\Delta \bar{u}^2+\overline{u\cdot \nabla u^2}+\partial_y\overline{P^{NL}}+\partial_y\overline{P^{\theta}}=\bar{\theta},\\
\partial_t\bar{u}^3-\nu\Delta \bar{u}^3+\overline{u\cdot \nabla u^3}+\partial_z\overline{P^{NL}}+\partial_z\overline{P^{\theta}}=0,\\
\bar{u}^j|_{y=\pm1}=0, \partial_y\bar{u}^2|_{y=\pm1}=0, \bar{u}^j|_{t=0}=\bar{u}^j(0).
\end{cases}
\end{align}
As $\partial_y\bar{u}^2+\partial_z\bar{u}^3=0$ and $(\overline{u}^2,\partial_y \overline{u}^2)|_{y=\pm 1}=(0,0)$, it holds that $\int_{\mathbb{T}}\bar{u}^2dz=0$.
In addition, the standard energy argument gives that
\begin{align*}
&\frac{1}{2}\frac{d}{dt}\left(\|\bar{u}^2\|_{L^2}^2+\|\bar{u}^3\|_{L^2}^2\right)+\nu\left(\|\nabla \bar{u}^2\|_{L^2}^2+\|\nabla\bar{u}^3\|_{L^2}^2\right)\\
=&\langle \bar{\theta},\bar{u}^2\rangle-\langle \overline{u_{\neq}\cdot\nabla u_{\neq}^2},\bar{u}^2\rangle-
\langle \overline{u_{\neq}\cdot\nabla u_{\neq}^3},\bar{u}^3\rangle\\
\leq& \||u_{\neq}|^2\|_{L^2}\left(\|\nabla \bar{u}^2\|_{L^2}+\|\nabla \bar{u}^3\|_{L^2}\right)+\|\partial_z\bar{\theta}\|_{L^2} \|\partial_z\bar{u}^2\|_{L^2}\\
\leq& \frac{1}{2}\nu \left(\|\nabla \bar{u}^2\|_{L^2}^2+\|\nabla \bar{u}^3\|_{L^2}^2\right)+C\nu^{-1}\left(\||u_{\neq}|^2\|_{L^2}^2+\|\partial_z\bar{\theta}\|_{L^2}^2\right),
\end{align*}
where we use the facts
\begin{align*}
&\int_{\Omega} \theta_{0,z}(t,y) \overline{u}^2 (t,y,z) dx dy dz=0,\\
& \langle \bar{\theta},\bar{u}^2\rangle
=\langle \bar{\theta}-\theta_{0,z},\bar{u}^2\rangle
\leq \|\bar{\theta}-\theta_{0,z}\|_{L^2} \|\bar{u}^2\|_{L^2}\leq C\|\partial_z\bar{\theta}\|_{L^2}\|\partial_z \bar{u}^2\|_{L^2},
\end{align*}
here $\theta_{0,z}:=\int_{\mathbb{T}}\bar{\theta}dz$.

Using $\|\nabla \bar{u}^j\|_{L^2}^2\geq \left(\pi/2\right)^2\|\bar{u}^j\|_{L^2}^2$, we have
\begin{align*}
\frac{d}{dt}\left(\|\bar{u}^2\|_{L^2}^2+\|\bar{u}^3\|_{L^2}^2\right)+\nu\left(\|\nabla \bar{u}^2\|_{L^2}^2+\|\nabla\bar{u}^3\|_{L^2}^2\right)
\leq& C\nu^{-1}\left(\||u_{\neq}|^2\|_{L^2}^2+\|\partial_z\bar{\theta}\|_{L^2}^2\right)
\end{align*}
and
\begin{align*}
&\frac{d}{dt}\left({\rm e}^{2\nu t}\|\bar{u}^2\|_{L^2}^2+{\rm e}^{2\nu t}\|\bar{u}^3\|_{L^2}^2\right)+\nu\left({\rm e}^{2\nu t}\|\nabla \bar{u}^2\|_{L^2}^2+{\rm e}^{2\nu t}\|\nabla\bar{u}^3\|_{L^2}^2\right)\\
\leq& C\nu^{-1}\left( {\rm e}^{2\nu t}\||u_{\neq}|^2\|_{L^2}^2+{\rm e}^{2\nu t}\|\partial_z\bar{\theta}\|_{L^2}^2\right)+\nu {\rm e}^{2\nu t}\left(\|\bar{u}^2\|_{L^2}^2+\|\bar{u}^3\|_{L^2}^2\right),
\end{align*}
which, together Lemma \ref{lemuu-11} and ${\rm e}^{\nu t}\leq {\rm e}^{4\epsilon \nu^{1/3}t}$, gives that
\begin{align*}
\|{\rm e}^{\nu t}(\bar{u}^2, \bar{u}^3)\|_{L^\infty L^2}^2&+\nu\|{\rm e}^{\nu t}(\nabla \bar{u}^2,\nabla\bar{u}^3)\|_{L^2L^2}^2\\
&\leq C\left(\|u(0)\|_{L^2}^2+\nu^{-2}E_3^4\right)+C\nu^{-1}\|{\rm e}^{\nu t}\nabla \bar{\theta}\|_{L^2L^2}^2.
\end{align*}
The first estimate follows.

 Multiplying $\eqref{equ023}_1,\eqref{equ023}_2$ by $\partial_t\bar{u}^2$ and $\partial_t\bar{u}^3$, respectively, it follows from the integration by part that
\begin{align*}
&\frac{\nu}{2}\frac{d}{dt}\left(\|\nabla \bar{u}^2\|_{L^2}^2+\|\nabla\bar{u}^3\|_{L^2}^2\right)+\|\partial_t(\bar{u}^2,\bar{u}^3)\|_{L^2}^2\\
\leq& \|(\bar{u}^2\partial_y+\bar{u}^3\partial_z)\bar{u}^2\|_{L^2}\|\partial_t\bar{u}^2\|_{L^2}
+\|(\bar{u}^2\partial_y+\bar{u}^3\partial_z)\bar{u}^3\|_{L^2} \|\partial_t\bar{u}^3\|_{L^2}+\|\partial_z\bar{\theta}\|_{L^2} \|\partial_t\bar{u}^2\|_{L^2}\\
&+\|u_{\neq}\cdot\nabla u_{\neq}^2\|_{L^2} \|\partial_t\bar{u}^2\|_{L^2}
+\|u_{\neq}\cdot\nabla u_{\neq}^3\|_{L^2} \|\partial_t\bar{u}^3\|_{L^2}.
\end{align*}
By Lemmas \ref{lemu0-2} and Lemma \ref{lemuu-11}, for $j\in \{2,3\}$, we have
\begin{align*}
&\|{\rm e}^{\nu t}(\bar{u}^2\partial_y+\bar{u}^3\partial_z)\bar{u}^j\|_{L^2}
\leq \|(\bar{u}^2,\bar{u}^3)\|_{L^\infty} \|{\rm e}^{\nu t}\nabla \bar{u}^j\|_{L^2} \leq CE_2\|{\rm e}^{\nu t}\nabla \bar{u}^j\|_{L^2},\\
&\|{\rm e}^{\nu t}(u_{\neq}\cdot\nabla u_{\neq})\|_{L^2L^2}^2\leq
\|{\rm e}^{4\epsilon \nu^{\frac{1}{3}} t}(u_{\neq}\cdot\nabla u_{\neq})\|_{L^2L^2}^2\leq C\nu^{-1}E_3^4.
\end{align*}
Therefore, we obtain
\begin{align*}
&\|{\rm e}^{\nu t}\nabla(\bar{u}^2, \bar{u}^3)\|_{L^\infty L^2}^2+\nu^{-1}\|{\rm e}^{\nu t}(\partial_t \bar{u}^2,\partial_t\bar{u}^3)\|_{L^2L^2}^2\\
\leq &\|\bar{u}(0)\|_{H^1}^2+C\nu^{-1}\left(\| {\rm e}^{\nu t}(\bar{u}^2\partial_y+\bar{u}^3\partial_z)\bar{u}^j\|_{L^2L^2}^2+\| {\rm e}^{\nu t} \partial_z\bar{\theta}\|_{L^2L^2}^2
+\|{\rm e}^{\nu t}(u_{\neq}\cdot\nabla u_{\neq})\|_{L^2L^2}^2\right)\\
\leq& \|u(0)\|_{H^2}^2+C\nu^{-1}E_2^2\|{\rm e}^{\nu t}\nabla \bar{u}^j\|_{L^2L^2}^2+C\nu^{-1}\| {\rm e}^{\nu t} \partial_z\bar{\theta}\|_{L^2L^2}^2
+C\nu^{-2}E_3^4\\
\leq& C(1+\nu^{-2}E_2^2)\left(\|u(0)\|_{H^2}^2+\nu^{-2}E_3^4+\nu^{-1}\|{\rm e}^{\nu t}\nabla \bar{\theta}\|_{L^2L^2}^2\right),
\end{align*}
which completes the proof.
\end{proof}

\begin{lemma}\label{lemE22}
It holds that
\begin{align*}
&\|{\rm e}^{\nu t}\Delta \bar{u}^2\|_{L^\infty L^2}^2+\nu^{-1}\|{\rm e}^{\nu t}\nabla \partial_t\bar{u}^2\|_{L^2 L^2}^2\\
\leq& C \left(\|u(0)\|_{H^2}^2+\nu^{-2}E_3^4+\nu^{-1}\|{\rm e}^{\nu t}\nabla \bar{\theta}\|_{L^2L^2}^2\right)(\nu^{-2}E_2^2+1).
\end{align*}
\end{lemma}
\begin{proof}
Recall that $\Delta \bar{u}^2$ satisfies
\begin{align*}
\begin{cases}
(\partial_t-\nu \Delta)\Delta \bar{u}^2+\Delta \partial_y \bar{P}^{NL}+\Delta(\overline{u\cdot \nabla u^2})=\partial_z^2\bar{\theta},\\
\nabla \bar{u}^2|_{y=\pm1}=0,
\end{cases}
\end{align*}
taking $L^2$ inner product with $-2\bar{u}^2$, one gets
\begin{align*}
&\frac{d}{dt}\|\nabla \bar{u}^2\|_{L^2}^2+2\nu\|\Delta \bar{u}^2\|_{L^2}^2\\
=&-2\langle \Delta\bar{P}^{NL}, \partial_y\bar{u}^2\rangle+2\langle \bar{u}^2\partial_y\bar{u}^2+\bar{u}^3\partial_z\bar{u}^2+\overline{u_{\neq}\cdot\nabla u_{\neq}^2},\Delta \bar{u}^2\rangle-2 \langle \partial_z \bar{\theta},\partial_z \bar{u}^2\rangle\\
\leq &2 \bigg(\|\Delta\bar{P}^{NL}\|_{L^2} \| \partial_y\bar{u}^2 \|_{L^2}+\| \bar{u}^2\partial_y\bar{u}^2+\bar{u}^3\partial_z\bar{u}^2\|_{L^2} \|\Delta \bar{u}^2\|_{L^2}\\
&+\| \overline{u_{\neq}\cdot\nabla u_{\neq}^2}\|_{L^2} \|\Delta \bar{u}^2\|_{L^2}+\|\partial_z \bar{\theta}\|_{L^2} \|\partial_z \bar{u}^2\|_{L^2}\bigg) ,
\end{align*}
which implies by the Young's inequality that
\begin{align*}
&\frac{d}{dt}\|\nabla \bar{u}^2\|_{L^2}^2+\nu\|\Delta \bar{u}^2\|_{L^2}^2\\
\leq&  C \nu^{-1} \bigg(\|\Delta\bar{P}^{NL}\|_{L^2}^2+\| \bar{u}^2\partial_y\bar{u}^2+\bar{u}^3\partial_z\bar{u}^2\|_{L^2}^2+\| \overline{u_{\neq}\cdot\nabla u_{\neq}^2}\|_{L^2}^2+\|\partial_z \bar{\theta}\|_{L^2}^2\bigg).
\end{align*}
Moreover, one has
\begin{align*}
\nu\|{\rm e}^{\nu t}\Delta \bar{u}^2\|_{L^2L^2}^2
\leq& C\|u(0)\|_{H^2}^2+C\nu^{-1}\left(\|{\rm e}^{\nu t}\Delta \bar{P}^{NL}\|_{L^2L^2}^2+\|{\rm e}^{\nu t}(\bar{u}^2\partial_y+\bar{u}^3\partial_z)\bar{u}^2\|_{L^2L^2}^2
\right.\\
&\left.+\|{\rm e}^{\nu t}(u_{\neq}\cdot\nabla u_{\neq})\|_{L^2L^2}^2+\|{\rm e}^{\nu t}\partial_z\bar{\theta}\|_{L^2L^2}^2\right).
\end{align*}
Thanks to Lemma \ref{lemE21} and Lemma \ref{lemuu-11}, for $i,j \in \{2,3\}$, one has
\begin{align}
&\|{\rm e}^{\nu t}\nabla(\bar{u}^i\partial_i\bar{u}^j)\|_{L^2L^2}^2
\leq \|{\rm e}^{\nu t}\bar{u}^i\|_{L^\infty H^1}^2\|(\partial_i\bar{u}^j,\partial_z\partial_i\bar{u}^j)\|_{L^2H^1}^2\nonumber\\
&\quad\quad\leq C\|{\rm e}^{\nu t}(\nabla\bar{u}^2,\nabla\bar{u}^3)\|_{L^\infty L^2}^2\|(\Delta\bar{u}^2,\Delta \bar{u}^3,\nabla\Delta \bar{u}^2)\|_{L^2L^2}^2\nonumber\\
&\quad\quad\leq C\nu^{-1}E_2^2\left(\|u(0)\|_{H^2}^2+\nu^{-2}E_3^4+\nu^{-1}\|{\rm e}^{\nu t}\nabla \bar{\theta}\|_{L^2L^2}^2\right)(\nu^{-2}E_2^2+1),\label{eqE2-101}\\
&\nu^{-1}\|{\rm e}^{\nu t}\partial_k(u_{\neq}\cdot\nabla u^k_{\neq})\|_{L^2L^2}^2+\nu^{-1}\|{\rm e}^{\nu t}\nabla(u_{\neq}\cdot\nabla u^2_{\neq})\|_{L^2L^2}^2\leq C\nu^{-2}E_3^4.\label{eqE2-102}
\end{align}
Note that $\Delta \bar{P}^{NL}=-\overline{\partial_iu^j\partial_ju^i}=-\overline{\partial_i(\bar{u}^j\partial_j\bar{u}^i)}
-\overline{\partial_i(u^j_{\neq}\partial_j u^i_{\neq})}$, we have
\begin{align*}
\nu^{-1}\|{\rm e}^{\nu t}\Delta \bar{P}^{NL}\|_{L^2L^2}^2\leq & \nu^{-1}\|{\rm e}^{\nu t}\nabla(\bar{u}^i\partial_i\bar{u}^j)\|_{L^2L^2}^2+\nu^{-1}\|{\rm e}^{\nu t}\partial_k(u_{\neq}\cdot\nabla u^k_{\neq})\|_{L^2L^2}^2\\
\leq& C\left(\|u(0)\|_{H^2}^2+\nu^{-2}E_3^4+\nu^{-1}\|{\rm e}^{\nu t}\nabla \bar{\theta}\|_{L^2L^2}^2\right)(\nu^{-2}E_2^2+1)^2\\
\leq&C\left(\|u(0)\|_{H^2}^2+\nu^{-2}E_3^4+\nu^{-1}\|{\rm e}^{\nu t}\nabla \bar{\theta}\|_{L^2L^2}^2\right)(\nu^{-2}E_2^2+1),
\end{align*}
where we have used $E_2\leq \varepsilon_0 \nu$.
Therefore, we obtain
\begin{align}
\nu\|{\rm e}^{\nu t}\Delta \bar{u}^2\|_{L^2L^2}^2
\leq C\left(\|u(0)\|_{H^2}^2+\nu^{-2}E_3^4+\nu^{-1}\|{\rm e}^{\nu t}\nabla \bar{\theta}\|_{L^2L^2}^2\right)(\nu^{-2}E_2^2+1).\label{eqnu1-1}
\end{align}

Next, we take the $L^2$ inner product with $-2\partial_t\bar{u}^2$ to obtain
\begin{align*}
2\|\partial_t\nabla \bar{u}^2\|_{L^2}^2+\nu\frac{d}{dt}\|\Delta \bar{u}^2\|_{L^2}^2+2\langle \Delta \bar{P}^{NL},\partial_t\partial_y \bar{u}^2\rangle+\langle\nabla\overline{(u\cdot\nabla u^2)},\partial_t\nabla\bar{u}^2\rangle=\langle\partial_z\bar{\theta},\partial_z\partial_t\bar{u}^2\rangle,
\end{align*}
from which, one has
\begin{align*}
&\frac{d}{dt}\left({\rm e}^{2\nu t}\|\Delta \bar{u}^2\|_{L^2}^2\right)+\nu^{-1}{\rm e}^{2\nu t}\|\partial_t\nabla \bar{u}^2\|_{L^2}^2\\
&\leq C\nu^{-1}\left(\|{\rm e}^{\nu t}\Delta \bar{P}^{NL}\|_{L^2}^2+\|{\rm e}^{\nu t}\nabla\overline{(u\cdot\nabla u^2)}\|_{L^2}^2+\|{\rm e}^{\nu t}\partial_z\bar{\theta}\|_{L^2}^2\right)+C\nu {\rm e}^{2\nu t}\|\Delta \bar{u}^2\|_{L^2}^2\\
&\leq C\nu^{-1}\left(\|{\rm e}^{\nu t}\Delta \bar{P}^{NL}\|_{L^2}^2+\|{\rm e}^{\nu t}\nabla(\bar{u}\cdot\nabla \bar{u}^2)\|_{L^2}^2+\|{\rm e}^{\nu t}\nabla\overline{(u_{\neq}\cdot\nabla u_{\neq}^2)}\|_{L^2}^2+\|{\rm e}^{\nu t}\partial_z\bar{\theta}\|_{L^2}^2\right)\\
&\quad+C\nu {\rm e}^{2\nu t}\|\Delta \bar{u}^2\|_{L^2}^2,
\end{align*}
which gives that
\begin{align*}
&\|{\rm e}^{\nu t} \Delta \bar{u}^2\|_{L^\infty L^2}^2+\nu^{-1} \|{\rm e}^{\nu t} \partial_t\nabla \bar{u}^2\|_{L^2 L^2}^2\\
\lesssim& \| \Delta \bar{u}^2 (0)\|_{L^2}^2+ \nu^{-1} \bigg(\|{\rm e}^{\nu t}\Delta \bar{P}^{NL}\|_{L^2 L^2}^2+\|{\rm e}^{\nu t}\nabla(\bar{u}\cdot\nabla \bar{u}^2)\|_{L^2 L^2}^2\\
&+\|{\rm e}^{\nu t}\nabla\overline{(u_{\neq}\cdot\nabla u_{\neq}^2)}\|_{L^2 L^2}^2+\|{\rm e}^{\nu t}\partial_z\bar{\theta}\|_{L^2 L^2}^2+\nu \|{\rm e}^{\nu t}\Delta \bar{u}^2\|_{L^2 L^2}^2\\
\lesssim& \left(\|u(0)\|_{H^2}^2+\nu^{-2}E_3^4+\nu^{-1}\|{\rm e}^{\nu t}\nabla \bar{\theta}\|_{L^2L^2}^2\right)(\nu^{-2}E_2^2+1),
\end{align*}
where we have used \eqref{eqE2-101}, \eqref{eqE2-102} and \eqref{eqnu1-1}.
\end{proof}
\begin{lemma}\label{lemE2-3}
It holds that
\begin{align*}
&\nu\|{\rm e}^{\nu t}\nabla \Delta \bar{u}^2\|_{L^2L^2}^2+\nu\|{\rm e}^{\nu t}\Delta \bar{u}^3\|_{L^2L^2}^2\\
\leq& C\left(\|u(0)\|_{H^2}^2+\nu^{-2}E_3^4+\nu^{-1}\|{\rm e}^{\nu t}\nabla \bar{\theta}\|_{L^2L^2}^2\right)(\nu^{-2}E_2^2+1).
\end{align*}
\end{lemma}
\begin{proof}
Note that $\bar{p}=\bar{P}^{NL}+\bar{P}^{\theta}$, and $\partial_z\bar{u}^2$ solves
\begin{align*}
\begin{cases}
(\partial_t-\nu \Delta )\partial_z\bar{u}^2+\partial_z\partial_y\bar{p}+\partial_z
(\overline{u\cdot \nabla u^2})=\partial_z\bar{\theta},\\
\nabla \bar{u}^2|_{y=\pm1}=0,
\end{cases}
\end{align*}
it follows from the integration by part that
\begin{align*}
&\|\partial_t\partial_z\bar{u}^2+\partial_z(\overline{u\cdot \nabla u^2})-\partial_z\bar{\theta}\|_{L^2}^2=\|\nu\Delta \partial_z\bar{u}^2-\partial_z\partial_y\bar{p}\|_{L^2}^2\\
&=\nu^2\|\Delta \partial_z\bar{u}^2\|_{L^2}^2+\|\partial_z\partial_y\bar{p}\|_{L^2}^2
-2\nu\langle \Delta \partial_z\bar{u}^2 \partial_z\partial_y\bar{p}\rangle\\
&=\nu^2\|\Delta \partial_z\bar{u}^2\|_{L^2}^2+\|\partial_z\partial_y\bar{p}\|_{L^2}^2
-2\nu\langle  \partial_z^2\partial_y\bar{u}^2 \Delta\bar{p}\rangle,
\end{align*}
which shows that
\begin{align*}
\nu^2\|\Delta \partial_z\bar{u}^2\|_{L^2}^2+\|\partial_z\partial_y\bar{p}\|_{L^2}^2\leq C\left(\|\Delta \bar{p}\|_{L^2}^2+\|\partial_t\partial_z\bar{u}^2+\partial_z(\overline{u\cdot \nabla u^2})-\partial_z\bar{\theta}\|_{L^2}^2\right).
\end{align*}

By \eqref{eqE2-101}, Lemma \ref{lemuu-11} and Lemma \ref{lemE22}, we obtain
\begin{align*}
&\nu\|{\rm e}^{\nu t}\Delta \partial_z\bar{u}^2\|_{L^2L^2}+\nu^{-1}\|{\rm e}^{\nu t}\partial_z\partial_y\bar{p}\|_{L^2L^2}^2\\
&\leq C\nu^{-1}\left(\|{\rm e}^{\nu t}\Delta \bar{p}\|_{L^2L^2}^2+\|{\rm e}^{\nu t}\partial_t\partial_z\bar{u}^2\|_{L^2L^2}^2+\|{\rm e}^{\nu t}\partial_z(\overline{u\cdot \nabla u^2})\|_{L^2L^2}^2+\|{\rm e}^{\nu t}\partial_z\bar{\theta}\|_{L^2L^2}^2\right)\\
&\leq C\nu^{-1}\left(\|{\rm e}^{\nu t}\Delta \bar{P}^{NL}\|_{L^2L^2}^2+\|{\rm e}^{\nu t}\partial_t\partial_z\bar{u}^2\|_{L^2L^2}^2+\|{\rm e}^{\nu t}\partial_z(\overline{u\cdot \nabla u^2})\|_{L^2L^2}^2+\|{\rm e}^{\nu t}\nabla\bar{\theta}\|_{L^2L^2}^2\right)\\
& \leq C\left(\|u(0)\|_{H^2}^2+\nu^{-2}E_3^4+\nu^{-1}\|{\rm e}^{\nu t}\nabla \bar{\theta}\|_{L^2L^2}^2\right)(\nu^{-2}E_2^2+1).
\end{align*}
Using Lemma \ref{pxy} to give that
\begin{align}
&\nu^{-1}\left(\|{\rm e}^{\nu t}\partial_y^2\bar{p}\|_{L^2L^2}^2+\|{\rm e}^{\nu t}\partial_z^2\bar{p}\|_{L^2L^2}^2\right)
\leq C\nu^{-1}\|{\rm e}^{\nu t}\Delta \bar{p}\|_{L^2L^2}^2+\|{\rm e}^{\nu t}\partial_y\partial_z\bar{p}\|_{L^2L^2}^2\nonumber\\
&\leq C\left(\|u(0)\|_{H^2}^2+\nu^{-2}E_3^4+\nu^{-1}\|{\rm e}^{\nu t}\nabla \bar{\theta}\|_{L^2L^2}^2\right)(\nu^{-2}E_2^2+1).\label{eqpyz2}
\end{align}

Note that
\begin{align*}
(\partial_t-\nu \Delta )\partial_y\bar{u}^2+\partial_y^2\bar{p}+\partial_y
(\overline{u\cdot \nabla u^2})=\partial_y\bar{\theta},
\end{align*}
then one obtains that
\begin{align*}
\nu\|{\rm e}^{\nu t}\Delta \partial_y\bar{u}^2\|_{L^2L^2}^2
\leq& C\nu^{-1}\left(\|{\rm e}^{\nu t}(\partial_t\partial_y\bar{u}^2+\partial_y^2\bar{p}-\partial_y\bar{\theta})\|_{L^2L^2}^2
+\|{\rm e}^{\nu t}\partial_y(\bar{u}\cdot\nabla\bar{u}^2)\|_{L^2L^2}^2\right.\\
&\left.+
\|{\rm e}^{\nu t}\partial_y(u_{\neq}\cdot\nabla u_{\neq}^2)\|_{L^2L^2}^2\right)\\
\leq& C\left(\|u(0)\|_{H^2}^2+\nu^{-2}E_3^4+\nu^{-1}\|{\rm e}^{\nu t}\nabla \bar{\theta}\|_{L^2L^2}^2\right)(\nu^{-2}E_2^2+1),
\end{align*}
where \eqref{eqpyz2} has been used.

The equation for $\bar{u}^3$ reads as
$$
\partial_t\bar{u}^3-\nu\Delta \bar{u}^3+\overline{u\cdot \nabla u^3}+\partial_z\bar{p}=0,
$$
then one has
\begin{align*}
\nu\|{\rm e}^{\nu t}\Delta \bar{u}^3\|_{L^2L^2}^2\leq& C\nu^{-1}\left(\|{\rm e}^{\nu t}\partial_t \bar{u}^3\|_{L^2L^2}^2+
\|{\rm e}^{\nu t}\partial_z^2 \bar{p}\|_{L^2L^2}^2\right.\\
&\left.+\|{\rm e}^{\nu t}\bar{u}\cdot\nabla \bar{u}^3\|_{L^2L^2}^2+\|{\rm e}^{\nu t}u_{\neq}\cdot\nabla u_{\neq}^3\|_{L^2L^2}^2\right)\\
\leq& C\left(\|u(0)\|_{H^2}^2+\nu^{-2}E_3^4+\nu^{-1}\|{\rm e}^{\nu t}\nabla \bar{\theta}\|_{L^2L^2}^2\right)(\nu^{-2}E_2^2+1),
\end{align*}
where we have used $\|\partial_z\bar{p}\|_{L^2} \leq\|\partial_z^2\bar{p}\|_{L^2}$.

The proof is completed.
\end{proof}
\begin{lemma}\label{lemE2-4}
It holds that
\begin{align*}
&\left\|\min\{(\nu^{\frac{2}{3}}+\nu t)^{\frac{1}{2}}, 1-y^2\}\Delta \bar{u}^3\right\|_{L^\infty L^2}^2+\nu \left\|\min\{(\nu^{\frac{2}{3}}+\nu t)^{\frac{1}{2}}, 1-y^2\}\nabla\Delta \bar{u}^3\right\|_{L^2 L^2}^2\\
&\quad+\nu^{-1}\left\|\min\{(\nu^{\frac{2}{3}}+\nu t)^{\frac{1}{2}}, 1-y^2\}\partial_t\nabla \bar{u}^3\right\|_{L^2L^2}^2\\
\leq& C\left(\|u(0)\|_{H^2}^2+\nu^{-2}E_3^4+\nu^{-1}\|{\rm e}^{\nu t}\nabla \bar{\theta}\|_{L^2L^2}^2\right)(\nu^{-2}E_2^2+1).
\end{align*}
\end{lemma}
\begin{proof}
The proof is very similar to that of Lemma 12.4 of \cite{chenwz}. The main difference is the appearance of temperature.

First, $\nabla \bar{u}^3$ solves
\begin{align}
\partial_t\nabla\bar{u}^3-\nu\Delta \nabla\bar{u}^3+\nabla(\overline{u\cdot \nabla u^3})+\partial_z\nabla\bar{p}=0.\label{eqnabu3}
\end{align}
Let $\rho$ be the same as in Lemma 12.4 of \cite{chenwz}, and it is clear that for $t\geq 0$, it holds that
\begin{align*}
\rho(t,y) \sim \min\{(\nu^{\frac{2}{3}}+\nu t)^{\frac{1}{2}}, 1-y^2\}, y\in [-1,1].
\end{align*}
Following the same arguments as in Lemma 12.4 of \cite{chenwz}, one can obtain that
\begin{align*}
&\nu^{-1}\|\rho\partial_t\nabla \bar{u}^3\|_{L^2L^2}^2+\nu\|\rho \Delta\nabla \bar{u}^3\|_{L^2L^2}^2+\|\rho\Delta\bar{u}^3\|_{L^\infty L^2}^2\\
&\leq C\left(\|u(0)\|_{H^2}^2+\nu^{-1}\|\rho (\nabla(\overline{u\cdot \nabla u^3})+\partial_z\nabla\bar{p})\|_{L^2L^2}^2+\nu\|\Delta \bar{u}^3\|_{L^2L^2}^2\right)\\
&\leq C\left(\|u(0)\|_{H^2}^2+\nu\|\Delta \bar{u}^3\|_{L^2L^2}^2+\nu^{-1}\|\partial_z\nabla\bar{p}\|_{L^2L^2}^2\right.\\
&\quad\quad\left.+\nu^{-1}\|\nabla(\bar{u}\cdot \nabla \bar{u}^3)\|_{L^2L^2}^2+\nu^{-1}\|(\nu^{2/3}+\nu t)^{1/2}\nabla(\overline{u_{\neq}\cdot \nabla u_{\neq}^3})\|_{L^2L^2}^2\right),
\end{align*}
 which together with \eqref{eqE2-101}, \eqref{eqpyz2}, Lemma \ref{lemuu-11} and Lemma \ref{lemE2-3} gives that
\begin{align*}
&\nu^{-1}\|\rho \partial_t\nabla \bar{u}^3\|_{L^2L^2}^2+\nu\|\rho \Delta\nabla \bar{u}^3\|_{L^2L^2}^2+\|\rho\Delta\bar{u}^3\|_{L^\infty L^2}^2\\
&\leq C\left(\|u(0)\|_{H^2}^2+\nu^{-2}E_3^4+\nu^{-1}\|{\rm e}^{\nu t}\nabla \bar{\theta}\|_{L^2L^2}^2\right)(\nu^{-2}E_2^2+1).
\end{align*}
The conclusion follows from the definition of $\rho.$
\end{proof}
It is ready to prove Proposition \ref{proE2}.
\begin{proof}[Proof of Proposition \ref{proE2}]
Using Proposition \ref{proE40}, Lemmas \ref{lemE21}--\ref{lemE2-4}, one obtains
\begin{align*}
E_2^2\leq & C\left(\|u(0)\|_{H^2}^2+\nu^{-2}E_3^4+\nu^{-1}\|{\rm e}^{\nu t}\nabla \bar{\theta}\|_{L^2L^2}^2\right)(\nu^{-2}E_2^2+1)\\
\leq&C\left(\|u(0)\|_{H^2}^2+\nu^{-2}E_3^4+\nu^{-2}\|\bar{\theta}(0)\|_2^2+ \nu^{-4}E_{3,0}^2E_{4,1}^2\right)(\nu^{-2}E_2^2+1),
\end{align*}
which yields the desired conclusion.
\end{proof}
\section{Estimate of non-zero modes}

Let us recall that $(\Delta u^2,\omega^2,\theta_{\neq})$ satisfy
\begin{align*}
\begin{cases}
\partial_t\Delta u^2-\nu\Delta^2u^2+y\partial_x\Delta u^2+(\partial_x^2+\partial_z^2)(u\cdot \nabla u^2)-\partial_y[\partial_x(u\cdot\nabla u^1)+\partial_z(u\cdot\nabla u^3)]\\
\quad\quad=(\partial_x^2+\partial_z^2)\theta,\\
\partial_t\omega^2-\nu\Delta \omega^2+y\partial_x\omega^2+\partial_zu^2+\partial_z(u\cdot\nabla u^1)-\partial_x(u\cdot\nabla u^3)=0,\\
\partial_t\theta_{\neq}-\nu \Delta \theta_{\neq}+y\partial_x\theta_{\neq}+(u\cdot \nabla \theta)_{\neq}=0,\\
(\partial_yu^2, u^2, \omega^2,\theta_{\neq})|_{y=\pm1}=(0,0,0,0),\\
 u^2|_{t=0}=u^2(0), \omega^2|_{t=0}=\omega^2 (0), \theta_{\neq}|_{t=0}=\theta_{\neq}(0),
\end{cases}
\end{align*}
We denote
\begin{align*}
\hat{\Delta}=\Delta_{k,l}=\partial_y^2-k^2-l^2, f_{k,l}=\frac{1}{2\pi}\int_{\mathbb{T}^2}f(x,y,z){\rm e}^{-{\rm i}kx-{\rm i}lz}dxdz.
\end{align*}
Taking Fourier transform in $(x,z)$ of equations for $(\Delta u^2,\omega^2)$, we obtain
\begin{align*}
\begin{cases}
\partial_t\hat{\Delta} u^2_{k,l}-\nu\hat{\Delta}^2u^2_{k,l}+{\rm i} ky\hat{\Delta} u^2_{k,l}-(k^2+l^2)(u\cdot \nabla u^2)_{k,l}\\
\quad\quad-\partial_y[{\rm i}k(u\cdot\nabla u^1)_{k,l}+{\rm i}l(u\cdot\nabla u^3)_{k,l}]
=-(k^2+l^2)\theta_{k,l},\\
\partial_t\omega_{k,l}^2-\nu\hat{\Delta} \omega^2_{k,l}+{\rm i}ky\omega^2_{k,l}+{\rm i}lu^2_{k,l}+{\rm i}l(u\cdot\nabla u^1)_{k,l}-{\rm i}k(u\cdot\nabla u^3)_{k,l}=0,\\
(\partial_yu^2_{k,l}, u^2_{k,l},\omega^2_{k,l})|_{y=\pm1}=(0,0), u^2_{k,l}|_{t=0}=u^2(0)_{k,l}, \omega^2_{k,l}|_{t=0}={\rm i}lu^1_{k,l}(0)-{\rm i} ku^3_{k,l}(0).
\end{cases}
\end{align*}

Let $a\geq 0$ and $\eta=\sqrt{k^2+l^2}$. We introduce the following norms:
\begin{align*}
\|f\|_{M^a_{k,l}}^2=&\eta|k|\| {\rm e}^{a\nu^{\frac{1}{3}}t}(\partial_y,{\rm i}\eta)f\|_{L^2L^2}^2
+\nu \eta^2\| {\rm e}^{a\nu^{\frac{1}{3}}t} (\partial_y^2-\eta^2)f\|_{L^2L^2}^2\\
&+\nu^{\frac{3}{2}}\| {\rm e}^{a\nu^{\frac{1}{3}}t} \partial_y(\partial_y^2-\eta^2)f\|_{L^2L^2}^2
+\eta^2\| {\rm e}^{a\nu^{\frac{1}{3}}t} (\partial_y,{\rm i}\eta)f\|_{L^\infty L^2}^2\\
&+\nu^{\frac{1}{2}}\| {\rm e}^{a\nu^{\frac{1}{3}}t} (\partial_y^2-\eta^2)f\|_{L^\infty L^2}^2,
\end{align*}
and
\begin{align*}
\|f\|_{N^a_{k,l}}^2=\|{\rm e}^{a\nu^{\frac{1}{3}}t}f\|_{L^\infty L^2}^2
+\nu\|{\rm e}^{a\nu^{\frac{1}{3}}t} \partial_yf\|_{L^2 L^2}^2+\left((\nu k^2)^{\frac{1}{3}}+\nu \eta^2\right)\|{\rm e}^{a\nu^{\frac{1}{3}}t} f\|_{L^2 L^2}^2,
\end{align*}
obviously, one has
\begin{align*}
\|f\|_{M_a}^2=\sum_{k\neq 0;l\in \mathbb{Z}}\|\hat{f}(k,l)\|_{M^a_{k,l}}^2,\quad\quad\|f\|_{N_a}^2=\sum_{k\neq 0;l\in \mathbb{Z}}\|\hat{f}(k,l)\|_{N^a_{k,l}}^2.
\end{align*}
Thus, it is clear that
\begin{align}
&\|{\rm e}^{a\nu^{\frac{1}{3}}t} \partial_x\nabla f_{\neq}\|_{L^2 L^2}^2+\nu\|{\rm e}^{a\nu^{\frac{1}{3}}t} (\partial_x,\partial_z)\Delta f_{\neq}\|_{L^2 L^2}^2+\nu^{\frac{3}{2}}\|{\rm e}^{a\nu^{\frac{1}{3}}t} \partial_y\Delta f_{\neq}\|_{L^2 L^2}^2\nonumber\\
&\quad\quad+\|{\rm e}^{a\nu^{\frac{1}{3}}t} (\partial_x,\partial_z)\nabla f_{\neq}\|_{L^\infty L^2}^2+\nu^{\frac{1}{2}}\|{\rm e}^{a\nu^{\frac{1}{3}}t} \Delta f_{\neq}\|_{L^\infty L^2}^2\leq \|f_{\neq}\|_{M_a}^2,\label{eqfM}\\
&\|{\rm e}^{a\nu^{\frac{1}{3}}t} f_{\neq}\|_{L^\infty L^2}^2+\nu \|{\rm e}^{a\nu^{\frac{1}{3}}t} \nabla f_{\neq}\|_{L^2 L^2}^2\leq \|f_{\neq}\|_{N_a}^2.
\end{align}
\begin{lemma}\label{lemE30}
It holds that
\begin{align*}
E_{3,0}^2\leq C\left(\|u_{\neq}^2\|_{M_{2\epsilon}}^2+\|\partial_x\omega_{\neq}^2\|_{N_{2\epsilon}}^2\right).
\end{align*}
\end{lemma}
\begin{proof}
See Lemma 13.1 of \cite{chenwz}.
\end{proof}
\subsection{Estimate of $E_{3,0}$}
\begin{proposition}\label{proE30}
It holds that
\begin{align*}
E_{3,0}^2+\|u_{\neq}^2\|_{M_a}^2\leq C\|u(0)\|_{H^2}^2+C\left(\nu^{-2}E_3^4+\nu^{-2}E_2^2E_3^2+E_1^2E_3E_5+E_1^2E_3^\frac{3}{2}E_5^\frac{1}{2}+\nu^{-\frac{4}{3}}E_{4,1}^2\right).
\end{align*}
\end{proposition}
\begin{proof}
The main idea of the proof is similar to that of Proposition 13.1 of \cite{chenwz}. The main difference is treatment of temperature.

By Proposition \ref{space-na-y} and Proposition \ref{space-time-non1} and $\partial_yu^2_{k,l}+{\rm i}ku^1_{k,l}+{\rm i}lu^3_{k,l}=0$, one has
\begin{align*}
&\|u_{k,l}^2\|_{M^{2\epsilon}_{k,l}}^2\leq C\left(\|\hat{\Delta}u^2_{k,l}(0)\|_{L^2}^2+\eta^{-2}\|\hat{\Delta}({\rm i}ku^1_{k,l}+{\rm i}lu^3_{k,l})(0)\|_{L^2}^2\right)
+C\nu^{-1}\|{\rm e}^{2\epsilon \nu^{\frac{1}{3}}t}(\partial_x,\partial_z)\theta_{k,l}\|_{L^2L^2}^2\\
&\quad\quad+C\nu^{-1}\left(\|{\rm e}^{2\epsilon \nu^{\frac{1}{3}}t}[\partial_x(u\cdot\nabla u^1)+\partial_z(u\cdot\nabla u^3)]_{k,l}\|_{L^2L^2}^2+
\|{\rm e}^{2\epsilon \nu^{\frac{1}{3}}t}(k,l)(u\cdot\nabla u^2)_{k,l}\|_{L^2L^2}^2\right),
\end{align*}
and
\begin{align*}
&\|\omega_{k,l}^2\|_{N^{2\epsilon}_{k,l}}^2\leq C\left(\|\omega_{k,l}^2(0)\|_{L^2}^2+(l^2(|k|\eta)^{-1})\| {\rm e}^{2\epsilon\nu^{\frac{1}{3}}t}\partial_yu_{k,l}^2\|_{L^2L^2}^2
+(l^2\eta|k|^{-1})\| {\rm e}^{2\epsilon\nu^{\frac{1}{3}}t} u_{k,l}^2\|_{L^2L^2}^2\right.\\
&\quad\quad+\left. \min\{(\nu\eta^2)^{-1},(\nu k^2)^{-1/3}\}\| {\rm e}^{2\epsilon\nu^{\frac{1}{3}}t} [k(u\cdot\nabla u^3)_{k,l}-l(u\cdot\nabla u^1)_{k,l}]\|_{L^2L^2}^2\right)\\
&\leq C\left(\|\omega_{k,l}^2(0)\|_{L^2}^2+(l^2(|k|\eta)^{-1})\| {\rm e}^{2\epsilon\nu^{\frac{1}{3}}t} \partial_yu_{k,l}^2\|_{L^2L^2}^2
+(l^2|\eta|k|^{-1})\| {\rm e}^{2\epsilon\nu^{\frac{1}{3}}t} u_{k,l}^2\|_{L^2L^2}^2\right.\\
&\quad\quad+\left. \nu^{-1}\| {\rm e}^{2\epsilon\nu^{\frac{1}{3}}t} (u\cdot\nabla u^3)_{k,l}\|_{L^2L^2}^2+\nu^{-1}\| {\rm e}^{2\epsilon\nu^{\frac{1}{3}}t} (u\cdot\nabla u^1)_{k,l}]\|_{L^2L^2}^2\right).
\end{align*}

Note that
\begin{align*}
&\sum_{k\neq 0;l\in \mathbb{Z}}k^2\left((l^2(|k|\eta)^{-1})\| {\rm e}^{2\epsilon\nu^{\frac{1}{3}}t} \partial_yu_{k,l}^2\|_{L^2L^2}^2
+(l^2\eta|k|^{-1})\| {\rm e}^{2\epsilon\nu^{\frac{1}{3}}t} u_{k,l}^2\|_{L^2L^2}^2\right)\\
&\leq \sum_{k\neq 0;l\in \mathbb{Z}}\eta|k|\| {\rm e}^{2\epsilon\nu^{\frac{1}{3}}t} (\partial_y,{\rm i}\eta)u^2_{k,l}\|_{L^2L^2}^2
\leq \sum_{k\neq 0;l\in \mathbb{Z}}\|u^2_{k,l}\|_{M^{2\epsilon}_{k,l}}^2=\|u_{\neq}^2\|_{M_{2\epsilon}}^2,
\end{align*}
therefore, it follows from the Lemma \ref{lemE30} that
\begin{align*}
&E_{3,0}^2+\|u_{\neq}^2\|_{M_{2\epsilon}}^2\leq C\left(\|u_{\neq}^2\|_{M_{2\epsilon}}^2+\|\partial_x\omega_{\neq}^2\|_{N_{2\epsilon}}^2\right)\\
&\leq C\left(\|\omega^2_{\neq}(0)\|_{H^1}^2+\nu^{-1}\| {\rm e}^{2\epsilon\nu^{\frac{1}{3}}t} \partial_x(u\cdot\nabla u^3)\|_{L^2L^2}^2+\nu^{-1}\| {\rm e}^{2\epsilon\nu^{\frac{1}{3}}t} \partial_x(u\cdot\nabla u^1)\|_{L^2L^2}^2\right)+C\|u_{\neq}^2\|_{M_{2\epsilon}}^2\\
&\leq C\left(\|\omega^2_{\neq}(0)\|_{H^1}^2+\nu^{-1}\| {\rm e}^{2\epsilon\nu^{\frac{1}{3}}t} \partial_x(u\cdot\nabla u^3)\|_{L^2L^2}^2+\nu^{-1}\| {\rm e}^{2\epsilon\nu^{\frac{1}{3}}t} \partial_x(u\cdot\nabla u^1)\|_{L^2L^2}^2\right)\\
&+C\nu^{-1}\sum_{k\neq 0;l\in \mathbb{Z}}\left(\| {\rm e}^{2\epsilon\nu^{\frac{1}{3}}t} [\partial_x(u\cdot\nabla u^1)+\partial_z(u\cdot\nabla u^3)]_{k,l}\|_{L^2L^2}^2+
\| {\rm e}^{2\epsilon\nu^{\frac{1}{3}}t} (k,l)(u\cdot\nabla u^2)_{k,l}\|_{L^2L^2}^2\right)\\
&+C\sum_{k\neq 0;l\in \mathbb{Z}}\left(\left(\|\hat{\Delta}u^2_{k,l}(0)\|_{L^2}^2+\eta^{-2}\|\hat{\Delta}({\rm i} ku^1_{k,l}+{\rm i}lu^3_{k,l})(0)\|_{L^2}^2\right)
+\nu^{-1}\| {\rm e}^{2\epsilon\nu^{\frac{1}{3}}t}(\partial_x,\partial_z)\theta_{k,l}\|_{L^2L^2}^2\right)\\
\leq& C\left(\|\omega_{\neq}^2\|_{H^1}^2+\|u_{\neq}(0)\|_{H^2}^2\right)\\
&
+C\nu^{-1}\left(\| {\rm e}^{2\epsilon\nu^{\frac{1}{3}}t} (\partial_x,\partial_z)(u\cdot\nabla u^3)_{\neq}\|_{L^2L^2}^2+\| {\rm e}^{2\epsilon\nu^{\frac{1}{3}}t} \partial_x(u\cdot\nabla u^1)_{\neq}\|_{L^2L^2}^2\right)\\
&+C\nu^{-1}\left(\| {\rm e}^{2\epsilon\nu^{\frac{1}{3}}t} (\partial_x,\partial_z)(u\cdot\nabla u^2)_{\neq}\|_{L^2L^2}^2+\| {\rm e}^{2\epsilon\nu^{\frac{1}{3}}t}(\partial_x,\partial_z)\theta_{\neq}\|_{L^2L^2}^2\right).
\end{align*}

By the definition of $E_{4,1}$, one has
\begin{align*}
\| {\rm e}^{2\epsilon\nu^{\frac{1}{3}}t} (\partial_x,\partial_z)\theta_{\neq}\|_{L^2L^2}^2
\leq C\nu^{-\frac{1}{3}}E_{4,1}^2.
\end{align*}
By Lemma \ref{lemuu-11} and Lemma \ref{lemuu-10}, and using the fact that $(fg)_{\neq}=\bar{f}g_{\neq}+f_{\neq}\bar{g}+(f_{\neq}g_{\neq})_{\neq}$, one can deduce that
\begin{align*}
E_{3,0}^2+\|u_{\neq}^2\|_{M_a}^2\leq C\|u(0)\|_{H^2}^2+C\left(\nu^{-2}E_3^4+\nu^{-2}E_2^2E_3^2+E_1^2E_3E_5+E_1^2E_3^\frac{3}{2}E_5^\frac{1}{2}+\nu^{-\frac{4}{3}}E_{4,1}^2\right),
\end{align*}
which completes the proof.
\end{proof}
\subsection{Estimate of $E_{3,1}$}
\begin{proposition}\label{proE31}
It holds that
\begin{align*}
E_{3,1}^2\leq C\left(\|u(0)\|_{H^2}^2+\nu^{-2}E_3^4+\nu^{-2}E_2^2E_3^2+E_1^2E_3E_5+E_1^2E_3^\frac{3}{2}E_5^\frac{1}{2}+E_1^2E_3^\frac{7}{4}E_5^\frac{1}{4}+\nu^{-\frac{4}{3}}E_{4,1}^2\right).
\end{align*}
\end{proposition}
\begin{proof}
Recall that $\omega_{k,l}^2$ solves
\begin{align*}
\begin{cases}
\partial_t\omega_{k,l}^2-\nu\hat{\Delta} \omega^2_{k,l}+{\rm i}ky\omega^2_{k,l}+{\rm i}lu^2_{k,l}+{\rm i}l(u\cdot\nabla u^1)_{k,l}-{\rm i}k(u\cdot\nabla u^3)_{k,l}=0,\\
\omega^2_{k,l}|_{y=\pm1}=0, \omega^2_{k,l}|_{t=0}={\rm i}lu^1_{k,l}(0)-{\rm i}ku^3_{k,l}(0).
\end{cases}
\end{align*}
The equation for $\omega^2$ is the same as in \cite{chenwz}, the main difference is the presence of temperature in equation for velocity. The proof can be achieved via following the same arguments as in Proposition 13.2 in \cite{chenwz} with application of Proposition \ref{proE30}. We omit the details here.
\end{proof}
\subsection{Estimate of $E_{4,1}$ and $E_5$}\label{E-4:E-5}
Inspired by \cite{chenwz}, we will use the method of freezing the coefficient in time to estimate $E_{4,1}$ and $E_5$.

Let $ \mathring{J} (\Omega)$ and $J_{2,0}^{(k)}(\Omega)$ denote the closure of the set of vector field, which is smooth and solenoidal in $\Omega$ and vanish on $\partial \Omega$, in
the topology, respectively, of $L^2(\Omega)$ and Sobolev space $W^{k,2}(\Omega)$.

We define
\begin{align*}
&Q(u,v)=\mathbb{P}(u\cdot \nabla v+v\cdot\nabla u),\\
&Q_1(f,v)=Q((f,0,0),v), A_{[V]}v=\nu\mathbb{P}\Delta v-Q_1(V,v), B_{[V]}=\nu\Delta -V\partial_x,
\end{align*}
where $V$ satisfies \eqref{eqVcon} and $\mathbb{P}$ is a projector in $L^2(\Omega)$ into $\mathring{J} (\Omega)$. Then $A_{[V]}$ defined on $J^{(2)}_{2,0}$ is invariant in the subspace $\mathcal{H}=\{f\in L^2(\Omega)|P_0f=0\}$.

We denote by $m(\nu, V)$ the upper bound of the real parts of points of the spectrum of $A_{[V]}$ in the subspace $J^{(2)}_{2,0}\cap \mathcal{H}$. More precisely, we consider $A_{[V]}$ as a closed linear operator in the Hilbert space $\mathring{J} \cap \mathcal{H}$
 with the domain $D(A_{[V]})=J^{(2)}_{2,0}\cap \mathcal{H}$.

Inspired by \cite{chenwz}, let $\kappa=\partial_zV/\partial_yV$, and for $v=(v_1,v_2,v_3)$, we introduce the notations:
\begin{align*}
\|v\|_{Z_{[V]}^1}^2=&\nu^{-1}(\|\nabla(v^2+\kappa v^3)\|_{L^2}^2+\|\partial_xv^3\|_{L^2}^2),\\
\|v\|_{Z_{[V]}^2}^2=&\nu^{\frac{1}{3}}(\|\partial_x^2v^3\|_{L^2}^2+\|(\partial_x(\partial_z-\kappa\partial_y)v^3\|_{L^2}^2)
+\nu(\|\nabla\partial_x^2 v^3\|_{L^2}^2+\|\nabla\partial_x(\partial_z-\kappa\partial_y)v^3\|_{L^2}^2)\\
&+\nu^{\frac{1}{3}}\|\partial_x\nabla(v^2+\kappa v^3)\|_{L^2}^2+\nu\|\partial_x\Delta(v^2+\kappa v^3)\|_{L^2}^2+\nu^{\frac{5}{3}}\|\partial_x\Delta v^3\|_{L^2}^2.
\end{align*}
\begin{lemma}\label{lemvZ-1}
It holds that for any $v\in J^{(3)}_{2,0}(\Omega)\cap\mathcal{H}$,
\begin{align*}
C^{-1}\nu^{\frac{1}{3}}\|\partial_x^2v\|_{L^2}^2\leq \|v\|_{Z_{[V]}^2}^2\leq C\nu^{\frac{1}{3}}\|\partial_x v\|_{H^2}^2.
\end{align*}
More precisely, it holds that
\begin{align*}
\nu^{\frac{1}{3}}(\|\partial_x^2 v^2\|_{L^2}^2+\|\partial_x(\partial_z-\kappa\partial_y)v^2\|_{L^2}^2)+\nu(\|\nabla\partial_x^2v\|_{L^2}^2
+\|\nabla\partial_x(\partial_z-\kappa\partial_y) v^2\|_{L^2}^2)\leq C\|v\|_{Z_{[V]}^2}^2.
\end{align*}
\end{lemma}
\begin{proof}
See Lemma 14.1 in \cite{chenwz}.
\end{proof}
\begin{proposition}\label{proZ1-1}
Let $\lambda\in \mathbb{C}$, ${\bf Re} (\lambda)\in [0,\epsilon_1\nu^{1/3}]$. It holds that for any $v\in J^{(3)}_{2,0}(\Omega)\cap\mathcal{H}$,
\begin{align*}
\|v\|_{Z_{[V]}^2}\leq C\|(A_{[V]}+\lambda)v\|_{Z_{[V]}^1}.
\end{align*}
\end{proposition}
\begin{proof}
See Proposition 14.1 of \cite{chenwz}.
\end{proof}

 Let $t_j=j\nu ^{-\frac{1}{3}}$, $I_j=[t_j, t_{j+1})\cap [0, T]$. We define
\begin{align*}
&V_j(y,z)=y+u^{1,0}(t_j,y,z), \kappa_j=\partial_zV_j/\partial_yV_j, A_j=A_{[V_j]}, B_j=B_{[V_j]}\\
&\tilde{u}^{1,1}=y+\bar{u}^1-V_j, t\in I_j, j\in [0, \nu^{\frac{1}{3}}T)\cap \mathbb{Z}.
\end{align*}
Then ${\rm e}^{a\nu^{1/3}t}\sim {\rm  e}^{aj}$, for $t\in I_j$. We denote
\begin{align*}
\|v\|_{Z^l_j}=\|v\|_{Z^l_{[V_j]}}, \text{ for } j\in [0, \nu^{\frac{1}{3}}T)\cap \mathbb{Z}, l\in\{1,2\}.
\end{align*}

Recall that $\theta_{\neq}$ and $u_{\neq}$ satisfy
\begin{align*}
\begin{cases}
\partial_t\theta_{\neq}-\nu \Delta \theta_{\neq}+(y+u^{1,0})\partial_x\theta_{\neq}+u^{1,1}\partial_x\theta_{\neq}
+(\bar{u}^2\partial_y+\bar{u}^3\partial_z)\theta_{\neq}\\
\quad\quad+u_{\neq}\cdot\nabla\bar{\theta}
+(u_{\neq}\cdot\nabla \theta_{\neq})_{\neq}=0,\\
\partial_tu_{\neq}-\nu\mathbb{P}\Delta u_{\neq}+Q_1(y,u_{\neq})+Q(\bar{u},u_{\neq})+\mathbb{P}(u_{\neq}\cdot\nabla u_{\neq})_{\neq}=\mathbb{P}(0,\theta_{\neq},0).
\end{cases}
\end{align*}
The equation for $u_{\neq}$ can be rewritten as
\begin{align*}
Q_1(y,u_{\neq})+Q(\bar{u},u_{\neq})=&Q_1(y+\bar{u}^1,u_{\neq})+Q((0,\bar{u}^2,\bar{u}^3),u_{\neq})\\
=&Q_1(V_j,u_{\neq})+Q_1(\tilde{u}^{1,1},u_{\neq})+Q((0,\bar{u}^2,\bar{u}^3),u_{\neq}).
\end{align*}
Then, we have
\begin{align*}
\begin{cases}
\partial_t\theta_{\neq}-\nu\Delta \theta_{\neq}+V_j\partial_x\theta_{\neq}+f=0, t\in I_j,\\
\partial_tu_{\neq}+A_ju_{\neq}+g=0, t\in I_j,
\end{cases}
\end{align*}
where
\begin{align*}
f=&(u^{1,0}(t,y,z)-u^{1,0}(t_j,y,z))\partial_x\theta_{\neq}+u^{1,1}\partial_x\theta_{\neq}
+(\bar{u}^2\partial_y+\bar{u}^3\partial_z)\theta_{\neq}\\
&+u_{\neq}\cdot\nabla\bar{\theta}
+(u_{\neq}\cdot\nabla \theta_{\neq})_{\neq},\\
g=&Q_1(\tilde{u}^{1,1},u_{\neq})+Q((0,\bar{u}^2,\bar{u}^3),u_{\neq})+
\mathbb{P}(u_{\neq}\cdot\nabla u_{\neq})_{\neq}-\mathbb{P}(0,\theta_{\neq},0).
\end{align*}
We define $f_{(j)}$, $g_{(j)}$, $u_{[j]}$, and $\theta_{[j]}$ for $j\in [0, \nu^{\frac{1}{3}}T)\cap \mathbb{Z}$ iteratively by solving
\begin{align*}
&f_{(j)}(t)=0 \text{ for } t\notin I_j, \quad f_{(j)}(t)=f+\sum_{i=0}^{j-1}(B_i-B_j)\theta_{[i]} \text{ for } t\in I_j,\\
&g_{(j)}(t)=0 \text{ for } t\notin I_j, \quad g_{(j)}(t)=g+\sum_{i=0}^{j-1}(A_i-A_j)u_{[i]} \text{ for } t\in I_j,\\
&\partial_t\theta_{[j]}-\nu\Delta \theta_{[j]}+V_j\partial_x\theta_{[j]}+f_{(j)}=0, \theta_{[j]}\in J^{(3)}_{2,0} \text{ for }t\in [0,\infty),\\
&\theta_{[0]}(0)=\theta_{\neq}(0), \theta_{[j]}(0)=0 \text{ for } j\in [0, \nu^{\frac{1}{3}}T)\cap \mathbb{Z},\\
&\partial_tu_{[j]}-A_ju_{[j]}+g_{(j)}=0, u_{[j]}(t)\in J^{(3)}_{2,0}(\Omega)\cap \mathcal{H},\text{ for } t\in[0,\infty),\\
&u_{[0]}(0)=u_{\neq}(0), u_{[j]}(0)=0 \text{ for } j\in [0, \nu^{\frac{1}{3}}T)\cap \mathbb{Z}.
\end{align*}
Then we find that
\begin{align*}
&\theta_{\neq}=\sum_{i=0}^j\theta_{[i]} \text{ for }t\in I_j, j\in [0, \nu^{\frac{1}{3}}T)\cap \mathbb{Z}, \theta_{[j]}(t)=0\text{ for }0\leq t\leq t_j,\\
&u_{\neq}=\sum_{i=0}^ju_{[i]} \text{ for }t\in I_j, j\in [0, \nu^{\frac{1}{3}}T)\cap \mathbb{Z}, u_{[j]}(t)=0\text{ for }0\leq t\leq t_j.
\end{align*}
\subsubsection{Estimate of $E_{4,1}$}
\begin{proposition}\label{proE41}
It holds that
\begin{align*}
E_{4,1}^2\leq C(\|\theta(0)\|_{H^2}^2+\nu^{-2}E_{3,0}^2E_{4,1}^2+\nu^{-4/3}E_{4,0}^2E_5^2+E_1^2E_{4,1}^2+\nu^{-2}E_{3,0}^2E_{4,0}^2+\nu^{-2}E_2^2E_{4,1}^2).
\end{align*}
\end{proposition}
\begin{proof}
Let $N=\max([0,\nu^{1/3}T)\cap \mathbb{Z})$.
Recall that $\theta_{[j]}$ satisfies
\begin{align*}
\partial_t\theta_{[j]}-\nu\Delta \theta_{[j]}+V_j\partial_x\theta_{[j]}+f_{(j)}=0, \text{ for } t\in I_j,
\end{align*}
and then
\begin{align*}
\partial_t\partial_x^2\theta_{[j]}-\nu\Delta \partial_x^2\theta_{[j]}+V_j\partial_x^3\theta_{[j]}+\partial_x^2f_{(j)}=0, \text{ for } t\in I_j.
\end{align*}

By Proposition \ref{space-time-V21}, one has
for $j>0$,
\begin{align*}
&\|{\rm e}^{4\epsilon\nu^{\frac{1}{3}}t} \partial_x^2\theta_{[j]}\|_{L^\infty L^2}^2+\nu\|{\rm e}^{4\epsilon\nu^{\frac{1}{3}}t} \nabla\partial_x^2\theta_{[j]}\|_{L^2 L^2}^2+\nu^{\frac{1}{3}}\|{\rm e}^{4\epsilon\nu^{\frac{1}{3}}t} \partial_x^2\theta_{[j]}\|_{L^2 L^2}^2\\
\leq& C\nu^{-\frac{1}{3}}\left(\|{\rm e}^{4\epsilon\nu^{\frac{1}{3}}t} (u^{1,0}-u^{1,0}|_{t=t_j})\partial_x^3\theta_{\neq}\|_{L^2_{I_j}L^2}^2
+\sum_{i=0}^{j-1}\|{\rm e}^{4\epsilon\nu^{\frac{1}{3}}t} (B_i-B_j)\partial_x^2\theta_{[i]}\|_{L^2_{I_j}L^2}^2\right)\\
&+C\nu^{-1}\left(\|{\rm e}^{4\epsilon\nu^{\frac{1}{3}}t} \partial_x^2(u_{\neq}\theta_{\neq})\|_{L^2_{I_j}L^2}
+\|{\rm e}^{4\epsilon\nu^{\frac{1}{3}}t} (\partial_x^2u^2_{\neq}\partial_y+\partial_x^2u_{\neq}^3\partial_z)\bar{\theta}\|_{L^2_{I_j}L^2}^2\right)\\
&+C\nu^{-\frac{1}{3}}\|{\rm e}^{4\epsilon\nu^{\frac{1}{3}}t} (u^{1,1}\partial_x+\bar{u}^2\partial_y+\bar{u}^3\partial_z)\partial_x^2\theta_{\neq}\|_{L^2_{I_j}L^2}^2\\
\leq& C\nu^{-\frac{1}{3}}{\rm e}^{8 \epsilon j}\left(\|(u^{1,0}-u^{1,0}|_{t=t_j})\partial_x^3\theta_{\neq}\|_{L^2_{I_j}L^2}^2
+\sum_{i=0}^{j-1}\|(B_i-B_j)\partial_x^2\theta_{[i]}\|_{L^2_{I_j}L^2}^2\right)\\
&+C\nu^{-1}{\rm e}^{8 \epsilon j}\left(\|\partial_x^2(u_{\neq}\theta_{\neq})\|_{L^2_{I_j}L^2}^2
+\|(\partial_x^2u^2_{\neq}\partial_y+\partial_x^2u_{\neq}^3\partial_z)\bar{\theta}\|_{L^2_{I_j}L^2}^2\right)\\
&+C\nu^{-\frac{1}{3}} {\rm e}^{8 \epsilon j}\|(u^{1,1}\partial_x+\bar{u}^2\partial_y+\bar{u}^3\partial_z)\partial_x^2\theta_{\neq}\|_{L^2_{I_j}L^2}^2.
\end{align*}
For $t\in I_0=[0,\nu^{-\frac{1}{3}}),$ it holds that $1 \leq {\rm e}^{4\epsilon \nu^{\frac{1}{3}}t}< {\rm e}^{4 \epsilon} \leq {\rm e}^{1/2}$ and
\begin{align*}
&\|{\rm e}^{4\epsilon\nu^{\frac{1}{3}}t} \partial_x^2\theta_{[0]}\|_{L^\infty L^2}^2+\nu\|{\rm e}^{4\epsilon\nu^{\frac{1}{3}}t} \nabla\partial_x^2\theta_{[0]}\|_{L^2 L^2}^2+\nu^{\frac{1}{3}}\|{\rm e}^{4\epsilon\nu^{\frac{1}{3}}t} \partial_x^2\theta_{[0]}\|_{L^2 L^2}^2\\
&\leq C\|\theta(0)\|_{H^2}^2+C\nu^{-\frac{1}{3}}\|{\rm e}^{4\epsilon\nu^{\frac{1}{3}}t}  (u^{1,0}-u^{1,0}|_{t=t_j})\partial_x^3\theta_{\neq}\|_{L^2_{I_0}L^2}^2\\
&\quad+C \nu^{-\frac{1}{3}} \|{\rm e}^{4\epsilon\nu^{\frac{1}{3}}t}  (u^{1,1}\partial_x+\bar{u}^2\partial_y+\bar{u}^3\partial_z)\partial_x^2\theta_{\neq}\|_{L^2_{I_0}L^2}^2\\
&\quad+C\nu^{-1}\left(\|{\rm e}^{4\epsilon\nu^{\frac{1}{3}}t}   \partial_x^2(u_{\neq}\theta_{\neq})\|_{L^2_{I_0}L^2}^2
+\|{\rm e}^{4 \epsilon\nu^{\frac{1}{3}}t}   (\partial_x^2u^2_{\neq}\partial_y+\partial_x^2u_{\neq}^3\partial_z)\bar{\theta}\|_{L^2_{I_0}L^2}^2\right)\\
&\leq C\|\theta(0)\|_{H^2}^2+C\nu^{-\frac{1}{3}}\|(u^{1,0}-u^{1,0}|_{t=t_j})\partial_x^3\theta_{\neq}\|_{L^2_{I_0}L^2}^2\\
&\quad+C\nu^{-\frac{1}{3}} \|(u^{1,1}\partial_x+\bar{u}^2\partial_y+\bar{u}^3\partial_z)\partial_x^2\theta_{\neq}\|_{L^2_{I_0}L^2}^2\\
&\quad+C\nu^{-1}\left(\|\partial_x^2(u_{\neq}\theta_{\neq})\|_{L^2_{I_0}L^2}^2
+\|(\partial_x^2u^2_{\neq}\partial_y+\partial_x^2u_{\neq}^3\partial_z)\bar{\theta}\|_{L^2_{I_0}L^2}^2\right).
\end{align*}
For $j\in [0,\nu^{\frac{1}{3}}T)\cap\mathbb{Z}$, let
\begin{align*}
&m_j=\sum_{i=0}^j(j-i+1)\|\partial_x^3 \theta_{[i]}\|_{L^2_{I_j}L^2},
\\
&s_j={\rm e}^{-4\epsilon j}\left(\|{\rm e}^{4\epsilon\nu^{\frac{1}{3}}t}   \partial_x^2\theta_{[j]}\|_{L^\infty L^2}+\nu^{\frac{1}{2}}\|{\rm e}^{4\epsilon\nu^{\frac{1}{3}}t}   \nabla\partial_x^2\theta_{[j]}\|_{L^2 L^2}+\nu^{\frac{1}{6} }\|{\rm e}^{4\epsilon\nu^{\frac{1}{3}}t}  \partial_x^2\theta_{[j]}\|_{L^2 L^2}\right).
\end{align*}
Indeed, one has
\begin{align*}
\|(B_i-B_j)\partial_x^3 \theta_{[i]}\|_{L^2_{I_j}L^2}=&\|(V_i-V_j)\partial_x^3\theta_{[i]}\|_{L^2_{I_j}L^2}
\leq\|V_i-V_j\|_{L^\infty L^\infty}\| \partial_x^3 \theta_{[i]}\|_{L^2_{I_j}L^2}\\
\leq& C\nu^{\frac{2}{3}}E_1|j-i|\|\partial_x^3\theta_{[i]}\|_{L^2_{I_j}L^2},\\
\|(u^{1,0}-u^{1,0}|_{t=t_j})\partial_x^3\theta_{\neq}\|_{L^2_{I_j}L^2}
\leq& \|u^{1,0}-u^{1,0}|_{t=t_j}\|_{L^\infty_{I_j}L^\infty}\|\partial_x^3\theta_{\neq}\|_{L^2_{I_j}L^2}\\
\leq& C\nu^{\frac{2}{3}}E_1\|\partial_x^3\theta_{\neq}\|_{L^2_{I_j}L^2},
\end{align*}
where the following fact has been used:
\begin{align*}
\|V_i-V_j\|_{L^\infty L^\infty}&=\|u^{1,0}(t_i)-u^{1,0}(t_j)\|_{L^\infty L^\infty}= \left\|\int_{t_i}^{t_j}\partial_tu^{1,0}dt\right\|_{L^\infty L^\infty}\\
&\leq C\nu|t_j-t_i|E_1\leq C\nu^{\frac{2}{3}}E_1|j-i|.
\end{align*}
Therefore, one has
\begin{align*}
&s_0\leq C\|\theta(0)\|_{H^2}+C\nu^{-\frac{1}{6}}\|(u^{1,0}-u^{1,0}|_{t=t_0})\partial_x^3\theta_{\neq}\|_{L^2_{I_0}L^2}\\
&\quad+C\nu^{-\frac{1}{6}} \|(u^{1,1}\partial_x+\bar{u}^2\partial_y+\bar{u}^3\partial_z)\partial_x^2\theta_{\neq}\|_{L^2_{I_0}L^2}\\
&\quad+C\nu^{-\frac{1}{2}}\left(\|\partial_x^2(u_{\neq}\theta_{\neq})\|_{L^2_{I_0}L^2}
+\|(\partial_x^2u^2_{\neq}\partial_y+\partial_x^2u_{\neq}^3\partial_z)\bar{\theta}\|_{L^2_{I_0}L^2}\right),\\
\end{align*}
and
\begin{align*}
s_j\leq& C\nu^{-\frac{1}{6}}\left(\|(u^{1,0}-u^{1,0}|_{t=t_j})\partial_x^3\theta_{\neq}\|_{L^2_{I_j}L^2}
+\sum_{i=0}^{j-1}\|(B_i-B_j)\partial_x^2\theta_{[i]}\|_{L^2_{I_j}L^2}\right)\\
&+C\nu^{-\frac{1}{2}}\left(\|\partial_x^2(u_{\neq}\theta_{\neq})\|_{L^2_{I_j}L^2}
+\|(\partial_x^2u^2_{\neq}+\partial_x^2u_{\neq}^3)\bar{\theta}\|_{L^2_{I_j}L^2}\right)\\
&+C\nu^{-\frac{1}{6}}\|(u^{1,1}\partial_x+\bar{u}^2\partial_y+\bar{u}^3\partial_z)\partial_x^2\theta_{\neq}\|_{L^2_{I_j}L^2}\\
\leq& C\nu^{-\frac{1}{6}}\left(\|(u^{1,0}-u^{1,0}|_{t=t_j})\partial_x^3\theta_{\neq}\|_{L^2_{I_j}L^2}
+\nu^{\frac{2}{3}}E_1\sum_{i=0}^{j-1}|j-i|\|\partial_x^3\theta_{[i]}\|_{L^2_{I_j}L^2}\right)\\
&+C\nu^{-\frac{1}{2}}\left(\|\partial_x^2(u_{\neq}\theta_{\neq})\|_{L^2_{I_j}L^2}
+\|(\partial_x^2u^2_{\neq}+\partial_x^2u_{\neq}^3)\bar{\theta}\|_{L^2_{I_j}L^2}\right)\\
&+C\nu^{-\frac{1}{6}}\|(u^{1,1}\partial_x+\bar{u}^2\partial_y+\bar{u}^3\partial_z)\partial_x^2\theta_{\neq}\|_{L^2_{I_j}L^2}
\\
\leq& C\nu^{-\frac{1}{6}}\left(\|(u^{1,0}-u^{1,0}|_{t=t_j})\partial_x^3\theta_{\neq}\|_{L^2_{I_j}L^2}
+\nu^{\frac{2}{3}}E_1m_j\right)\\
&+C\nu^{-\frac{1}{2}}\left(\|\partial_x^2(u_{\neq}\theta_{\neq})\|_{L^2_{I_j}L^2}
+\|(\partial_x^2u^2_{\neq}+\partial_x^2u_{\neq}^3)\bar{\theta}\|_{L^2_{I_j}L^2}\right)\\
&+C\nu^{-\frac{1}{6}}\|(u^{1,1}\partial_x+\bar{u}^2\partial_y+\bar{u}^3\partial_z)\partial_x^2\theta_{\neq}\|_{L^2_{I_j}L^2}.
\end{align*}
Notice that
\begin{align*}
m_j=&\sum_{i=0}^j(j-i+1)\|\partial_x^3\theta_{[i]}\|_{L^2_{I_j}L^2}\leq \sum_{i=0}^j(j-i+1) {\rm e}^{-4\epsilon j}\|{\rm e}^{4\epsilon\nu^{\frac{1}{3}}t}\partial_x^3\theta_{[i]}\|_{L^2_{I_j}L^2}\\
\leq& \sum_{i=0}^j(j-i+1){\rm e}^{-4\epsilon j}\|{\rm e}^{4\epsilon\nu^{\frac{1}{3}}t}\partial_x^3\theta_{[i]}\|_{L^2L^2}\leq C\nu^{-\frac{1}{2}}\sum_{i=0}^j(j-i+1){\rm e}^{-4\epsilon (j-i)}s_i,
\end{align*}
and hence, we have
\begin{align*}
\sum_{j=0}^N {\rm e}^{3\epsilon j}m_j\leq & C\nu^{-\frac{1}{2}}\sum_{j=0}^N\sum_{i=0}^j(j-i+1) {\rm e}^{3\epsilon j} {\rm e}^{-4\epsilon (j-i)}s_i\\
\leq & C\nu^{-\frac{1}{2}}\sum_{j=0}^N\sum_{i=0}^j(j-i+1) {\rm e}^{-\epsilon (j-i)} {\rm e}^{3\epsilon i}s_i\\
\leq& C\nu^{-\frac{1}{2}}\sum_{j=0}^N {\rm e}^{3\epsilon j}s_i.
\end{align*}
Therefore, one obtains
\begin{align*}
 \sum_{j=0}^N {\rm e}^{3\epsilon j}s_j\leq & C\|\theta(0)\|_{H^2}+C\nu^{-\frac{1}{6}}\sum_{j=0}^N \nu^{\frac{2}{3}}E_1 {\rm e}^{3\epsilon j}\|\partial_x^3\theta_{\neq}\|_{L^2_{I_j}L^2}
+CE_1\sum_{j=0}^N {\rm e}^{3\epsilon j}s_j
\\
&+C\nu^{-\frac{1}{2}}\sum_{j=0}^N {\rm e}^{3\epsilon j}\left(\|\partial_x^2(u_{\neq}\theta_{\neq})\|_{L^2_{I_j}L^2}
+\|(\partial_x^2u^2_{\neq}+\partial_x^2u_{\neq}^3)\bar{\theta}\|_{L^2_{I_j}L^2}\right)\\
&+C\nu^{-\frac{1}{6}}\sum_{j=0}^N {\rm e}^{3\epsilon j}\|(u^{1,1}\partial_x+\bar{u}^2\partial_y+\bar{u}^3\partial_z)\partial_x^2\theta_{\neq}\|_{L^2_{I_j}L^2}\\
\leq& C\|\theta(0)\|_{H^2}+C\nu^{\frac{1}{2}}E_1\|{\rm e}^{3\epsilon \nu^{\frac{1}{3}}t}\partial_x^3\theta_{\neq}\|_{L^2L^2}+CE_1\sum_{j=0}^N {\rm e}^{3\epsilon j}s_j\\
&+C\nu^{-\frac{1}{2}}\left(\|{\rm e}^{3\epsilon \nu^{\frac{1}{3}}t}\partial_x^2(u_{\neq}\theta_{\neq})\|_{L^2L^2}
+\|{\rm e}^{3\epsilon \nu^{\frac{1}{3} }t}(\partial_x^2u^2_{\neq}+\partial_x^2u_{\neq}^3)\bar{\theta}\|_{L^2L^2}\right)\\
&+C\nu^{-\frac{1}{6}}\|{\rm e}^{3\epsilon \nu^{\frac{1}{3}}t}(u^{1,1}\partial_x+\bar{u}^2\partial_y+\bar{u}^3\partial_z)\partial_x^2\theta_{\neq}\|_{L^2L^2},
\end{align*}
which implies that
\begin{align*}
\|\partial_x^2\theta_{\neq}\|_{X_3}\leq& \sum_{j=0}^N {\rm e}^{3\epsilon j}s_j\leq C\|\theta(0)\|_{H^2}+CE_1\nu^{\frac{1}{2}}\|{\rm e}^{3\epsilon \nu^{\frac{1}{3}}t}\partial_x^3\theta_{\neq}\|_{L^2L^2}\\
&+C\nu^{-\frac{1}{2}}\left(\|{\rm e}^{3\epsilon \nu^{\frac{1}{3}}t}\partial_x^2(u_{\neq}\theta_{\neq})\|_{L^2L^2}
+\|{\rm e}^{3\epsilon \nu^{\frac{1}{3}}t}(\partial_x^2u^2_{\neq}+\partial_x^2u_{\neq}^3)\bar{\theta}\|_{L^2L^2}\right)\\
&+C\nu^{-\frac{1}{6}}\|{\rm e}^{3\epsilon \nu^{\frac{1}{3}}t}(u^{1,1}\partial_x+\bar{u}^2\partial_y+\bar{u}^3\partial_z)\partial_x^2\theta_{\neq}\|_{L^2L^2},
\end{align*}
where we have used $CE_1\leq C\epsilon_0<1/2$.

In addition, Lemma \ref{lemuthe-11} and Lemma \ref{lemuthe-10} give that
\begin{align*}
&\|{\rm e}^{3\epsilon \nu^{\frac{1}{3}}t} \partial_x^2(u_{\neq}\theta_{\neq})\|_{L^2L^2}\leq \nu^{-\frac{1}{2}}E_{3,0}E_{4,1},\\
&\|{\rm e}^{3\epsilon \nu^{\frac{1}{3}}t} (\partial_x^2u^2_{\neq}+\partial_x^2u_{\neq}^3)\bar{\theta}\|_{L^2L^2}
\leq C\nu^{-\frac{1}{6}}E_{4,0}E_5,\\
&\|{\rm e}^{3\epsilon \nu^{\frac{1}{3}}t} (u^{1,1}\partial_x+\bar{u}^2\partial_y+\bar{u}^3\partial_z)\partial_x^2\theta_{\neq}\|_{L^2L^2}
\leq C\nu^{\frac{1}{6}}E_1E_{4,1}+C\nu^{-\frac{1}{2}}E_2E_{4,1},
\end{align*}
which implies that
\begin{align*}
\|\partial_x^2\theta_{\neq}\|_{X_3}\leq C\|\theta(0)\|_{H^2}+CE_1E_{4,1}+C\nu^{-1}E_{3,0}E_{4,1}+C\nu^{-\frac{2}{3}}E_{4,0}E_5
+C\nu^{-\frac{2}{3}}E_2E_{4,1}.
\end{align*}

Let us focus on $\partial_z^2\theta_{\neq}$.
Recall that $\partial_z^2\theta_{\neq}$ solves
\begin{align*}
\begin{cases}
\partial_t \partial_z^2\theta_{\neq}-\nu\Delta \partial_z^2\theta_{\neq}+y\partial_x\partial_z^2\theta_{\neq} +\partial_z^2(u_{\neq}\cdot\nabla\theta_{\neq})_{\neq}+\partial_z^2((u_{\neq}^2\partial_y+u^3_{\neq}\partial_z)\bar{\theta})\\
\quad \quad+ \partial_z^2(\bar{u}^1\partial_x\theta_{\neq})+\partial_z^2\partial_y(\bar{u}^2\theta_{\neq})+\partial_z^3(\bar{u}^3\theta_{\neq})=0,\\
\partial_z^2\theta_{\neq}|_{t=0}=\partial_z^2\theta_{\neq}(0), \partial_z^2\theta_{\neq}|_{y=\pm1}=0,
\end{cases}
\end{align*}
then applying Proposition \ref{space-time-V21} with $V=y$, one has
\begin{align*}
&\|\partial_z^2\theta_{\neq}\|_{X_2}^2\leq C\|\theta(0)\|_{H^2}^2+C\nu^{-\frac{1}{3}}\|{\rm e}^{2\epsilon\nu^{\frac{1}{3}}t} \partial_z^2(\bar{u}^1\partial_x\theta_{\neq})\|_{L^2L^2}^2\\
&\quad\quad+C\nu^{-1}\left(\|{\rm e}^{2\epsilon\nu^{\frac{1}{3}}t}  \partial_z(u_{\neq}\cdot\nabla\theta_{\neq})_{\neq}\|_{L^2L^2}^2+
\|{\rm e}^{2\epsilon\nu^{\frac{1}{3}}t}  \partial_z((u_{\neq}^2\partial_y+u^3_{\neq}\partial_z)\bar{\theta})\|_{L^2L^2}^2\right)\\
&\quad\quad +C\nu^{-1}\left(\|{\rm e}^{2\epsilon\nu^{\frac{1}{3}}t}  \partial_z^2(\bar{u}^2\theta_{\neq})\|_{L^2L^2}^2+
\|{\rm e}^{2\epsilon\nu^{\frac{1}{3}}t}  \partial_z^2(\bar{u}^3\theta_{\neq})\|_{L^2L^2}^2\right).
\end{align*}
Recall that $\bar{u}^1=u^{1,0}+u^{1,1}$, by Lemma \ref{lemuthe-11} and Lemma \ref{lemuthe-10}, one has
\begin{align*}
&\|{\rm e}^{2\epsilon\nu^{\frac{1}{3}}t} \partial_z^2(\bar{u}^1\partial_x\theta_{\neq})\|_{L^2L^2}
\leq \|{\rm e}^{2\epsilon\nu^{\frac{1}{3}}t}\partial_z^2(u^{1,0}\partial_x\theta_{\neq})\|_{L^2L^2}
+\|{\rm e}^{2\epsilon\nu^{\frac{1}{3}}t} \partial_z^2(u^{1,1}\partial_x\theta_{\neq})\|_{L^2L^2}\\
&\ \ \quad\quad\quad\quad\quad\quad\quad\quad\quad\quad\leq C\nu^{\frac{1}{6}}E_1E_{4,1},\\
&\|{\rm e}^{2\epsilon\nu^{\frac{1}{3}}t} \partial_z(u_{\neq}\cdot\nabla\theta_{\neq})_{\neq}\|_{L^2L^2}
\leq C\nu^{-\frac{1}{2}}E_{3,0}E_{4,1},\\
&\|{\rm e}^{2\epsilon\nu^{\frac{1}{3}}t}\partial_z((u_{\neq}^2\partial_y+u^3_{\neq}\partial_z)\bar{\theta})\|_{L^2L^2}
\leq C\nu^{-\frac{1}{2}}E_{3,0}E_{4,0},\\
&\|{\rm e}^{2\epsilon\nu^{\frac{1}{3}}t}\partial_z^2(\bar{u}^2\theta_{\neq})\|_{L^2L^2}+
\|{\rm e}^{2\epsilon\nu^{\frac{1}{3}}t}\partial_z^2(\bar{u}^3\theta_{\neq})\|_{L^2L^2}
\leq C\nu^{-\frac{1}{2}}E_2E_{4,1},
\end{align*}
which implies that
\begin{align*}
\|\partial_z^2\theta_{\neq}\|_{X_2}\leq C\|\theta(0)\|_{H^2}+CE_1E_{4,1}+C\nu^{-1}E_{3,0}E_{4,1}+C\nu^{-1}E_{3,0}E_{4,0}+C\nu^{-1}E_2E_{4,1}.
\end{align*}
This completes the proof.
\end{proof}
\subsection{Estimate of $E_5$}
\begin{proposition}\label{pro14.2}
Let $\epsilon=\epsilon_1/8$ and $V_j$ satisfies \eqref{eqVcon} for $j\in [0,\nu^{\frac{1}{3}}T)\cap \mathbb{Z}$. Then it holds that
\begin{align*}
&\|{\rm e}^{4\epsilon \nu^{\frac{1}{3}}t} u_{[j]}\|_{L^2 Z_j^2}\leq C\|{\rm e}^{4\epsilon \nu^{\frac{1}{3}}t} g_{(j)}\|_{L^2Z_j^1}\leq C{\rm e}^{4\epsilon j}\|g_{(j)}\|_{L^2Z_j^1} \text{ for }j\in [0,\nu^{\frac{1}{3}}T)\cap \mathbb{Z},\\
&\|{\rm e}^{4\epsilon \nu^{\frac{1}{3}}t} u_{[0]}\|_{L^2 Z_j^2}\leq C(\|u(0)\|_{H^2}+{\rm e}^{4\epsilon j}\|g_{(0)}\|_{L^2Z_0^1})\leq C(\|u(0)\|_{H^2}+\|g_{(0)}\|_{L^2(I_0,Z_0^1)} ).
\end{align*}
\begin{proof}
See Proposition 14.2 of \cite{chenwz}.
\end{proof}
\end{proposition}
\begin{lemma}\label{lem14.2}
Suppose that $V_j$ satisfies \eqref{eqVcon} for $j\in [0,\nu^{\frac{1}{3}}T)\cap \mathbb{Z}$. For $j,k\in [0,\nu^{\frac{1}{3}}T)\cap \mathbb{Z}$, $v=(v^1,v^2,v^3)\in J^{(3)}_{2,0}(\Omega)\cap\mathcal{H}$, it holds that
\begin{align*}
&\|V_j-V_k\|_{H^2}+\|\kappa_j-\kappa_k\|_{H^1}\leq C\nu^{\frac{2}{3}}|j-k|E_1,
\|\kappa_j-\kappa_k\|_{H^2}\leq C\nu^{\frac{1}{3}}|j-k|^{\frac{1}{2}}E_1,\\
&\|u^{1,1}\|_{H^2}\leq C\nu^{\frac{2}{3}}E_1, \|v\|_{Z_j^2}-\|v\|_{Z_k^2}\leq C|j-k|^{\frac{1}{2} }E_1\|v\|_{Z_k^2}.
\end{align*}
For $j\in [0,\nu^{\frac{1}{3}}T)\cap \mathbb{Z}$, $t\in I_j$, there holds that
\begin{align*}
\|\kappa_j\nabla \bar{u}^3\|_{H^1}\leq CE_1E_2.
\end{align*}
For $j\in [0,\nu^{\frac{1}{3}}T)\cap \mathbb{Z}$, $f=(f_1,f_2,f_3)\in H_0^1(\Omega)$, we have
\begin{align*}
\nu^{\frac{1}{2}}\|\mathbb{P}f\|_{Z_j^1}\leq C(\|\nabla f_2\|_{L^2}+\|(\partial_x,\partial_z)f_3\|_{L^2}+\|\partial_xf_1\|_{L^2}+\nu^{\frac{1}{3}}j^{\frac{1}{2}}\|\nabla f_3\|_{L^2}).
\end{align*}
\end{lemma}
\begin{proof}
See Lemma 14.2 of \cite{chenwz}.
\end{proof}
\begin{lemma}\label{lem14.3}
For $\partial_xf=0, v=(v^1,v^2,v^3)\in J^{(3)}_{2,0}(\Omega)\cap\mathcal{H}$, it holds that
\begin{align*}
\|Q_1(f,v)\|_{Z_{[V]}^1}\leq C\nu^{-\frac{2}{3}}\|f\|_{H^2}\|v\|_{Z_{[V]}^2}.
\end{align*}
For $j\in [0,\nu^{\frac{1}{3}}T)\cap \mathbb{Z}$, $ v=(v^1,v^2,v^3)\in J^{(3)}_{2,0}(\Omega)\cap\mathcal{H}$, $t\in I_j$, we have
\begin{align*}
\|Q((0,\bar{u}^2,\bar{u}^3),v)\|_{Z_j^1}\leq C\nu^{-1}E_2\|v\|_{Z_j^2}.
\end{align*}
\end{lemma}
\begin{proof}
 See Lemma 14.3, Lemma14.4 of \cite{chenwz}.
\end{proof}
\begin{lemma}\label{lemPthe}
For $j\in [0,\nu^{\frac{1}{3}}T)\cap \mathbb{Z}$, $t\in I_j$, it holds that
\begin{align*}
\|\mathbb{P}(0,\theta_{\neq},0)\|_{Z_j^1}\leq C\|\nabla \theta_{\neq}\|_{L^2}.
\end{align*}
\end{lemma}
\begin{proof}
It is clear that
\begin{align*}
\begin{cases}
\mathbb{P}(0,\theta_{\neq},0)=(0,\theta_{\neq},0)+\nabla p,\\
\Delta p=-\partial_y \theta_{\neq},\\
\partial_yp|_{y=\pm 1}=0.
\end{cases}
\end{align*}
Let $f=(0,\theta_{\neq},0)=(f^1,f^2,f^3)$, since
\begin{align*}
\|\nabla(f^2+\kappa f^3)\|_{L^2}+\|\partial_xf^3\|_{L^2}=\|\nabla f^2\|_{L^2}=\|\nabla \theta_{\neq}\|_{L^2},
\end{align*}
which implies that $\|f\|_{Z_j^1}= \nu^{-\frac{1}{2}}\|\nabla \theta_{\neq}\|_{L^2}$.

Thanks to $\partial_yp|_{y=\pm1}=0$, the classical elliptic estimate gives that
\begin{align*}
\|\nabla p\|_{Z_j^1}=&\nu^{-\frac{1}{2}}\|\nabla(\partial_yp+\kappa \partial_zp)\|_{L^2}+\nu^{-\frac{1}{2}}\|\partial_x\partial_zp\|_{L^2}\leq C\nu^{-\frac{1}{2}}\|\nabla p\|_{H^1}\\
\leq& C\nu^{-\frac{1}{2}}\|\Delta p\|_{L^2}\leq C\nu^{-\frac{1}{2}}\|\partial_y\theta_{\neq}\|_{L^2},
\end{align*}
which yields that
\begin{align*}
\|\mathbb{P}(0,\theta_{\neq},0)\|_{Z_j^1}\leq \|(0,\theta_{\neq},0)\|_{Z_j^1}+\|\nabla p\|_{Z_j^1}\leq C\nu^{-\frac{1}{2}}\|\nabla \theta_{\neq}\|_{L^2}.
\end{align*}
The proof is completed.
\end{proof}
 Let $N=\max([0,\nu^{\frac{1}{3}}T)\cap \mathbb{Z})$, we define
$$E_6^2=\sum_{j=0}^N {\rm e}^{6\epsilon j}\|u_{\neq}\|_{L^2(I_j, Z_j^2)}^2.$$
\begin{proposition}\label{proE5}
It holds that
\begin{align*}
E_5^2\leq C E_6^2\leq C\|u(0)\|_{H^2}^2+ CE_1^2E_6^2+C\nu^{-2}E_2^2E_6^2+C\nu^{-2}E_{4,1}^2+C\nu^{-2}E_3^4.
\end{align*}
\end{proposition}
\begin{proof}
The proof can be obtained by following the similar arguments as in Proposition 14.3 of \cite{chenwz}.
The main differences here are the pressure and temperature, which have been conducted in $E_{4,0}, E_{4,1}$ and Lemma \ref{lemPthe}. We omit the details here.
\end{proof}
\section{Proof of the main result}
In this section, we will prove the main result, i.e., Theorem \ref{the1}. Let us assume that $\nu\in(0,\nu_0], \nu_0,\epsilon\in (0,1), \nu_0^{\frac{2}{3}}\leq 4\epsilon<\epsilon_1$. Then ${\rm e}^{\nu t}\leq {\rm e}^{4\epsilon \nu^{\frac{1}{3}}t}$ for $t>0$.

\subsection{The global well-posedness for \eqref{eqBou} with \eqref{bdd}}
The local well-posedness theory for the solution to \eqref{eqBou} with \eqref{bdd} is classical, and we will extend the solution to global time.

The proof is based on a bootstrap argument.
By the Propositions \ref{proE1}--\ref{proE2}, Proposition \ref{proE30}, Proposition \ref{proE31}, Proposition \ref{proE41}, and Proposition \ref{proE5}, one has
\begin{align*}
&E_{1,0}\leq C\nu^{-1}\left(\|\bar{u}(0)\|_{H^2}+E_2+E_2E_{1,0}\right),\\
&E_{1,1}\leq C\left(\|\bar{u}(0)\|_{H^2}+\nu^{-1}E_2E_{1,1}+\nu^{-\frac{4}{3}}E_3^2\right),\\
&E_2\leq C(1+\nu^{-1}E_2)(\|u(0)\|_{H^2}+\nu^{-1}E_3^2+\nu^{-1}\|\bar{\theta}(0)\|_2+ \nu^{-2}E_{3,0}E_{4,1}),\\
&E_{3,0}^2\leq C\|u(0)\|_{H^2}^2+C\left(\nu^{-2}E_3^4+\nu^{-2}E_2^2E_3^2+E_1^2E_3E_5+E_1^2E_3^\frac{3}{2}E_5^\frac{1}{2}+\nu^{-\frac{4}{3}}E_{4,1}^2\right),\\
&E_{3,1}^2\leq C\left(\|u(0)\|_{H^2}^2+\nu^{-2}E_3^4+\nu^{-2}E_2^2E_3^2+E_1^2E_3E_5+E_1^2E_3^\frac{3}{2}E_5^\frac{1}{2}+E_1^2E_3^\frac{7}{4}E_5^\frac{1}{4}+\nu^{-\frac{4}{3}}E_{4,1}^2\right),\\
&E_{4,0}\leq C(\nu^{-1}E_2E_{4,0}+\nu^{-1}E_{3,0}E_{4,1}+\|\bar{\theta}(0)\|_{H^2}),\\
&E_{4,1}^2\leq C(\|\theta(0)\|_{H^2}^2+\nu^{-2}E_{3,0}^2E_{4,1}^2+\nu^{-\frac{4}{3}}E_{4,0}^2E_5^2+E_1^2E_{4,1}^2+\nu^{-2}E_{3,0}^2E_{4,0}^2+\nu^{-2}E_2^2E_{4,1}^2),\\
&E_5^2\leq CE_6^2\leq C\|u(0)\|_{H^2}^2+ CE_1^2E_6^2+C\nu^{-2}E_2^2E_6^2+C\nu^{-2}E_{4,1}^2+C\nu^{-2}E_3^4.
\end{align*}

Let us assume that
 \begin{align}\label{bootstrap-ass}
 E_1\leq \varepsilon_1, E_2\leq \varepsilon_1 \nu, E_3\leq \varepsilon_1\nu, E_{4}\leq \varepsilon_1\nu^2,%E_5\leq \varepsilon_1 \nu, E_6\leq \varepsilon_1\nu.
 \end{align}
 where $\varepsilon_1>0$ is a suitable small constant will be determined latter.

 By taking $\varepsilon_1$ small enough, we have
 \begin{align*}
 E_3+\nu^{-1} E_4 +E_5 \leq C (\|u(0)\|_{H^2}+\nu^{-1} \| \theta(0)\|_{H^2}) \leq C c_0 \nu.
 \end{align*}
 Then we get that
  \begin{align*}
 E_2 \leq C( \|\bar{u}(0)\|_{H^2}+\nu^{-1} \|\bar{\theta} (0)\|_{H^2}+\varepsilon_1 E_3) \leq C ( \|u(0)\|_{H^2}+\nu^{-1} \|\theta(0)\|_{H^2}) \leq C c_0 \nu.
 \end{align*}
Moreover, one has
  \begin{align*}
 &E_{1,0} \leq C \nu^{-1} (\|\bar{u}(0)\|_{H^2}+E_2) \leq C\nu^{-1} ( \|u(0)\|_{H^2}+\nu^{-1} \|\theta(0)\|_{H^2}) \leq C c_0,\\
 &E_{1,1} \leq C\left(\|\bar{u}(0)\|_{H^2}+\varepsilon_1 \nu^{-\frac{1}{3}}E_3\right) \leq C \nu^{-\frac{1}{3}}  ( \|u(0)\|_{H^2}+\nu^{-1} \|\theta(0)\|_{H^2}) \leq C c_0 \nu^{\frac{2}{3}},
 \end{align*}
 which implies that
   \begin{align*}
 E_1=E_{1,0}+\nu^{-\frac{2}{3}} E_{1,1} \leq C c_0.
 \end{align*}
 We take $c_0>0$ small enough such that $Cc_0<\varepsilon_1/2$, then the continuity argument implies that the problem \eqref{eqBou} with \eqref{bdd} admits the global unique solution.

 \subsection{Large-time dynamics}
 In this subsection, we establish the large time dynamics for $u$ and $\theta$.

  First of all, it follows from the Lemma \ref{lemu0-2}, Lemma \ref{lemE21} and Lemma \ref{lemE2-4} that
 \begin{align*}
 \|\bar{u}^2(t)\|_{H^2}+\|\bar{u}^3(t)\|_{H^1}+\|(\bar{u}^2,\bar{u}^3)(t)\|_{L^\infty}
\leq C {\rm e}^{-\nu t} \left(\|u_{in}\|_{H^2}+\nu^{-1}\|\theta_{in}\|_{H^2}\right),
 \end{align*}
 which implies \eqref{dynamics-2}.

By Proposition \ref{proE40},
 one has
 \begin{align}
 &\|{\rm e}^{\nu t}\bar{\theta}\|_{L^\infty L^2}^2+\nu\|{\rm e}^{\nu t}\nabla\bar{\theta}\|_{L^2L^2}^2 \leq C (\|\overline{\theta} (0)\|_{L^2}^2+ \nu^{-2}E_{3,0}^2E_{4,1}^2),\nonumber\\
 &
\frac{d}{dt}\|\nabla\bar{\theta}\|_{L^2}^2+\nu\|\Delta \bar{\theta}\|_{L^2}^2\leq C\nu^{-1} \left(\|\bar{u}\cdot\nabla\bar{\theta}\|_{L^2}^2+\|\overline{u_{\neq}\cdot\nabla\theta_{\neq}}\|_{L^2}^2\right),\label{eqthe0-2}
\\
&
\frac{d}{dt}\|\partial_z\nabla \bar{\theta}\|_{L^2}^2+\nu\|\partial_z\Delta \bar{\theta}\|_{L^2}^2\leq C\nu^{-1} \left(\|\partial_z(\bar{u}\cdot\nabla\bar{\theta})\|_{L^2}^2+\|\partial_z(\overline{u_{\neq}\cdot\nabla\theta_{\neq}})\|_{L^2}^2\right).\label{eqthe0-3}
\end{align}

Therefore, from \eqref{eqthe0-2}, we have
\begin{align*}
\frac{d}{dt}&({\rm e}^{2\nu t}\|\nabla\bar{\theta}\|_{L^2}^2)+\nu {\rm e}^{2\nu t}\|\Delta \bar{\theta}\|_{L^2}^2\\
\leq & 2\nu {\rm e}^{2\nu t}\|\nabla\bar{\theta}\|_{L^2}^2+C\nu^{-1} {\rm e}^{2\nu t} \left(\|\bar{u}\cdot\nabla\bar{\theta}\|_{L^2}^2+\|\overline{u_{\neq}\cdot\nabla\theta_{\neq}}\|_{L^2}^2\right).
\end{align*}
It is clear that
\begin{align*}
{\rm e}^{2\nu t} \|\bar{u}\cdot\nabla\bar{\theta}\|_{L^2}^2\leq {\rm e}^{2\nu t}\left(\|\bar{u}^2\|_{L^\infty}+\|\bar{u}^3\|_{L^\infty}\right)^2\|\nabla \bar{\theta}\|_{L^2}^2\leq CE_2^2\|{\rm e}^{\nu t}\nabla \bar{\theta}\|_{L^2}^2
\end{align*}
and
\begin{align*}
\|{\rm e}^{2\nu t} \overline{u_{\neq}\cdot\nabla\theta_{\neq}}\|_{L^2L^2}^2\leq \|{\rm e}^{4\epsilon\nu^{1/3}t}\overline{u_{\neq}\cdot\nabla\theta_{\neq}}\|_{L^2L^2}^2\leq C\nu^{-1}E_{3,0}^2E_{4,1}^2,
\end{align*}
which along with the fact that $\|\nabla f\|_{L^2}^2\geq (\pi/2)^2 \|f\|_{L^2}^2$ gives that
\begin{align}
& {\rm e}^{2\nu t} \|\nabla\bar{\theta} (t) \|_{L^2}^2+\nu {\rm e}^{2\nu t} \|\Delta \bar{\theta}\|_{L^2L^2}^2\nonumber \\
\leq & \|\nabla \bar{\theta}_{in}\|_{L^2}^2+ C\nu^{-1}E_2^2\|{\rm e}^{\nu t}\nabla \bar{\theta}\|_{L^2L^2}^2 +C\nu^{-2}E_{3,0}^2E_{4,1}^2 \nonumber\\
\leq & \|\nabla \bar{\theta}_{in}\|_{L^2}^2+ C\nu^{-2}E_2^2\left(\|\bar{\theta} (0)\|_{L^2}^2+ C\nu^{-2}E_{3,0}^2E_{4,1}^2\right) +C\nu^{-2}E_{3,0}^2E_{4,1}^2.\label{eqthe0-nal}
\end{align}
From \eqref{eqthe0-3}, we have
\begin{align*}
&\frac{d}{dt}({\rm e}^{2\nu t}\|\partial_z\nabla\bar{\theta}\|_{L^2}^2)+\nu {\rm e}^{2\nu t}\|\partial_z\Delta \bar{\theta}\|_{L^2}^2\\
\leq& 2\nu {\rm e}^{2\nu t}\|\partial_z\nabla\bar{\theta}\|_{L^2}^2+C\nu^{-1} {\rm e}^{2\nu t} \left(\|\partial_z(\bar{u}\cdot\nabla\bar{\theta})\|_{L^2}^2+\|\partial_z(\overline{u_{\neq}\cdot\nabla\theta_{\neq}})\|_{L^2}^2\right).
\end{align*}
It is easy to deduce that
\begin{align*}
&{\rm e}^{2\nu t} \|\partial_z(\bar{u}\cdot\nabla\bar{\theta})\|_{L^2}^2\\
\leq& {\rm e}^{2\nu t}\left(\|\bar{u}^2\|_{L^\infty}+\|\bar{u}^3\|_{L^\infty}\right)^2\|\partial_z\nabla \bar{\theta}\|_{L^2}^2+{\rm e}^{2\nu t}\left(\|\partial_z\bar{u}^2\|_{L^\infty}+\|\partial_z\bar{u}^3\|_{L^\infty}\right)^2\|\nabla \bar{\theta}\|_{L^2}^2\\
\leq& CE_2^2\|{\rm e}^{\nu t}\partial_z\nabla \bar{\theta}\|_{L^2}^2+\left(\|\partial_z\bar{u}^2\|_{L^\infty}+\|\partial_z\bar{u}^3\|_{L^\infty}\right)^2\|{\rm e}^{\nu t}\nabla \bar{\theta}\|_{L^2}^2
\end{align*}
and
\begin{align*}
\|{\rm e}^{2\nu t} \partial_z\overline{u_{\neq}\cdot\nabla\theta_{\neq}}\|_{L^2L^2}^2\leq \|{\rm e}^{4\epsilon\nu^{1/3}t}\partial_z\overline{u_{\neq}\cdot\nabla\theta_{\neq}}\|_{L^2L^2}^2\leq C\nu^{-1}E_{3,0}^2E_{4,1}^2.
\end{align*}
By \eqref{eqthe0-nal}, one has
\begin{align}
&{\rm e}^{2\nu t}\|\partial_z\nabla\bar{\theta}\|_{L^2}^2(t)+\nu {\rm e}^{2\nu t}\|\partial_z\Delta \bar{\theta}\|_{L^2L^2}^2 \nonumber \\
\leq& \|\partial_z\nabla \bar{\theta}_{in}\|_{L^2}^2+ C\nu^{-1}E_2^2\|{\rm e}^{\nu t}\partial_z\nabla \bar{\theta}\|_{L^2L^2}^2 +C\nu^{-2}E_{3,0}^2E_{4,1}^2\nonumber\\
&+C\nu^{-1}\left(\|\partial_z\bar{u}^2\|_{L^2L^\infty}+\|\partial_z\bar{u}^3\|_{L^2L^\infty}\right)^2\|{\rm e}^{\nu t}\nabla \bar{\theta}\|_{L^\infty L^2}^2\nonumber\\
\leq& (1+\nu^{-2}E_2^2)\|\bar{\theta}_{in}\|_{H^2}^2 +C\nu^{-2}(1+\nu^{-2}E_2^2)E_{3,0}^2E_{4,1}^2.\label{eqthe-0-zn2}
\end{align}
Therefore, \eqref{dynamics-3} follows from \eqref{eqthe0-nal}, \eqref{eqthe-0-zn2} and estimates for $E_2$, $E_{3,0}$ and $E_{4,1}$.

By similar arguments as that for $\bar{u}^1$ in \cite{chenwz}, we have
\begin{align*}
\|{\rm e}^{\nu t}\Delta \bar{u}^1\|_{L^\infty L^2}^2\leq & \|u_{in}\|_{H^2}^2+C\nu^{-1}E_1^2(\|u_{in}\|_{H^2}^2+\nu^{-\frac{8}{3}}E_3^4)+C\nu^{-1}\|{\rm e}^{\nu t}\nabla \bar{u}^2\|_{L^2L^2}^2\\
\leq& C\nu^{-1} \left(\|u_{in}\|_{H^2}+\nu^{-1}\|\theta_{in}\|_{H^2}\right),
\end{align*}
which gives that
\begin{align*}
\|\bar{u}^1\|_{H^2}\leq C\|\Delta \bar{u}^1\|_{L^2}\leq C {\rm e}^{-\nu t}\nu^{-1} \left(\|u_{in}\|_{H^2}+\nu^{-1}\|\theta_{in}\|_{H^2}\right).
\end{align*}

On the other hands,
by the Lemma \ref{lemu1-0}, one obtains that
\begin{align*}
\|\bar{u}^1\|_{H^2}\leq CE_1(\nu^\frac{2}{3}+\nu t)\leq C\nu^{-1}(\nu^\frac{2}{3}+\nu t) \left(\|u_{in}\|_{H^2}+\nu^{-1}\|\theta_{in}\|_{H^2}\right).
\end{align*}
Thus, we have
\begin{align*}
\|\bar{u}^1\|_{L^\infty}\leq \|\bar{u}^1\|_{H^2}\leq
C\nu^{-1}\min\{\nu^\frac{2}{3}+\nu t, {\rm e}^{-\nu t}\} \left(\|u_{in}\|_{H^2}+\nu^{-1}\|\theta_{in}\|_{H^2}\right),
\end{align*}
which implies \eqref{dynamics-1}.

For the definition of $E_{4,1}$, we have
\begin{align*}
{\rm e}^{\epsilon \nu^{\frac{1}{3}} t}\|\partial_x^2\theta_{\neq}\|_{L^2}(t)+\|\partial_z^2\theta_{\neq}\|_{L^2}(t)\leq C {\rm e}^{-2\epsilon \nu^{\frac{1}{3}} t}\nu\left(\|u_{in}\|_{H^2}+\nu^{-1}\|\theta_{in}\|_{H^2}\right).
\end{align*}

For the definition of $E_{3,0}$, one has
\begin{align*}
&\|u_{\neq}(t)\|_{L^2}+\|(\partial_x, \partial_z)\nabla u^2_{\neq}(t)\|_{L^2}+\|(\partial_x^2+\partial_z^2)u_{\neq}^3(t)\|_{L^2}
+\|(\partial_x,\partial_z)\partial_x u_{\neq}(t)\|_{L^2}\\
\leq& C(\|(\partial_x, \partial_z)\nabla u^2_{\neq}(t)\|_{L^2}+\|(\partial_x^2+\partial_z^2)u_{\neq}^3(t)\|_{L^2}) \leq E_{3,0}\leq C\left(\|u_{in}\|_{H^2}+\nu^{-1}\|\theta_{in}\|_{H^2}\right).
\end{align*}
By Proposition \ref{proE30}, one has
\begin{align*}
&\nu^{\frac{1}{2}} {\rm e}^{4\epsilon \nu^{\frac{1}{3}}t}\|\Delta u_{\neq}^2(t)\|_{L^2}^2\leq \|u_{\neq}^2\|_{M_{2\epsilon}}^2\\
\leq &
C\|u(0)\|_{H^2}^2+C\left(\nu^{-2}E_3^4+\nu^{-2}E_2^2E_3^2+E_1^2E_3E_5+E_1^2E_3^\frac{3}{2}E_5^\frac{1}{2}+\nu^{-\frac{4}{3}}E_{4,1}^2\right)\\
\leq& C\left(\|u_{in}\|_{H^2}+\nu^{-1}\|\theta_{in}\|_{H^2}\right).
\end{align*}
By the definition of $E_{3,1}$, we have
\begin{align*}
{\rm e}^{2\epsilon \nu^{\frac{1}{3}}t}\|\nabla \omega_{\neq}^2(t)\|_{L^2}\leq C \nu^{-\frac{1}{3}}E_{3,1}\leq C\nu^{-\frac{1}{3}}\left(\|u_{in}\|_{H^2}+\nu^{-1}\|\theta_{in}\|_{H^2}\right).
\end{align*}
Then, using the fact that $\|\nabla(\partial_x,\partial_z)(u_{\neq}^1,u_{\neq}^3)\|_{L^2}^2=\|\nabla \omega_{\neq}^2\|_{L^2}^2+\|\nabla \partial_y u_{\neq}^2\|_{L^2}^2$, we have
\begin{align*}
&\|(u_{\neq}^1,u_{\neq}^3)(t)\|_{H^1}\leq C\|\nabla\partial_x(u_{\neq}^1,u_{\neq}^3)\|_{L^2}\leq C\nu^{-\frac{1}{3}} {\rm e}^{-2\epsilon \nu^{\frac{1}{3}}t}\left(\|u_{in}\|_{H^2}+\nu^{-1}\|\theta_{in}\|_{H^2}\right),\\
&\nu^{\frac{1}{4}}\|u_{\neq}^2(t)\|_{H^2}\leq C\left(\|u_{in}\|_{H^2}+\nu^{-1}\|\theta_{in}\|_{H^2}\right),\\
&\nu^{\frac{1}{2}}\|t(u_{\neq}^1,u_{\neq}^3)\|_{L^2L^2}\leq \nu^{\frac{1}{2}}\|t{\rm e}^{-2\epsilon \nu^{\frac{1}{3}}t}\|_{L^2(0,\infty)}\|u_{\neq}\|_{L^\infty L^2}\leq C\left(\|u_{in}\|_{H^2}+\nu^{-1}\|\theta_{in}\|_{H^2}\right),\\
&\|\nabla u_{\neq}^2\|_{L^\infty L^2}+\|\nabla u_{\neq}^2\|_{L^2L^2}\leq CE_{3,0}\leq C\left(\|u_{in}\|_{H^2}+\nu^{-1}\|\theta_{in}\|_{H^2}\right).
\end{align*}

 By the Lemma 16.2 of \cite{chenwz}, we have
\begin{align*}
\|u_{\neq}^2(t)\|_{L^\infty}\leq & C\|(\partial_x,\partial_z)\partial_xu_{\neq}^2(t)\|_{L^2}^\frac{1}{2}\|(\partial_x,\partial_z)\nabla u_{\neq}^2(t)\|_{L^2}^\frac{1}{2}\\
\leq & C{\rm e}^{-2\epsilon \nu^{\frac{1}{3}}t}\left(\|u_{in}\|_{H^2}+\nu^{-1}\|\theta_{in}\|_{H^2}\right),\\
\|u_{\neq}^j(t)\|_{L^\infty}\leq & C\|(\partial_x,\partial_z)\partial_xu_{\neq}^j(t)\|_{L^2}^\frac{1}{2}\|(\partial_x,\partial_z)\nabla u_{\neq}^j(t)\|_{L^2}^\frac{1}{2}\\ \leq & C\nu^{-\frac{1}{6}} {\rm e}^{-2\epsilon \nu^{\frac{1}{3}}t}\left(\|u_{in}\|_{H^2}+\nu^{-1}\|\theta_{in}\|_{H^2}\right),
\end{align*}
here $j\in \{1,3\}$.

It remains to show the estimate \eqref{dynamics-5}.
Recall that $\theta_{\neq}$ satisfies
\begin{align*}
\partial_t\theta_{\neq}-\nu \Delta \theta_{\neq}+y\partial_x\theta_{\neq}+\bar{u}^1\partial_x\theta_{\neq}
+(\bar{u}^2\partial_y+\bar{u}^3\partial_z)\theta_{\neq}+u_{\neq}\cdot\nabla\bar{\theta}
+(u_{\neq}\cdot\nabla \theta_{\neq})_{\neq}=0,
\end{align*}
then by Proposition \ref{space-na-y}, one has
\begin{align*}
&\|{\rm e}^{2\epsilon \nu^\frac{1}{3}t}\partial_y \theta_{\neq}\|_{L^\infty L^2}^2+\nu\|{\rm e}^{2\epsilon \nu^\frac{1}{3}t}\nabla \partial_y \theta_{\neq}\|_{L^2L^2}^2\\
\leq &  C\|\theta_{in}\|_{H^2}^2+C\nu^{-1}\bigg(\left\|{\rm e}^{2\epsilon \nu^\frac{1}{3}t}\left[\bar{u}^1\partial_x\theta_{\neq}
+(\bar{u}^2\partial_y+\bar{u}^3\partial_z)\theta_{\neq}\right]\right\|_{L^2L^2}^2\\
&+\left\|{\rm e}^{2\epsilon \nu^\frac{1}{3}t}\left[u_{\neq}\cdot\nabla\bar{\theta}
+(u_{\neq}\cdot\nabla \theta_{\neq})_{\neq}\right]\right\|_{L^2L^2}^2\bigg)
\end{align*}
and
\begin{align*}
&\|{\rm e}^{2\epsilon \nu^\frac{1}{3}t}(\partial_x,\partial_z)\partial_y \theta_{\neq}\|_{L^\infty L^2}^2+\nu\|{\rm e}^{2\epsilon \nu^\frac{1}{3}t}\nabla (\partial_x,\partial_z)\partial_y \theta_{\neq}\|_{L^2L^2}^2\\
\leq &  C\|\theta_{in}\|_{H^2}^2+C\nu^{-1} \bigg(\left\|{\rm e}^{2\epsilon \nu^\frac{1}{3}t}(\partial_x,\partial_z)\left[\bar{u}^1\partial_x\theta_{\neq}
+(\bar{u}^2\partial_y+\bar{u}^3\partial_z)\theta_{\neq}\right]\right\|_{L^2L^2}^2\\
&+\left\|{\rm e}^{2\epsilon \nu^\frac{1}{3}t}(\partial_x,\partial_z)\left[u_{\neq}\cdot\nabla\bar{\theta}
+(u_{\neq}\cdot\nabla \theta_{\neq})_{\neq}\right]\right\|_{L^2L^2}^2\bigg).
\end{align*}

For $j\in \{2,3\}$, by Lemma \ref{lemu1-0} and Lemma \ref{lemu0-2}, one has
 \begin{align*}
& \left\|{\rm e}^{2\epsilon \nu^\frac{1}{3}t}\left[\bar{u}^1\partial_x\theta_{\neq}
+(\bar{u}^2\partial_y+\bar{u}^3\partial_z)\theta_{\neq}\right]\right\|_{L^2}
\\
\leq &\|\bar{u}^1\|_{L^\infty }\|{\rm e}^{2\epsilon \nu^\frac{1}{3}t}\partial_x\theta_{\neq}\|_{L^2}+\|\bar{u}^j\|_{ L^\infty}\|{\rm e}^{2\epsilon \nu^\frac{1}{3}t}\nabla \theta_{\neq}\|_{L^2}\\
\leq &CE_1 \nu^{\frac{2}{3}}(1+\nu^\frac{1}{3}t){\rm e}^{2\epsilon \nu^\frac{1}{3}t}\|\partial_x\theta_{\neq}\|_{L^2}+CE_2\|{\rm e}^{2 \epsilon \nu^\frac{1}{3}t}\nabla \theta_{\neq}\|_{L^2}\\
\leq &CE_1 \nu^{\frac{2}{3}}{\rm e}^{\frac{5}{2}\epsilon \nu^\frac{1}{3}t}\|\partial_x\theta_{\neq}\|_{L^2}+CE_2\|{\rm e}^{2 \epsilon \nu^\frac{1}{3}t}\nabla \theta_{\neq}\|_{L^2},
 \end{align*}
 \begin{align*}
 &\left\|{\rm e}^{2\epsilon \nu^\frac{1}{3}t}\partial_x\left[\bar{u}^1\partial_x\theta_{\neq}
+(\bar{u}^2\partial_y+\bar{u}^3\partial_z)\theta_{\neq}\right]\right\|_{L^2}
\\
\leq &\|\bar{u}^1\|_{L^\infty }\|{\rm e}^{2\epsilon \nu^\frac{1}{3}t}\partial_x^2\theta_{\neq}\|_{L^2}+\|\bar{u}^j\|_{ L^\infty}\|{\rm e}^{2\epsilon \nu^\frac{1}{3}t}\partial_x\nabla \theta_{\neq}\|_{L^2}\\
\leq &CE_1 \nu^{\frac{2}{3}}{\rm e}^{\frac{5}{2}\epsilon \nu^\frac{1}{3}t}\|\partial_x^2\theta_{\neq}\|_{L^2}+CE_2\|{\rm e}^{2 \epsilon \nu^\frac{1}{3}t}\partial_x\nabla \theta_{\neq}\|_{L^2},
 \end{align*}
 and
 \begin{align*}
& \left\|{\rm e}^{2\epsilon \nu^\frac{1}{3}t}\partial_z\left[\bar{u}^1\partial_x\theta_{\neq}
+(\bar{u}^2\partial_y+\bar{u}^3\partial_z)\theta_{\neq}\right]\right\|_{L^2}
\\
\leq &\|\bar{u}^1\|_{L^\infty }\|{\rm e}^{2\epsilon \nu^\frac{1}{3}t}\partial_x\partial_z\theta_{\neq}\|_{L^2}+\|\bar{u}^j\|_{ L^\infty}\|{\rm e}^{2\epsilon \nu^\frac{1}{3}t}\partial_z\nabla \theta_{\neq}\|_{L^2}\\
&+\|\partial_z\bar{u}^1\|_{L^\infty_y L^2_z}\|{\rm e}^{2\epsilon \nu^\frac{1}{3}t}\partial_x\theta_{\neq}\|_{L^2_{x,y}L^\infty_z}
+
\|\partial_z\bar{u}^j\|_{L^\infty_y L^2_z}\|{\rm e}^{2\epsilon \nu^\frac{1}{3}t}\nabla\theta_{\neq}\|_{L^2_{x,y}L^\infty_z}\\
\leq &CE_1 \nu^{\frac{2}{3}}(1+\nu^\frac{1}{3}t){\rm e}^{2\epsilon \nu^\frac{1}{3}t}\|\partial_x(\partial_x,\partial_z)\theta_{\neq}\|_{L^2}+CE_2\|{\rm e}^{2 \epsilon \nu^\frac{1}{3}t}(\partial_z,\partial_x)\nabla \theta_{\neq}\|_{L^2}\\
\leq &CE_1 \nu^{\frac{2}{3}}{\rm e}^{\frac{9}{4}\epsilon \nu^\frac{1}{3}t}(\|\partial_x^2\theta_{\neq}\|_{L^2}^\frac{1}{2}\|\partial_z^2\theta_{\neq}\|_{L^2}^\frac{1}{2}+\|\partial_x^2\theta_{\neq}\|_{L^2})+CE_2\|{\rm e}^{2 \epsilon \nu^\frac{1}{3}t}(\partial_x,\partial_z)\nabla \theta_{\neq}\|_{L^2},
 \end{align*}
 which imply that
 \begin{align*}
 &\left\|{\rm e}^{2\epsilon \nu^\frac{1}{3}t}(\partial_x,1)\left[\bar{u}^1\partial_x\theta_{\neq}
+(\bar{u}^2\partial_y+\bar{u}^3\partial_z)\theta_{\neq}\right]\right\|_{L^2L^2}\\
\leq& CE_1\nu^{\frac{2}{3}}\|{\rm e}^{-\frac{1}{2}\epsilon \nu^\frac{1}{3}t}\|_{L^2}\|{\rm e}^{3\epsilon \nu^\frac{1}{3}t}\partial_x(\partial_x,1)\theta_{\neq}\|_{L^\infty L^2}+CE_2\|{\rm e}^{2\epsilon \nu^\frac{1}{3}t}\nabla(\partial_x,1)\theta_{\neq}\|_{L^2L^2}\\
\leq& CE_1\nu^{\frac{1}{2}}E_{4,1}+C\nu^{-\frac{1}{2}}E_2E_{4,1},
 \end{align*}
 and
 \begin{align*}
 &\left\|{\rm e}^{2\epsilon \nu^\frac{1}{3}t}\partial_z\left[\bar{u}^1\partial_x\theta_{\neq}
+(\bar{u}^2\partial_y+\bar{u}^3\partial_z)\theta_{\neq}\right]\right\|_{L^2L^2}\\
\leq& CE_1\nu^{\frac{2}{3}}\|{\rm e}^{-\frac{1}{4}\epsilon \nu^\frac{1}{3}t}\|_{L^2}\|{\rm e}^{3\epsilon \nu^\frac{1}{3}t}\partial_x^2\theta_{\neq}\|_{L^\infty L^2}^\frac{1}{2}\|{\rm e}^{2\epsilon \nu^\frac{1}{3}t}\partial_z^2\theta_{\neq}\|_{L^\infty L^2}^\frac{1}{2}\\
&+CE_2\|{\rm e}^{2\epsilon \nu^\frac{1}{3}t}\nabla(\partial_x,\partial_z)\theta_{\neq}\|_{L^2L^2}\\
\leq& CE_1\nu^{\frac{1}{2}}E_{4,1}+C\nu^{-\frac{1}{2}}E_2E_{4,1}.
 \end{align*}

 Thanks to Lemma \ref{lemuthe-11} and Lemma \ref{lemuthe-10}, it holds that
\begin{align*}
\|{\rm e}^{2\epsilon \nu^{1/3}t}(\partial_z,1)(u_{\neq}\cdot \nabla \theta_{\neq})\|_{L^2L^2}\leq C\nu^{-\frac{1}{2}}E_{3,0}E_{4,1},
\end{align*}
and
\begin{align*}
\|{\rm e}^{2\epsilon \nu^\frac{1}{3}t}(\partial_z,1)(u_{\neq}\cdot\nabla \bar{\theta})\|_{L^2L^2}\leq C\nu^{-\frac{1}{2}}E_{3,0}E_{4,0}.
\end{align*}

Therefore, we get
\begin{align*}
&\|{\rm e}^{2\epsilon \nu^\frac{1}{3}t}\partial_y \theta_{\neq}\|_{L^\infty L^2}^2+\nu\|{\rm e}^{2\epsilon \nu^\frac{1}{3}t}\nabla \partial_y \theta_{\neq}\|_{L^2L^2}^2\\
\leq& C\|\theta_{in}\|_{H^1}^2+CE_1^2E_{4,1}^2+C\nu^{-2}E_2^2E_{4,1}^2+C\nu^{-2}E_{3,0}^2E_{4,1}^2
+C\nu^{-2}E_{3,0}^2E_{4,0}^2\\
\leq& C\nu^2 \left(\|u_{in}\|_{H^2}+\nu^{-1}\|\theta_{in}\|_{H^2}\right)^2,
\end{align*}
which implies that
\begin{align*}
\|\partial_y\theta_{\neq}(t)\|_{L^2}\leq C\nu{\rm e}^{-2\epsilon \nu^\frac{1}{3}t}\left(\|u_{in}\|_{H^2}+\nu^{-1}\|\theta_{in}\|_{H^2}\right).
\end{align*}

For $j\in \{1,3\}$, one has
 \begin{align*}
 \|\partial_x(u_{\neq}^j\partial_j \theta_{\neq})\|_{L^2}\leq
 & \|\partial_xu_{\neq}^j\partial_j\theta_{\neq}\|_{L^2}+\|u_{\neq}^j\partial_j\partial_x\theta_{\neq}\|_{L^2}\\
\leq& \|\partial_xu_{\neq}\|_{L^\infty_zL^2_{y,x}}\|\partial_j\theta_{\neq}\|_{L^2_zL^\infty_{y,x}}
+\|u_{\neq}^j\|_{L^\infty_{x,z}L^2_y}\|\partial_x\partial_j\theta_{\neq}\|_{L^2_{x,z}L^\infty_y}\\
\leq& C\left(\|\partial_z\partial_xu_{\neq}\|_{L^2}+\|(\partial_x,1)u_{\neq}\|_{L^2}\right)\|(\partial_x^2+\partial_z^2)\nabla \theta_{\neq}\|_{L^2}\\
\leq& C\left(\|(\partial_x,\partial_z)\nabla u_{\neq}^2\|_{L^2}+\|(\partial_x^2+\partial_z^2)u_{\neq}^3\|_{L^2} \right)\|(\partial_x^2+\partial_z^2)\nabla \theta_{\neq}\|_{L^2},
 \end{align*}
 \begin{align*}
 \|\partial_x(u_{\neq}^2\partial_y\theta_{\neq})\|_{L^2}\leq & \|\partial_xu_{\neq}^2\partial_y\theta_{\neq}\|_{L^2}+
 \|u_{\neq}^2\partial_y\partial_x\theta_{\neq}\|_{L^2}\\
 \leq& \|\partial_xu_{\neq}^2\|_{L^\infty_y L^2_{x,z}}\|\partial_y\theta_{\neq}\|_{L^2_yL^\infty_{x,z}}
 +\|u_{\neq}^2\|_{L^\infty_{x,y} L^2_{z}}\|\partial_y\theta_{\neq}\|_{L^2_{x,y}L^\infty_{z}}\\
 \leq& C\|(\partial_x,\partial_z)\nabla u_{\neq}^2\|_{L^2}\|(\partial_x^2+\partial_z^2)\nabla \theta_{\neq}\|_{L^2},
 \end{align*}
 and
 \begin{align*}
 &\|\partial_x u_{\neq}\cdot\nabla \bar{\theta}\|_{L^2}\leq \|\partial_xu_{\neq}^2\partial_y\bar{\theta}\|_{L^2}+\|\partial_xu_{\neq}^3\partial_z\bar{\theta}\|_{L^2}\\
 \leq& \|\partial_xu_{\neq}^2\|_{L^\infty_{y}L^2_{x,z}}\|\partial_y\bar{\theta}\|_{L^2_yL^\infty_z}
 +\|\partial_xu_{\neq}^3\|_{L^2_{x,y} L^\infty_z }\|\partial_z\bar{\theta}\|_{L^\infty_y L^2_z}\\
 \leq& \left(\|(\partial_x,\partial_z)\nabla u_{\neq}^2\|_{L^2}+\|(\partial_x^2+\partial_z^2)u_{\neq}^3\|_{L^2} \right)\left(\|\nabla\bar{\theta}\|_{L^2}+\|\partial_y\partial_z\bar{\theta}\|_{L^2}\right).
 \end{align*}

Therefore, we have
 \begin{align*}
 \|{\rm e}^{2\epsilon \nu^{\frac{1}{3}}t}\partial_x(u_{\neq}\cdot \nabla \theta_{\neq})\|_{L^2L^2}\leq C\nu^{-\frac{1}{2}}E_{3,0}E_{4,1},
 \end{align*}
 and
 \begin{align*}
 \|{\rm e}^{2\epsilon \nu^{\frac{1}{3}}t}\partial_x(u_{\neq}\cdot \nabla \bar{\theta})\|_{L^2L^2}\leq C\nu^{-\frac{1}{2}}E_{3,0}E_{4,0}.
 \end{align*}
 Therefore, we have
 \begin{align*}
 &\|{\rm e}^{2\epsilon \nu^\frac{1}{3}t}(\partial_x,\partial_z)\partial_y \theta_{\neq}\|_{L^\infty L^2}^2+\nu\|{\rm e}^{2\epsilon \nu^\frac{1}{3}t}\nabla (\partial_x,\partial_z)\partial_y \theta_{\neq}\|_{L^2L^2}^2\\
 \leq& \|\theta_{in}\|_{H^2}^2+CE_1^2E_{4,1}^2+C\nu^{-2}E_2^2E_{4,1}^2+C\nu^{-2}E_{3,0}^2E_{4,1}^2
+C\nu^{-2}E_{3,0}^2E_{4,0}^2\\
\leq& C\nu^2 \left(\|u_{in}\|_{H^2}+\nu^{-1}\|\theta_{in}\|_{H^2}\right)^2,
\end{align*}
which implies that
\begin{align*}
\|\partial_y(\partial_x,\partial_z)\theta_{\neq}(t)\|_{L^2}\leq C\nu{\rm e}^{-2\epsilon \nu^\frac{1}{3}t}\left(\|u_{in}\|_{H^2}+\nu^{-1}\|\theta_{in}\|_{H^2}\right).
\end{align*}
 It follows from the Lemma \ref{lemP0f-1}, one has
 \begin{align*}
 \|\theta_{\neq}(t)\|_{L^\infty}\leq & C\|(\partial_x,\partial_z)\partial_x\theta_{\neq}\|_{L^2}^\frac{1}{2}
 \|(\partial_x,\partial_z)\nabla\theta_{\neq}\|_{L^2}^\frac{1}{2}\\
 \leq & C\nu {\rm e}^{-2\epsilon \nu^\frac{1}{3}t}\left(\|u_{in}\|_{H^2}+\nu^{-1}\|\theta_{in}\|_{H^2}\right),
 \end{align*}
 which gives \eqref{dynamics-5}.

\appendix

\section{Some useful lemmas}

An operator $A$ in a Hilbert space $H$ is accretive if ${\bf Re}\langle Af, f\rangle\geq 0$ for all $f\in D(A)$, or equivalently $\|(\lambda+A)f\|\geq \lambda \|f\|$ for all $f\in D(A)$ and all $\lambda>0$. The operator $A$ is called m-accretive if in addition any $\lambda<0$ belongs to the resolvent set of $A$. We define
\begin{align*}
\Psi(A)=\inf\{\|(A-{\rm i} \lambda)f\|:f\in D(A), \lambda\in \mathbb{R}, \|f\|=1\}.
\end{align*}

\begin{lemma}\label{Gearhart}
Let $A$ be an m-accretive operator in a Hilbert space $H$. Then we have
\begin{align*}
\|{\rm e}^{-tA}\|\leq {\rm e}^{-t\Psi(A)+\pi/2},
\end{align*}
for any $t\geq 0$.
\end{lemma}
\begin{proof}
See \cite{Wei} for details.
\end{proof}

\begin{lemma}
\label{pxy}
If $f=f(y,z)$, then we have
\begin{align*}
\|\partial_y^2f\|_{L^2}+\|\partial_z^2f\|_{L^2} \leq C\left(\|\Delta f\|_{L^2}+\|\partial_z\partial_y f\|_{L^2}\right).
\end{align*}
\end{lemma}
\begin{proof}
See Lemma 16.7 of \cite{chenwz}.
\end{proof}

\begin{lemma}
\label{lemP0f-1}
If $\mathbb{P}_0 f=0$, it holds that
\begin{align*}
\|f\|_{L^\infty}^2\leq C\|(\partial_x,\partial_z)\partial_zf\|_{L^2}\|(\partial_x,\partial_z)\nabla f\|_{L^2}.
\end{align*}
\end{lemma}
\begin{proof}
See Lemma 16.2 of \cite{chenwz}.
\end{proof}

\section*{Acknowledgments}
J. Li is supported in part by the National Natural Science Foundation of China (No. 12371204), the Key Project of National Natural Science Foundation of China (No. 12131010), and the Guangdong Basic and Applied Basic Research Foundation (No. 2026A1515010778).
 Z. Lin is supported by the National Natural Science Foundation of China (No. 12401292), and the Guangdong Basic and Applied Basic Research Foundation (No. 2026A1515010778).

\end{document}